\documentclass[12pt]{amsart}

\newtheorem{theorem}{Theorem}[section] 
\newtheorem{corollary}[theorem]{Corollary} 
 
\newtheorem{lemma}[theorem]{Lemma} 
\newtheorem{proposition}[theorem]{Proposition}

\theoremstyle{definition}
\newtheorem{definition}{Definition}

\theoremstyle{remark}

  {\end{list}}

\usepackage{amssymb} 
\usepackage{verbatim}

\newcommand{\ord}{\operatorname{ord}}

\renewcommand{\deg}{{\operatorname{deg}}}

\newcommand{\Spec}{\operatorname{Spec}}
\newcommand{\CPA}{\operatorname{CPA}}
\newcommand{\GL}{\operatorname{GL}}

\newcommand{\Res}{\operatorname{Res}}

\newcommand{\id}{\operatorname{id}}

\newcommand{\Char}{\operatorname{char}}

\newcommand{\Ann}{\operatorname{Ann}}

\newcommand{\MinResLoc}{\operatorname{MinResLoc}}

\newcommand{\GCD}{\operatorname{GCD}}
\newcommand{\FR}{\operatorname{FR}}
\newcommand{\Can}{\operatorname{Can}}
\newcommand{\SL}{\operatorname{SL}}

\def\ord{{\mathop{\rm ord}}}
\def\ordRes{{\mathop{\rm ordRes}}}

\def\deg{{\mathop{\rm deg}}}

\def\GL{{\mathop{\rm GL}}}

\def\Spec{{\mathop{\rm Spec}}}
\def\CPA{{\mathop{\rm CPA}}}

\def\Berk{{\mathop{\rm Berk}}}
\def\Res{{\mathop{\rm Res}}}

\def\Spec{{\mathop{\rm Spec}}}

\def\Fix{{\mathop{\rm Fix}}}
\def\Repel{{\mathop{\rm Repel}}}
\def\id{{\mathop{\rm id}}}
\def\Visible{{\mathop{\rm Visible}}}
\def\Shearing{{\mathop{\rm Shearing}}}

\def\Rat{{\mathop{\rm Rat}}}
\def\cCr{{\mathop{{\mathcal{C} {\rm r}}}}}

\def\vv{{\vec{v}}}
\def\vw{{\vec{w}}}

\def\v1{{\vec{1}}}
\def\vbb1{\vec{{\mathbf 1}}}

\def\valpha{{\vec{\alpha}}}

\def\BF1{{\mathbf 1}}

\def\HH{{\mathbb H}}

\def\NN{{\mathbb N}}

\def\PP{{\mathbb P}}

\def\RR{{\mathbb R}}

\def\ZZ{{\mathbb Z}}

\def\cB{{\mathcal B}}

\def\cE{{\mathcal E}}

\def\cM{{\mathcal M}}

\def\cO{{\mathcal O}}

\def\ha{{\widehat{a}}}
\def\hb{{\widehat{b}}}

\def\hf{{\widehat{f}}}
\def\hg{{\widehat{g}}}

\def\hA{{\widehat{A}}}

\def\hF{{\widehat{F}}}
\def\hG{{\widehat{G}}}

\def\hGamma{{\widehat{\Gamma}}}

\def\hFR{{\widehat{FR}}}

\def\hbar{{\overline{h}}}

\def\BPP{{{\bf P}}^1}
\def\BHH{{{\bf H}}^1}

\def\fM{{\mathfrak M}}

\def\ta{{\widetilde{a}}}
\def\tb{{\widetilde{b}}}
\def\tc{{\widetilde{c}}}
\def\td{{\widetilde{d}}}

\def\tf{{\widetilde{f}}}
\def\tg{{\widetilde{g}}}
\def\th{{\widetilde{h}}}

\def\tk{{\widetilde{k}}}

\def\tu{{\widetilde{u}}}
\def\tv{{\widetilde{v}}}

\def\tvarphi{{\widetilde{\varphi}}}
\def\ttau{{\widetilde{\tau}}}
\def\teta{{\widetilde{\eta}}}

\def\tA{{\widetilde{A}}}

\def\tC{{\widetilde{C}}}

\def\tF{{\widetilde{F}}}
\def\tG{{\widetilde{G}}}
\def\tH{{\widetilde{H}}}

\def\tL{{\widetilde{L}}}
\def\tM{{\widetilde{M}}}

\def\tZ{{\widetilde{Z}}}

\def\talpha{{\widetilde{\alpha}}}

\def\tphi{{\widetilde{\varphi}}}

\def\tlambda{{\widetilde{\lambda}}}

\def\tPhi{{\widetilde{\Phi}}}

\def\tpsi{{\widetilde{\psi}}}

\def\tid{{\widetilde{\rm 1}}}
\def\trot{{\widetilde{\rm rot}}}

\def\Berk{{\rm Berk}}

\def\<{{\langle }}
\def\>{{\rangle }}

\def\<<{{\langle \! \langle}}
\def\>>{{\rangle \! \rangle}} 

\def\({(\!(}
\def\){)\!)}

\def\[{[\![}
\def\]{]\!]}

\DeclareMathSymbol{\varnothing} {\mathord}{AMSb}{"3F} 
 
\theoremstyle{definition} 
\theoremstyle{remark} 

\begin{document}
\title{The Geometry of the Minimal Resultant Locus}

\author{Robert Rumely}
\address{Robert Rumely\\ 
Department of Mathematics\\
University of Georgia\\
Athens, Georgia 30602\\
USA}
\email{rr@math.uga.edu}

\date{February 23, 2014}

\begin{abstract}
Let $K$ be a complete, algebraically closed nonarchimedean valued field, 
and let $\varphi(z) \in K(z)$ have degree $d \ge 2$.  
We show there is a natural way to assign non-negative integer weights $w_\varphi(P)$ to points of the
Berkovich projective line over $K$ in such a way that $\sum_P w_\varphi(P) = d-1$.  When $\varphi$ has bad reduction,
the set of points with nonzero weight forms a distributed analogue of the unique point which occurs when $\varphi$ has
potential good reduction.  
Using this, we characterize the Minimal Resultant Locus of $\varphi$ in dynamical and moduli-theoretic terms:
dynamically, it is the barycenter of the weight-measure associated to $\varphi$;  moduli-theoretically, it is the  
closure of the set of points where $\varphi$ has semi-stable reduction, in the sense of Geometric Invariant Theory.
\end{abstract}

\maketitle

Let $K$ be a complete, algebraically closed nonarchimedean valued field with absolute value $| \cdot |$
and associated valuation $\ord(\cdot)$. 
Write $\cO$ for the ring of integers of $K$, $\fM$ for its maximal ideal, and $\tk$ for its residue field.

Let $\varphi(z)  \in K(z)$  be a function with degree $d \ge 2$.  Suppose $(F,G)$ is a normalized
representation for $\varphi$:  a pair of homogeneous functions $F(X,Y), G(X,Y) \in \cO[X,Y]$ of degree $d$, 
such that $\varphi(z) = F(z,1)/G(z,1)$ and some coefficient of $F$ or $G$ belongs to $\cO^{\times}$.  
Such a pair $(F,G)$ is unique up to scaling by a unit.  Let $\Res(F,G)$ be the homogeneous resultant of $F$ and $G$;
then $\ordRes(\varphi) := \ord(\Res(F,G))$ is well-defined and non-negative.    
 
Let $\BPP_K$ be the Berkovich projective line over $K$:  a compact, uniquely path-connected Hausdorff 
space which contains $\PP^1(K)$ as a dense subset.  By Berkovich's classification theorem, 
points of $\BPP_K \backslash \PP^1(K)$ correspond to discs $D(a,r) \subset K$, or to nested sequences of discs;  
points corresponding to discs with radius $r \in |K^{\times}|$ are said to be of  {type {\rm II}}.  
The point $\zeta_G$ corresponding to $D(0,1)$ is called the {\em Gauss point}.  
The natural action of $\GL_2(K)$ on $\PP^1(K)$ extends continuously to $\BPP_K$, 
and $\GL_2(K)$ acts transitively on type II points. 

If $\gamma \in \GL_2(K)$, we denote the conjugate 
$\gamma^{-1} \!\circ\! \varphi \!\circ\! \gamma$ by $\varphi^{\gamma}$.
In \cite{RR-MRL} it is shown that the map 
$\gamma \mapsto \ordRes(\varphi^\gamma)$ factors through a 
function $\ordRes_\varphi(\cdot)$ on $\BPP_K$,  
given on type II points by 
\begin{equation*} 
\ordRes_\varphi(\gamma(\zeta_G)) \ = \ \ordRes(\varphi^\gamma) \ . 
\end{equation*} 
By (\cite{RR-MRL}, Theorem 0.1) $\ordRes_\varphi(\cdot)$ is piecewise affine and convex upward on paths, 
and takes the value $\infty$ on $\PP^1(K)$.
It achieves a minimum on $\BPP_K$. 
The set $\MinResLoc(\varphi)$ where the minimum occurs  
is called the {\em Minimal Resultant Locus} of $\varphi$. 
It is either a single type II point, or a closed segment joining two type II points.

\smallskip
In this paper we investigate the dynamical and geometrical meaning of $\MinResLoc(\varphi)$. 
Our key discovery is the fact that there is a canonical way to assign 
non-negative integer weights $w_\varphi(P)$ to points of $\BPP_K$, such that  
\begin{equation} \label{WeightSumFormula}
\sum_{\ \ P \in \,\BPP_K} w_\varphi(P) \ = \ d-1 \ . 
\end{equation} 
We call the set of points with positive weight the {\em crucial set} of $\varphi$.
The elements of the crucial set are all of type II.  When $\varphi$ has potential good reduction,
the crucial set consists of the unique point where $\varphi$ achieves good reduction.  When 
$\varphi$ has bad reduction, we regard the crucial set as a distributed analogue of the point 
of good reduction (though even when $\varphi$ has bad reduction, the crucial set sometimes
consists of a single point;  see Example G of \S\ref{ExamplesSection}). 

Each repelling fixed point of $\varphi$ in the Berkovich upper half-space $\BHH_K := \BPP_K \backslash \PP^1(K)$ 
belongs to the crucial set.  One consequence of the Weight Formula is that $\varphi$
can have at most $d-1$ repelling fixed points in $\BHH_K$.  
Example C in \S\ref{ExamplesSection} shows this is sharp.

Write $\delta_P$ for the Dirac measure at a point $P \in \BPP_K$.
The crucial set carries a natural probability measure  
\begin{equation*}
\nu_\varphi \ = \ \frac{1}{d-1} \sum_{ \ P \in \,\BPP_K} w_\varphi(P) \delta_P   
\end{equation*}
called the {\em crucial measure}.  The {\em barycenter}
of $\nu_\varphi$ is the set of points $P$ such that each component of $\BPP_K \backslash \{P\}$ 
has $\nu_\varphi$-mass at most $1/2$.  
The tree $\Gamma_\varphi$ spanned by the crucial set will be called
the {\em crucial tree}; we define its vertices to be the elements of the crucial set and the branch points,
and its edges to be the closed segments connecting vertices. 

\smallskip
We can now state our two main theorems:

\medskip
\noindent{\bf Theorem  A} {\rm \big(Dynamical Characterization of $\MinResLoc(\varphi)$\big).} 
{\em Let $\varphi(z) \in K(z)$ have degree $d \ge 2$.  
Then $\MinResLoc(\varphi)$ is the barycenter of $\nu_\varphi$.  
If $d$ is even, $\MinResLoc(\varphi)$ is a vertex 
o\!f \,$\Gamma_\varphi$.  
If $d$ is odd, $\MinResLoc(\varphi)$ is either a vertex
or an edge o\!f \,$\Gamma_\varphi$.  
}

\medskip
Using Geometric Invariant Theory, Silverman (\cite{Sil2}) has constructed a moduli space 
$\cM_d/\Spec(\ZZ)$ for rational functions of degree $d$.
Building on work of Szpiro, Tepper, and Williams in (\cite{STW}), we show  

\medskip
\noindent{\bf Theorem  B} {\rm \big(\rm Moduli-Theoretic Characterization of $\MinResLoc(\varphi)$\big).} 
{\em Suppose $\varphi(z) \in K(z)$ has degree $d \ge 2$.  Let $P \in \BHH_K$ be a point of type {\rm II}, and let 
$\gamma \in \GL_2(K)$ be such that $P = \gamma(\zeta_G)$.  Then 

$(A)$ $P \in \MinResLoc(\varphi)$ if and only if $\varphi^\gamma$ is has semi-stable reduction
in the sense of Geometric Invariant Theory.

$(B)$  $\MinResLoc(\varphi) = \{P\}$ if and only if  
$\varphi^\gamma$ has stable reduction in the sense of Geometric Invariant Theory.
} 
\medskip

The plan of the paper is as follows.   

By (\cite{RR-MRL}, Theorem 0.1), $\MinResLoc(\varphi)$ is contained in the tree $\Gamma_{\Fix,\varphi^{-1}(a)}$
spanned by the classical fixed points of $\varphi$ and the preimages of $a$, for each $a \in \PP^1(K)$.  
We first show that the intersection 
of the trees $\Gamma_{\Fix,\varphi^{-1}(a)}$ is the tree $\Gamma_{\Fix,\Repel}$  
spanned by the classical fixed points and the repelling fixed points of $\varphi$ in $\BHH_K$.  
In \S\ref{DefinitionSection}  we recall some basic facts and definitions. 
After preliminaries concerning the surplus multiplicity $s_\varphi(P,\vv)$ in \S\ref{IdentificationLemmaSection},  
and repelling fixed points in \S\ref{GammaFixRepelSection}, 
we prove the tree intersection theorem in \S\ref{TreeTheoremSection}.
As an application, we obtain a fixed point theorem for balls which $\varphi$ maps onto $\BPP_K$
(Theorem \ref{IndifFixedPtThm}). 

The Weight Formula (\ref{WeightSumFormula}) follows by computing 
the Laplacian of the restriction of $\ordRes_{\varphi}(\cdot)$ to $\Gamma_{\Fix,\Repel}$.
In \S\ref{SlopeFormulaSection} we carry out some slope computations needed for the Weight Formula,
and in \S\ref{WeightFormulaSection}
we establish the Weight Formula  and study the crucial set. 
In \S\ref{CharacterizationSection}, we prove Theorems A and B  
as Theorems \ref{BaryCenterTheorem} and \ref{ModuliTheorem}, respectively.

The remainder of the paper contains supplements to Theorems A and B.  
In \S\ref{BalanceConditionSection} we give 
``balance conditions'' for a point to belong to $\MinResLoc(\varphi)$, in terms of surplus multiplicities and 
directional fixed point counts. 
In \S\ref{PersistenceTheoremSection} and \S\ref{ApplicationsSection}, 
we establish several structure theorems 
concerning the dynamics of $\varphi$, including a sharpening of Theorem A when $\MinResLoc(\varphi)$
is a segment (Theorem \ref{SegmentCharacterization}). In \S\ref{ExamplesSection} we conclude with examples.  

\section{\bf Basic Facts and Definitions.} \label{DefinitionSection} 

The Berkovich projective line 
is a path-connected Hausdorff space containing $\PP^1(K)$.
Following now-standard notation, 
we denote the Berkovich projective line by $\BPP_K$ (in \cite{B-R}, it was written $\PP^1_\Berk$). 
By Berkovich's classification theorem (see e.g. \cite{B-R}, p.5), 
points of $\BPP_K \backslash \{\infty\}$ correspond to discs, or nested sequences of discs, in $K$.  
There are four kinds of points:   
type I points are the points of $\PP^1(K)$, which we regard as discs of radius $0$.  
Type II points correspond to discs $D(a,r) = \{z \in K : |z-a| \le r\}$ such that $r$ belongs to the value group
$|K^{\times}|$. 
Type III points correspond to discs $D(a,r)$  with $r \notin |K^{\times}|$. 
We write $\zeta_{a,r}$ for the point corresponding to $D(a,r)$.  
The type II point $\zeta_{0,1}$ plays a special role;  
it is called the {\em Gauss point}, and is denoted $\zeta_G$.
Type IV points serve to complete $\BPP_K$; they correspond to (cofinal equivalence classes of) 
sequences of nested discs with empty intersection. If $\{D(a_i,r_i)\}_{i \ge 1}$ is such a sequence,
by abuse of notation we continue to write $\zeta_{a,r}$ for the associated point in $\BPP_K$. 

Paths in $\BPP_K$ correspond to ascending or descending chains of discs, 
or unions of chains sharing an endpoint.  
For example the path from $0$ to $1$
in $\BPP_K$ corresponds to the chains 
$\{D(0,r) : 0 \le r \le 1\}$ and $\{D(1,r) : 1 \ge r \ge 0\}$;  here $D(0,1) = D(1,1)$. 
Topologically, $\BPP_K$ is a tree:     
there is a unique path $[P,Q]$ between any two points $P, Q \in \BPP_K$.
We write $(P,Q)$ for the interior of that path, with similar notation for half-open segments.

If $P \in \BPP_K$, the {\em tangent space} $T_P$ is the set of equivalence classes of paths $(P,x]$
as $x$ varies over $\BPP_K \backslash \{P\}$;  
we call paths $(P,x]$ and $(P,y]$ {\em equivalent} if they 
share a common initial segment.  We call elements of $T_P$ 
{\em tangent vectors} or {\em directions}, and denote them by vectors;   
given $\vv \in T_P$, we write $B_P(\vv)^- = \{x \in \BPP_K : [P,x] \in \vv\}$ 
for the associated path-component of $\BPP_K \backslash \{P\}$.
If $x \in B_P(\vv)^-$, we will say that $x$ lies in the direction $\vv$ at $P$.  
If $P$ is of type I or type IV, then $T_P$ has one element;  
if $P$ is of type III, $T_P$ has two elements; and if $P$ is of type II then  $T_P$ is infinite.  
When $P = \zeta_G$, there is a natural $1-1$ correspondence between elements of $T_{\zeta_G}$ and $\PP^1(\tk)$.  
More generally, for an arbitrary type II point $P$, a map $\gamma \in \GL_2(K)$ with $\gamma(\zeta_G) = P$ 
induces a parametrization of $T_P$ by $\PP^1(\tk)$; 
if $a \in \PP^1(\tk)$ we write $\vv_a \in T_P$ for the corresponding direction.     
 
The set $\BHH_K = \BPP_K \backslash \PP^1(K)$ (written $\HH_\Berk$ in \cite{B-R})
is called the {\em Berkovich upper halfspace}; 
it carries a metric $\rho(x,y)$ called the {\em logarithmic path distance}, for which the length of the 
path corresponding to $\{D(a,r) : R_1 \le r \le R_2\}$ is $\log(R_2/R_1)$. 
(We normalize the function $\log(t)$ so that $\ord(x) = -\log(|x|)$.)  
There are two natural topologies on $\BPP_K$, called the {\em weak}\, and {\em strong} topologies. 
The weak topology on $\BPP_K$ is the coarsest one which makes the evaluation functionals 
$z \rightarrow |f(z)|$ continuous for all $f(z) \in K(z)$;   
under the weak topology, $\BPP_K$ is compact and $\PP^1(K)$ is dense in it.
The basic open sets for the weak topology are the path-components 
of $\BPP_K \backslash \{P_1, \ldots, P_n\}$
as $\{P_1, \ldots, P_n\}$ ranges over finite subsets of $\BHH_K$. 
The strong topology on $\BPP_K$ (which is finer than the weak topology) 
restricts to the topology on $\BHH_K$ induced by $\rho(x,y)$.  
The basic open sets for the strong topology 
are the $\rho(x,y)$-balls in $\BHH_K$, 
together with the basic open sets from the weak topology.  
Type II points are dense in $\BPP_K$ for both topologies.

The action of $\varphi$ on $\PP^1(K)$ extends functorially to $\BPP_K$. 
If $\varphi$ is nonconstant, the induced map $\varphi : \BPP_K \rightarrow \BPP_K$ 
is surjective, open and continuous for both topologies, and takes points of a given type to points of the same type;  
if $\varphi(P) = Q$, there is an induced surjection $\varphi_* : T_P \rightarrow T_Q$.  
The action of  $\GL_2(K)$ on $\PP^1(K)$ extends 
to an action on $\BPP_K$ which is continuous for both topologies  
and preserves the type of each point.  
$\GL_2(K)$ acts transitively on type II points.  It preserves the logarithmic path distance:  
$\rho(\gamma(x),\gamma(y)) = \rho(x,y)$ for all $x, y \in \BHH_K$ and $\gamma \in \GL_2(K)$.   

At each  $P \in \BPP_K$, the map $\varphi$ has a {\em local degree} $\deg_\varphi(P)$ 
(called the multiplicity $m_{\varphi}(P)$ in \cite{B-R}), 
which is a positive integer in the range $1 \le \deg_{\varphi}(P) \le d$.
It has the property that for each $Q \in \BPP_K$, $\sum_{\varphi(P) = Q} \deg_\varphi(P) = d$.  
When $P \in \PP^1(K)$, $\deg_\varphi(P)$ coincides with the classical algebraic multiplicity of $\varphi$ at $P$.   

For further properties of $\BPP_K$, see   
(\cite{B-R}, \cite{BIJL}, \cite{Berk}, \cite{Fab}, \cite{F-RL2}, \cite{FRLErgodic}, and \cite{R-L1}). 

\vskip .08 in
We will use two notions of ``reduction'' for $\varphi$.
The first, which we simply call {\em the reduction of $\varphi$}, is defined as follows.  
If $(F,G)$ is a normalized representation of $\varphi$, 
the reduction $\tphi$ is the rational map on $\PP^1(\tk)$ 
gotten by reducing $F$ and $G$ $\pmod{\fM}$ and eliminating common factors. 
Explicitly, 
let $\tF, \tG \in \tk[X,Y]$ be the reductions of $F, G \pmod{\fM}$ and put $\tA(X,Y) = \GCD(\tF(X,Y), \tG(X,Y))$.  
Write $\tF(X,Y) = \tA(X,Y) \tF_0(X,Y)$, $\tG(X,Y) = \tA(X,Y) \tG_0(X,Y)$.  Then $\tphi : \PP^1(\tk) \rightarrow \PP^1(\tk)$ 
is the map defined by $(X,Y) \mapsto (\tF_0(X,Y):\tG_0(X,Y))$.  
If $\tphi$ has degree $d = \deg(\varphi)$, then $\varphi$ is said to have
{\em good reduction}.  

\vskip .08 in
The second, which we call the {\em reduction of $\varphi$ at $P$},
is obtained by fixing a type II point $P \in \HH_\Berk$, 
choosing a $\gamma \in \GL_2(K)$ with $\gamma(\zeta_G) = P$,  
and taking the reduction $\tphi_P$ of $\varphi^{\gamma} = \gamma^{-1} \circ \varphi \circ \gamma$ in the sense above. 
(If $(F,G)$ is a normalized representation of $\varphi^\gamma$, 
we call $(F,G)$ a {\em normalized representation of $\varphi$ at $P$}.) 
When $P = \zeta_G$ and $\gamma = \id$, this notion of reduction coincides with previous one.
The map $\tphi_P$ depends on the choice of $\gamma$, 
but if $\gamma^\prime \in \GL_2(K)$ also satisfies $\gamma^\prime(\zeta_G) = \zeta$, and $\tphi^{\prime}_P$
is the reduction of $\varphi$ at $P$ corresponding to $\gamma^{\prime}$, there is an $\teta \in \GL_2(\tk)$
such that $\tphi^{\prime}_P  = \teta^{-1} \circ \tphi_P \circ \teta$.  
Thus $\deg(\tphi)$, and the properties that $\tphi$ is constant, is nonconstant, 
or is  the identity map, are independent of the choice of $\gamma$. 
If $\varphi$ does not have good reduction,
but after a change of coordinates by some $\gamma \in \GL_2(K)$ 
the map $\varphi^{\gamma}$ has good reduction, then $\varphi$ is said to have {\em potential good reduction}.  
    
A theorem of Rivera-Letelier in \cite{R-L0} (see \cite{B-R}, Corollary 9.27) says that when $P$ is of type II, 
then $\varphi(P) = P$  if and only if $\tphi_P$ is nonconstant. Rivera-Letelier  
shows that in that case  $\deg_\varphi(P) = \deg(\tphi_P)$;  he calls      
$P$ a {\em repelling fixed point} if $\deg(\tphi_P) \ge 2$, 
and an {\em indifferent fixed point} if $\deg(\tphi_P) = 1$.  
When $\varphi(P) = P$, the induced map $\varphi_* : T_P \rightarrow T_P$ on the tangent space 
comes from the action of $\tphi_P$ on $\PP^1(\tk)$.   
If $\varphi(P) \ne P$, the map $\tphi_P$ is constant, and in that case,
if $\tphi_P(z) \equiv a \in \PP^1(\tk)$, then $\varphi(P)$ belongs to the ball $B_P(\vv_a)^-$.   

\medskip
Another theorem of Rivera-Letelier 
(see \cite{R-L0}, Lemma 2.1, or \cite{B-R}, Proposition 9.41; 
see Faber \cite{Fab} I: Proposition 3.10  for the definitive version), 
says that for each $P \in \BPP_K$ and each $\vv \in T_P$, 
there are integers $m = m_\varphi(P,\vv) \ge 1$ and $s = s_\varphi(P,\vv) \ge 0$ 
such that 

\begin{itemize}
\item[(a)] for each $y \in B_{\varphi(P)}(\varphi_*(\vv))^-$ 
there are exactly $m + s$ solutions to $\varphi(x) = y$ in $B_P(\vv)^-$ 
(counted with multiplicities), and 

\item[(b)] for each $y \in \BPP_K \backslash \big(B_{\varphi(P)}(\varphi_*(\vv))^- \cup \{\varphi(P)\}\big)$,  
there are exactly $s$ solutions to $\varphi(x) = y$ in $B_P(\vv)^-$ 
(counted with multiplicities).
\end{itemize}
The number $m_\varphi(P,\vv)$ is called the {\em directional multiplicity} 
of $\varphi$ at $P$ in the direction $\vv$,
and $s_\varphi(P,\vv)$ is called the {\em surplus multiplicity} 
of $\varphi$ at $P$ in the direction $\vv$.
Several formulas relating $m_\varphi(P,\vv)$ 
to geometric quantities are given in (\cite{B-R}, Theorem 9.26).
In particular, when $P$ is of type II and $\varphi(P) = P$, then 
then $m_\varphi(P,\vv_a)$ is the algebraic multiplicity of $\tphi_P$ at $a$,
for each for $a \in \PP^1(\tk)$. 
An important theorem of Faber (\cite{Fab}, I: Lemma 3.17), 
says that when $\varphi(P) = P$, if $(F,G)$ is a normalized representation of $\varphi$ at $P$,
and if $\tF(X,Y) = \tA(X,Y) \tF_0(X,Y)$,
$\tG(X,Y) = \tA(X,Y) \tG_0(X,Y)$, then 
$s_\varphi(P,\vv_a)$ is the multiplicity of $a$ as a root of $\tA(X,Y)$.

\medskip

Let $\varphi(z) \in K(z)$ have degree $d \ge 2$, and let $(F,G)$
be a normalized representation of $\varphi$.  
Writing $F(X,Y) = f_d X^d +f_{d-1} X^{d-1}Y + \cdots + f_0 Y^d$ 
and  $G(X,Y) = g_d X^d + g_{d-1} X^{d-1}Y + \cdots + g_0 Y^d$,  
the resultant of $F$ and $G$ is 
\begin{equation} \label{FGResultant} 
\Res(F,G) \ = \ \det\Bigg( \ \left[ \begin{array}{cccccccc} 
                          f_d & f_{d-1} & \cdots &  f_1    &    f_0     &         &      &      \\
                              & f_d & f_{d-1}    & \cdots  &    f_1     &    f_0  &      &      \\
                              &     &            &         & \vdots     &         &      &      \\
                              &     &            &    f_d  &   f_{d-1}  & \cdots  & f_1  & f_0  \\
                          g_d & g_{d-1} & \cdots &  g_1    &    g_0     &         &      &      \\
                              & g_d & g_{d-1}    & \cdots  &    g_1     &    g_0  &      &      \\
                              &     &            &         & \vdots     &         &      &      \\
                              &     &            &    g_d  &   g_{d-1}  & \cdots  & g_1  & g_0 
                             \end{array} \right] \ \Bigg) \ .
\end{equation} 
The quantity $\ordRes(\varphi) := \ord(\Res(F,G))$
is independent of the choice of normalized representation; by construction, it is non-negative. 
It is well-known (see e.g. \cite{Sil}, Theorem 2.15) 
that $\varphi$ has good reduction if and only if $\ordRes(\varphi) = 0$.

The starting point for the investigation in \cite{RR-MRL} was the following observation.  
By standard formulas for the resultant (see for example (Silverman \cite{Sil}, Exercise 2.7, p.75)), 
for each $\gamma \in \GL_2(K)$ and each $\tau \in K^{\times} \cdot \GL_2(\cO)$,  one has 
\begin{equation*}
\ordRes(\varphi^{\gamma}) \ = \ \ordRes(\varphi^{\gamma \tau}) \ . 
\end{equation*}  
On the other hand, $K^{\times} \cdot \GL_2(\cO)$ is the stabilizer of the Gauss point.  
Since $\GL_2(K)$ acts transitively on the type II points in $\BPP_K$,   
the map $\gamma \mapsto \ordRes(\varphi^{\gamma})$
factors through a well-defined function \ $\ordRes_{\varphi}(\cdot)$ \ 
on type II points given by 
\begin{equation} \label{ordRes_varphiDef}
\ordRes_{\varphi}\big(\gamma(\zeta_G)\big) \ := \ \ordRes(\varphi^{\gamma}) \ .
\end{equation}

The main result in \cite{RR-MRL}
(a combination of \cite{RR-MRL}, Theorem 0.1 and Proposition 3.5) is 

\begin{theorem} \label{ResThm} Let $\varphi(z) \in K(z)$, and suppose $d = \deg(\varphi) \ge 2$.  
The function $\ordRes_{\varphi}(\cdot)$ on type {\rm II} points 
extends to a function $\ordRes_{\varphi} :\BPP_K \rightarrow [0,\infty]$
which is continuous for the strong topology, 
is finite on $\BHH_K$, and takes the value $\infty$ on $\PP^1(K)$.  The extended
function $f(\cdot) = \ordRes_\varphi(\cdot)$ has the following properties$:$  
\begin{itemize}
\item[(A)] $f(\cdot)$  is piecewise affine and convex upwards on each path in $\BPP_K$ 
relative to the logarithmic path distance$;$ the slope of each affine piece is an integer $m$  
satisfying $m \equiv d^2 + d \pmod{2d}$ and lying in the range $-(d^2+d) \le m \le (d^2 + d);$ 
the breaks between affine pieces occur at type {\rm II} points.   

\item[(B)] $\ordRes_\varphi(\cdot)$ achieves a minimum value on $\BPP_K$. 
The set $\MinResLoc(\varphi)$  where the minimum is taken on 
is a single type {\rm II} point if $d$ is even, 
and is either a single type {\rm II} point or a segment with type {\rm II} endpoints if $d$ is odd.  

\item[(C)] For each $a \in \PP^1(K)$, $f(\cdot)$ is strictly increasing away from the tree 
$\Gamma_{\Fix,\varphi^{-1}(a)} \subset \BPP_K$ 
spanned by the fixed points of $\varphi$ and the pre-images of $a$ in $\PP^1(K)$.
\end{itemize}
\nopagebreak
In particular, $\MinResLoc(\varphi)$ belongs to $\Gamma_{\Fix,\varphi^{-1}(a)}$, for each $a \in \PP^1(K)$.  
\end{theorem} 

\noindent{The} aim of this paper is to explain the dynamical and geometric meaning of $\MinResLoc(\varphi)$.

\section{\bf The Identification Lemmas.} \label{IdentificationLemmaSection}

As will be seen, the surplus multiplicity $s_\varphi(P,\vv)$ plays an important role in the study of $\MinResLoc(\varphi)$.
The discovery which launched this investigation is the fact that there is a relation between directions $\vv \in T_P$ 
for which $s_\varphi(P,\vv) >0$, and directions containing type I fixed points.  This relation is expressed in two
lemmas which we call ``Identification Lemmas'', because they identify the directions which can have $s_\varphi(P,\vv) > 0$.  


\smallskip
If $P \in \BHH_K$ is a type II fixed point of $\varphi$, and $\tphi_P$ is the reduction of $\varphi$ at $P$,
we say that $P$ is id-indifferent for $\varphi$ if $\tphi_P$ is the identity map.  (A more general classification 
of fixed points is given in Definitions \ref{IndifferentClassificationDef}
and \ref{RepellingFixedPtClassification} of \S\ref{GammaFixRepelSection}.)

\begin{definition}[Directional and Reduced Fixed Point Multiplicities] \label{DFA_Defs} { \ } 

{\em Suppose $\varphi \in K(z)$ is nonconstant.

$(A)$ For each $P \in \BHH_K$, and each $\vv \in T_P$, 
we define the {\em directional fixed point multiplicity} $\#F_\varphi(P,\vv)$ to be the number of 
fixed points of $\varphi$ in $B_P(\vv)^-$ $($counting multiplicities$)$.

$(B)$ If $P \in \BHH_K$ is a type II fixed point of $\varphi$, which is not id-indifferent for $\varphi$,
let $\tphi_P$ be the reduction of $\varphi$ at $P$, and 
parametrize $T_P$ by $\PP^1(\tk)$ in such a way that $\varphi_*(\vv_a) = \vv_{\tphi_P(a)}$ for each $a \in \PP^1(\tk)$.
If $\vv = \vv_a$, we define the {\em reduced fixed point multiplicity} $\#\tF_\varphi(P,\vv)$ 
to be the multiplicity of $a$ as a fixed point of $\tphi_P$.}
\end{definition}   

Clearly $\sum_{\vv \in T_P} F_\varphi(P,\vv) = \deg(\varphi)+1$. If $P$ is a type II fixed point which is not
id-indifferent for $\varphi$, then  $\sum_{\vv \in T_P} \#\tF_\varphi(P,\vv) = \deg_\varphi(P)+1$ 

\begin{lemma}[First Identification Lemma] \label{FirstIdentificationLemma}
Suppose $P \in \BHH_K$ is of type {\rm II}, and that $\varphi(P) = P$, but $P$ is not id-indifferent.
Let $\vv \in T_P$. Then 
\begin{equation} \label{MultSumFormula} 
\#F_\varphi(P,\vv) \ = \ s_\varphi(P,\vv) \ + \ \#\tF_\varphi(P,\vv) \ .
\end{equation} 
In particular, $B_P(\vv)^-$ contains a type {\rm I} fixed point of $\varphi$ 
if and only if 

$(A)$ $\varphi_*(\vv) = \vv$, and/or 

$(B)$ $s_\varphi(P,\vv) > 0$. 
\end{lemma}

\begin{proof} 
Choose $\gamma \in \GL_2(K)$ with $\gamma(P) = \zeta_G$, and put $\Phi = \varphi^\gamma$;
then $\Phi(\zeta_G) = \zeta_G$.
Note that $\xi \in \PP^1(K)$ is fixed by  $\varphi$ if and only if $\gamma^{-1}(\xi)$
is fixed by $\Phi$, and that $\vv \in T_P$ is fixed by $\varphi_*$ if and only if $(\gamma^{-1})_*(\vv) \in T_{\zeta_G}$
is fixed by $\Phi_*$. Hence it suffices to prove the result with $\varphi$ replaced by $\Phi$.

Choose $F(X,Y), G(X,Y) \in \cO[X,Y]$ so $(F,G)$ is a normalized representation of $\Phi$. 
Let $\tF, \tG \in \tk[X,Y]$ be the reductions of $F, G$, 
and put $\tA(X,Y) = \GCD(\tF(X,Y), \tG(X,Y))$.  
Write $\tF(X,Y) = \tA(X,Y) \cdot \tF_0(X,Y)$ and  $\tG(X,Y) = \tA(X,Y) \cdot \tG_0(X,Y)$.  Then  $\tvarphi$
is the map on $\PP^1(\tk)$ defined by $(X,Y) \mapsto (\tF_0(X,Y):\tG_0(X,Y))$.  

Put $H(X,Y) = X G(X,Y) - Y F(X,Y)$, and $\tH_0(X,Y) = X \tG_0(X,Y) - Y \tF_0(X,Y)$. 
Here $H(X,Y) \ne 0$ since $\deg(\varphi) > 1$, 
and $\tH_0(X,Y) \ne 0$ since $\tphi(z) \not\equiv z$ by assumption.
The fixed points of $\Phi$ are the zeros of $H(X,Y)$ in $\PP^1(K)$, 
and the fixed points of $\tphi$ are the zeros of $\tH_0(X,Y)$ in $\PP^1(\tk)$. 
   
Reducing $H(X,Y)$ modulo $\fM$, we see that 
\begin{equation} \label{HF1} 
\tH(X,Y) \ = \ X \tG(X,Y) - Y \tF(X,Y) \ = \ \tA(X,Y) \cdot \tH_0(X,Y) \ .
\end{equation}     
Since $K$ is algebraically closed, $H(X,Y)$ factors over $K[X,Y]$ 
as a product of linear factors.
After scaling the factors if necessary, we can assume that the factorization has the form
\begin{equation} \label{HF2}
H(X,Y) \ = \ \prod_{i=1}^{d+1} (b_i X - a_i Y)  
\end{equation}
where $\max(|a_i|,|b_i|) = 1$ for each $i = 1, \ldots, d+1$.  We claim that $|\,C| = 1$ as well. 
To see this, note that 
if we choose $u,v$ in $\cO$ so that $(\tu:\tv)$ is not a zero of either $\tA(X,Y)$ or $\tH_0(X,Y)$, 
then $\tH(\tu,\tv) \ne 0$ by (\ref{HF1}).  
    
It follows that 
\begin{equation} \label{HF3} 
\tH(X,Y) = \ \prod_{i=1}^{d+1} (\tb_i X - \ta_i Y) \ \not\equiv \ 0 \ .
\end{equation}
Since $\tk[X,Y]$ is a unique factorization domain, 
comparing (\ref{HF1}) and (\ref{HF3}) 
we see that after reordering factors if necessary, there are  
constants $\tC_1, \tC_2 \in \tk^\times$ and an integer $0 \le n \le d$ such that 
\begin{equation} \label{SeparatedMultiplicity}
\tA(X,Y) \ = \tC_1 \cdot \prod_{i=1}^n (\tb_i X - \ta_i Y) \ , \quad 
\tH_0(X,Y) \ = \ \tC_2 \cdot \prod_{i=n+1}^{d+1} (\tb_i X - \ta_i Y) \ .
\end{equation}

The fixed points of $\Phi(X,Y)$ are the points $(a_i:b_i) \in  \PP^1(K)$, $i = 1, \ldots, d+1$.
By (\ref{HF1}) and (\ref{SeparatedMultiplicity}),  each fixed point of $\Phi(X,Y)$ 
specializes to a zero of $\tH_0(X,Y)$ or $\tA(X,Y)$,
and conversely each such zero is the specialization of a fixed point of $\Phi(X,Y)$.  
For a given $\vv \in T_{\zeta_G}$, if $\ta \in \PP^1(\tk)$ is such that $\vv = \vv_{\ta}$, 
then $\#F_\varphi(P,\vv)$ is the multiplicity of $\ta$ as a zero of $\tH(X,Y)$, 
$s_\varphi(P,\vv)$ is the multiplicity of $\ta$ as a zero of $\tA(X,Y)$, and $\#\tF_\varphi(P,\vv)$
is the multiplicity of $\ta$ as a root of $\tH_0(X,Y)$.   
Thus $\#F_\varphi(P,\vv) = s_\varphi(P,\vv) + \#\tF_\varphi(P,\vv)$. 
\end{proof}

\begin{lemma}[Second Identification Lemma] \label{SecondIdentificationLemma}
Suppose $P \in \BHH_K$ is of type {\rm II}, and that $\varphi(P) = Q \ne P$.
Let $\vv \in T_P$.  Then $B_P(\vv)^-$ contains a type {\rm I} fixed point of $\varphi$ 
iff 

$(A)$ $Q \in B_P(\vv)^-$, and/or 

$(B)$ $P \in B_Q(\varphi_*(\vv))^-$, and/or 

$(C)$ $s_\varphi(P,\vv) > 0$. 
\end{lemma}

\begin{proof} 
Let $Q = \varphi(P)$.  
We claim that there is a $\gamma \in \GL_2(K)$ 
such that $\gamma(\zeta_G) = P$ and $\gamma^{-1}(Q) = \zeta_{0,r)}$,  
for some $0 < r < 1$.  
To see this, take any $\tau_0 \in \GL_2(K)$ with $\tau_0(\zeta_G) = P$.
Then $\tau_0^{-1}(Q) \ne \zeta_G$, since $\tau_0$ is $1-1$.  
Let $\alpha \in \PP^1(K)$ be such that the path $[\alpha,\zeta_G]$ contains $\tau_0^{-1}(Q)$,
and choose $\tau_1 \in \GL_2(\cO)$ with $\tau_1(0) = \tau_0^{-1}(\alpha)$. (Such a $\tau_1$ exists because
$\GL_2(\cO)$ acts transitively on $\PP^1(K)$).  Setting $\gamma = \tau_0 \circ \tau_1$, 
we see that $\gamma(\zeta_G) = P$ and $\gamma(0) = \gamma_0(\tau_1(0)) = \alpha$.  This means that  
$\gamma^{-1}(Q) = \tau_1^{-1}(\tau_0^{-1}(Q))$ 
belongs to $\tau_1^{-1}([\alpha,\zeta_G]) = [0,\zeta_G]$, 
so $\gamma^{-1}(Q) = \zeta_{0,r)}$ for some $0 < r < 1$.  Here $r \in |K^{\times}|$, 
since $\zeta_{0,r} = \gamma^{-1}(\varphi(P))$ is of type II. 
Since the fixed points of $\varphi$ and conditions (A), (B), (C) in the Lemma are equivariant under conjugation, 
by replacing $\varphi$ with $\varphi^{\gamma} = \gamma^{-1} \circ \varphi \circ \gamma$, 
we can reduce to the case where $P = \zeta_G$ and $Q = \zeta_{0,r}$. 
 
Take any $c \in K^{\times}$ with $|\,c| = r$.  
Put $\gamma_1(z) = \id$ and $\gamma_2(z) = c z$;  then $\gamma_1(\zeta_G) = \zeta_G  = P$ and  
$\gamma_2(\zeta_G) = \zeta_{0,r} = Q$.  
Put $\Phi = \gamma_2^{-1} \circ \varphi \circ \gamma_1$, 
noting that $\varphi(z) = c \cdot \Phi(z)$.

Choose $F(X,Y), G(X,Y) \in \cO[X,Y]$ so $(F,G)$ is a normalized representation of $\Phi$. 
Let $\tF, \tG \in \tk[X,Y]$ be their reductions, 
and put $\tA(X,Y) = \GCD(\tF(X,Y), \tG(X,Y))$.  
Write $\tF(X,Y) = \tA(X,Y) \cdot \tF_0(X,Y)$,  $\tG(X,Y) = \tA(X,Y) \cdot \tG_0(X,Y)$.  The
reduction $\tPhi$ of $\Phi$ is the map $(X,Y) \mapsto (\tF_0(X,Y):\tG_0(X,Y))$, 
so by Faber's theorem (\cite{Fab}, I: Lemma 3.17), the directions $\vv_a \in T_P$ with $s_\varphi(P,\vv_a) > 0$
(which are the same as the ones with $s_\Phi(P,\vv_a) > 0$) 
are precisely those corresponding to points $a \in \PP^1(\tk)$ with $\tA(a) = 0$. 

Since $\varphi = c \cdot \Phi$, the pair $(cF,G)$ is a normalized representation of $\varphi$.    
The fixed points of $\varphi$ in $\PP^1(K)$ are the 
zeros $(a_i:b_i)$ of $H(X,Y) = X G(X,Y) - Y \cdot c F(X,Y)$.  
As in the proof of Lemma \ref{FirstIdentificationLemma}, we can write
\begin{equation} \label{HHF1}
H(X,Y) \ = \ \prod_{i=1}^{d+1} (b_i X - a_i Y)  
\end{equation}
where $\max(|a_i|,|b_i|) = 1$ for each $i = 1, \ldots, d+1$. 
Reducing $H(X,Y)$ modulo $\fM$, we see that 
\begin{equation} \label{HHF2} 
\tH(X,Y) = \ \prod_{i=1}^{d+1} (\tb_i X - \ta_i Y) \ \not\equiv \ 0 \ .
\end{equation}

On the other hand, since $H(X,Y) = X G(X,Y) - Y \cdot cF(X,Y)$ with $|\,c| < 1$, we also have 
\begin{equation} \label{HHF3} 
\tH(X,Y) \ = \ X \tG(X,Y) \ = \ X \cdot \tA(X,Y) \cdot \tG_0(X,Y) \ .
\end{equation}     
Comparing (\ref{HHF2}) and (\ref{HHF3}),
we see that after reordering the factors if necessary, there are 
constants $\tC_1, \tC_2, \tC_3 \in \tk^\times$ and an integer $1 \le n \le d$ such that in $\tk[X,Y]$ 
\begin{eqnarray} 
X & = & \tC_1 \cdot (\tb_1 X - \ta_1 Y) \ , \label{QH1} \\
\tA(X,Y) & = & \tC_2 \cdot \prod_{i=2}^n (\tb_i X - \ta_i Y) \ , \label{QH2} \\
\tG_0(X,Y) & = & \tC_3 \cdot \prod_{i=n+1}^{d+1} (\tb_i X - \ta_i Y) \ , \label{QH3} 
\end{eqnarray}
and each fixed point of $\varphi$ corresponds to a factor of one of these terms.  

From (\ref{QH1}) it follows that $|\,a_1| < 1$ and $|\,b_1| = 1$. Thus the fixed point 
$(a_1:b_1)$ belongs to $B_P(\vv_0)^-$, where $\vv_0 \in T_P$ 
is the tangent direction corresponding to $0 \in \PP^1(\tk)$.  We can express this in a coordinate-free way 
by noting that $B_P(\vv_0)^-$ is the ball containing $Q = \zeta_{0,r} = \varphi(P)$.

From (\ref{QH2}), and the fact that the zeros of $\tA(X,Y)$ correspond to the directions $\vv_a$ 
for which $s_\varphi(P,\vv_a) > 0$, 
we see that each direction $\vv \in T_P$ with $s_\varphi(P,\vv) > 0$
contains a fixed point of $\varphi$. 

Finally, from (\ref{QH3}), we see that each direction $\vv_a \in T_P$ corresponding to a zero of $\tG_0(X,Y)$ 
contains a fixed point of $\varphi$.  However, the zeros of $\tG_0(X,Y)$ are the poles of $\tPhi$,  
and if $\tPhi(a) = \infty$ then $\Phi_*(\vv_a) = \vv_{\infty} \in T_P$.  
Since $\varphi = c \cdot \Phi$, and $|\, c| < 1$,
it follows that $\varphi_*(\vv_a) \in T_Q = T_{\zeta_{0,r}}$ is the ``upwards'' 
direction $\vv_{Q,\infty} \in T_Q$.  
This can be expressed a coordinate-free manner by noting that $B_Q(\varphi_*(\vv_a))^- = B_Q(\vv_{Q,\infty})^-$ 
is the unique ball containing  $P = \zeta_G$.  
\end{proof}  

As an immediate consequence of Lemmas \ref{FirstIdentificationLemma} and \ref{SecondIdentificationLemma} 
we have  

\begin{corollary} \label{SurplusIdentificationCor}
Let $P$ be of type {\rm II}, and suppose $P$ is not id-indifferent for $\varphi$.  
Then for each $\vv \in T_P$ such that $s_\varphi(P,\vv) > 0$, 
the ball $B_P(\vv)^-$ contains a type {\rm I} fixed point of $\varphi$.  
\end{corollary}

\noindent{\bf Remark.}  Later, when we have proved the Tree Intersection Theorem (Theorem \ref{FixRepelThm})  
we will establish a Third Identification Lemma (Lemma \ref{ThirdIdentificationLemma}), 
dealing with $s_\varphi(P,\vv)$ for points $P$ which are id-indifferent. 


\section{Classification of Fixed points in $\BHH_K$, and the tree $\Gamma_{\Fix,\Repel}$} 
\label{GammaFixRepelSection} 

Recall that a fixed point $P$ of $\varphi$ in $\BHH_K$ is called 
{\em indifferent} if $\deg_\varphi(P) = 1$, and {\em repelling} if $\deg_\varphi(P) > 1$. 
In this section we refine these notions.  

 

\begin{definition}[Classification of Indifferent Fixed Points in $\BHH_K$] \label{IndifferentClassificationDef} 
{\em Let $P \in \BHH_K$ be a type {\rm II} indifferent fixed point of $\varphi$, and let $\tphi_P$ be the reduction
of $\varphi$ at $P$.  Then after an appropriate change of coordinates, $\tphi_P$ is of one of three types:

\begin{itemize}
\item[(A)] $\tphi_P(z) = z$;  in this case we will say $P$ is {\em id-indifferent} for $\varphi$.

\item[(B)] $\tphi_P(z) = \ta z$ for some $\ta \in \tk^\times$ with $\ta \ne 1$;  in this case we will say 
$P$ is {\em multiplicatively indifferent} for $\varphi$, and that $\tphi_P$ has reduced rotation number
$\ta$.  The reduced rotation number is only well-defined as an element of $\{\ta,\ta^{-1}\};$
if coordinates on $\PP^1(\tk)$ are changed by conjugating with $1/z$, then $\ta$ is replaced by $\ta^{-1}$. 
If we want to be more precise, we will proceed as follows. 
Note that the directions $\vv_0, \vv_\infty \in T_P$ corresponding to $0, \infty \in \PP^1(\tk)$ are the only
directions  $\vv \in T_P$ fixed by $\varphi_*$.  Let points 
$P_0 \in B_P(\vv_0)^-$, $P_\infty \in B_P(\vv_\infty)^-$ be given.  We will say that  
$\varphi$ has reduced rotation number $\ta$ for the axis $(P_0,P_\infty)$, and it 
has reduced rotation number $\ta^{-1}$ for the axis $(P_\infty,P_0)$.    

\item[(C)] $\tphi_P(z) = z + \ta$ for some $\ta \in \tk$ with $\ta \ne 0$;  in this case we will say that 
$P$ is {\em additively indifferent} for $\varphi$, with reduced translation number $\ta$.  
The reduced translation number is well-defined, independent of the choice of coordinates on $\PP^1(\tk)$
chosen so that $\tphi_P(\infty) = \infty$.  
\end{itemize} 
}
\end{definition} 

\begin{definition}[Reduced Multiplier] 
{\em Suppose $P \in \BHH_K$ is a type {\rm II} fixed point of $\varphi$ which is not id-indifferent,
and let $\vv \in T_P$ be a direction 
with $\varphi_*(\vv) = \vv$ and $m_\varphi(P,\vv) = 1$.  Let $\tphi_P$ be the reduction of $\varphi$ at
$P$; without loss, assume coordinates have been chosen so that $\vv = \vv_0$. 
Then $0$ is a fixed point of $\tphi_P$.   
Put $\ta = \tphi^{\prime}_P(0) \in \tk$.  We will call $\ta$ the {\em reduced multiplier} of $\varphi$
at $\vv$.}
\end{definition}

\begin{definition}[Classification of Repelling Fixed Points] 
\label{RepellingFixedPtClassification}

{\em Suppose $\varphi(z) \in K(z)$ has degree $d \ge 2$, 
and let $P \in \BHH_K$ be a repelling fixed point of $\varphi$. 
Call the directions $\vv_1, \ldots, \vv_m \in T_P$ such that $B_P(\vv_i)^-$ contains 
a type {\rm I} fixed point of $\varphi$, the {\em focal directions} at $P$ . 

$(A)$ We will say that $P$ is {\rm focused} $\!($or {\rm uni-focused}$)\!$ 
if it has a unique focal direction $\vv_1$.    

$(B)$ We will say that $P$ is {\rm bi-focused}
if it has exactly two focal directions $\vv_1, \vv_2$. 

$(C)$ We will say that $P$ is {\rm multi-focused} if it has $m \ge 3$ focal directions.
}
\end{definition}  

\begin{definition}[The trees $\Gamma_{\Fix,\Repel}$ and $\Gamma_{\Fix}$] \label{TreeDefinitions} 
{\em Suppose $\varphi(z) \in K(z)$ has degree $d \ge 2$.  
Let $\Gamma_{\Fix,\Repel}$  be the tree in $\BPP_K$
spanned by the type {\rm I} $($classical$)$ fixed points of $\varphi$ and the type 
{\rm II} repelling fixed points of $\varphi$ in $\BHH_K$. Let $\Gamma_{\Fix}$  
be the tree in $\BPP_K$ spanned by the type {\rm I} fixed points of $\varphi$.}
\end{definition} 

Clearly $\Gamma_{\Fix} \subset \Gamma_{\Fix,\Repel}$.  Since $\varphi$ has at most $d+1$ distinct type I
fixed points, $\Gamma_{\Fix}$ is a finitely generated tree.  It is possible that $\varphi$ has a single type I 
fixed point of multiplicity $d+1$, in which case $\Gamma_{\Fix}$ is reduced to a point.  

\smallskip
We next describe some properties of focused repelling fixed points.

\begin{proposition} \label{FocusedRepellingProp} 
A repelling fixed point of $\varphi$ in $\BHH_K$ is a focused repelling fixed point 
if and only if it does not belong to $\Gamma_{\Fix}$.  
Each focused repelling fixed point is an endpoint of $\Gamma_{\Fix,\Repel}$.
If $P$ is a focused repelling fixed point, with focus $\vv_1$, then   

$(A)$  $\varphi_*(\vv_1) = \vv_1$,  $m_\varphi(P,\vv_1) = 1$, and $\#\tF_\varphi(P,\vv_1) = \deg_\varphi(P) + 1 \ge 3$.

$(B)$  $s_{\varphi}(P,\vv_1) = d - \deg_{\varphi}(P)$.

$(C)$ For each $\vv \in T_P$ with $\vv \ne \vv_1$ we have $\varphi_*(\vv) \ne \vv$, 

\qquad \quad and $\varphi(B_P(\vv)^-)$ is the ball $B_P(\varphi_*(\vv))^-$.   

$(D)$ For each $\vw \in T_P$, there is at least one $\vv \in T_P$ with $\vv \ne \vv_1$ such that $\varphi_*(\vv) = \vw$.
\end{proposition} 

\begin{proof} 
Let $P$ be a repelling fixed point of $\varphi$.  If $P \notin \Gamma_{\Fix}$ 
all the type I fixed points of $\varphi$ lie in a single ball $B_P(\vv_1)^-$, 
so $P$ is a focused repelling fixed point.  
Conversely, suppose $P$ is a focused repelling fixed point with focus $\vv_1$.  By assumption,
all the type I fixed points of $\varphi$ belong to $B_P(\vv_1)^-$, 
so $\Gamma_{\Fix} \subset B_P(\vv_1)^-$. Thus, $P \notin \Gamma_{\Fix}$.  

We next show that if $P$ is a focused repelling fixed point, it must be an endpoint of $\Gamma_{\Fix,\Repel}$, 
and the only $\vv \in T_P$ fixed by $\varphi_*$ is $\vv_1$. 
After a change of coordinates, we can assume that $P = \zeta_G$.  
Index directions $\vv \in T_P$ by points $\talpha \in \PP^1(\tk)$
and choose coordinates so that $\vv_1$ corresponds to $\talpha = 1$. 
Take any $\vv \in T_P$ with $\vv \ne \vv_1$.  
Since $B_P(\vv)^-$ contains no type I fixed points, Lemma \ref{FirstIdentificationLemma} 
shows that $s_\varphi(P,\vv) = 0$ and $\varphi_*(\vv) \ne \vv$.  
Thus $\varphi(B_P(\vv)^-)$ is a ball, necessarily $B_P(\varphi_*(\vv)^-)$.  
If $P$ were not an endpoint of $\Gamma_{\Fix,\Repel}$, there would be a direction   
$\vv \in T_P$ with $\vv \ne \vv_1$ such that $B_P(\vv)^-$ contained a repelling fixed point $Q$ of $\varphi$.
This is impossible, since $\varphi(B_P(\vv)^-) = B_P(\varphi_*(\vv))^-$ where $\varphi_*(\vv) \ne \vv$,
yet $Q \in B_P(\vv)^-$ and $\varphi(Q) = Q$.  From the fact that $\varphi_*(\vv) \ne \vv$ for all $\vv \ne \vv_1$,
it follows that $\varphi_*(\vv_1) = \vv_1$, since the action of $\varphi_*$ on $T_P$
corresponds to the action of the reduction $\tphi$ on $\PP^1(\tk)$, and $\tphi$ has at least one fixed point.  

Since $P$ is a repelling fixed point, necessarily $\deg_\varphi(P) \ge 2$.  Since $\vv_1$ 
is the only direction fixed by $\varphi_*$, $\talpha = 1$ is the only point of $\PP^1(\tk)$ 
fixed by $\tvarphi(z)$.  Thus, $\talpha = 1$ is a fixed point of $\tphi$ of multiplicity
$\deg(\tphi) + 1 = \deg_\varphi(P) + 1$.  This means $\varphi_*(\vv_1) = \vv_1$ and 
$\#\tF_{\varphi}(P,\vv_1) = \deg_\varphi(P) + 1 \ge 3$.  
Since $\talpha = 1$ is fixed point of $\tphi$ of multiplicity $> 1$, $\alpha$ is not a critical point of $\tphi$, 
so $m_\varphi(P,\vv_1) = 1$.  

By Lemma \ref{FirstIdentificationLemma}, $\vv_1$ is the only direction $\vv \in T_P$ for which $s_{\varphi}(P,\vv) > 0$, 
so $s_\varphi(P,\vv_1) = d - \deg_\varphi(P)$.

Finally, we show that for each $\vw \in T_P$, there is a $\vv \in T_P$
with $\vv \ne \vv_1$ satisfying $\varphi_*(\vv) = \vw$. 
If $\vw \ne \vv_1$, this is trivial, since $\varphi_*: T_P \rightarrow T_P$ is surjective and $\varphi_*(\vv_1) = \vv_1$.
If $\vw = \vv_1$, it follows from the fact that $m_{\varphi}(P,\vv_1) = 1$, since 
\begin{equation*}
\sum_{\vv \in T_P, \, \varphi_*(\vv) = \vv_1} m_\varphi(P,\vv) \ = \ \deg_\varphi(P) \ \ge \ 2 \ .
\end{equation*}
\vskip -.2 in
\end{proof}

Before proceeding further, it may be good to note that focused repelling fixed points can exist.  
Below, we describe a class of maps with a focused repelling fixed point:

\smallskip
\noindent{{\bf Example A.} (Maps with a focused repelling fixed point at $\zeta_G$).}

Fix $d \ge 2$, and let $\tL(X,Y) \in \tk[X,Y]$ be a homogeneous form of degree $d-1$  
with $\tL(X,Y) \ne \tC X^{d-1}$.  
Let $L_1(X,Y), L_2(X,Y) \in \cO[X,Y]$ be homogeneous forms of degree $d-1$
whose reductions satisfy 
\begin{equation*}
\tL_1(X,Y) \ \equiv \ \tL_2(X,Y) \ \equiv \ \tL(X,Y) \pmod{\fM} \ .
\end{equation*} 
Put 
\begin{equation*}
F(X,Y) \ = \ X  L_1(X,Y) \ , \qquad  G(X,Y) \ = \ X^d + Y L_2(X,Y) \ .
\end{equation*} 
For generic $L_1(X,Y), L_2(X,Y)$ we will have $\GCD(F,G) = 1$;  if this fails, 
it can be achieved by perturbing $L_1(X,Y)$ or $L_2(X,Y)$ slightly.
Assume that $\GCD(F,G) = 1$.    

Consider the map $\varphi$ with representation $(F,G)$.  
Write $\tF(X,Y) = \tA(X,Y) \tF_0(X,Y)$, $\tG(X,Y) = \tA(X,Y) \tG_0(X,Y)$.  
The reduced map $\tphi$ is represented by $(\tF_0,\tG_0)$. 
Since $\tF(X,Y) = X \tL(X,Y)$, $\tG(X,Y) = X^d + Y \tL(X,Y)$, we have   
\begin{equation*}
\tA(X,Y) \ = \ \GCD(\tF(X,Y),\tG(X,Y)) \ = \ \GCD(X^{d-1}, \tL(X,Y)) \ .
\end{equation*}  
Since $\tL(X,Y) \ne \tC X^{d-1}$, we must have $\tA(X,Y) = X^s$ for some $0 \le s \le d-2$.
Putting $n = d-s$, it follows that 
$\deg(\tphi) = n \ge 2$.  Thus $\zeta_G$ is a repelling fixed point of $\varphi$,
and $\deg_\varphi(\zeta_G) = n$.  Let $\vv_0 \in T_{\zeta_G}$ be the direction corresponding
to $0 \in \tk \subset \PP^1(\tk)$.  

The fixed points of $\varphi$ in $\PP^1(K)$ are the 
zeros of $H(X,Y) = X G(X,Y) - Y F(X,Y)$.  The reduction of $H(X,Y)$ is  
\begin{eqnarray}
\tH(X,Y) & = & X \tG(X,Y) - Y \tF(X,Y)  \label{HRedComp} \\
         & = & X\cdot (X^d + Y \tL_2(X,Y)) - Y \cdot (X \tL_1(X,Y) ) \ = \ X^{d+1} \notag
\end{eqnarray}    
since $\tL_1(X,Y) = \tL_2(X,Y)$. By the theory of Newton Polygons, the type I fixed points 
all belong to $B_{\zeta_G}(\vv_0)^-$.  
It follows that $\zeta_G$ is a focused repelling fixed point for $\varphi$.  To see this,
note that $\Gamma_\Fix$ is contained in  $B_{\zeta_G}(\vv_0)^-$, so $\zeta_G \notin \Gamma_\Fix$.
Since $\zeta_G$ is a repelling fixed point of $\varphi$, it belongs to $\Gamma_{\Fix,\Repel}$. 
By Proposition \ref{FocusedRepellingProp}, $\zeta_G$ must be an endpoint of $\Gamma_{\Fix,\Repel}$, 
hence a focused repelling fixed point.
   
By taking $\tL(X,Y) = X^s Y^{d-1-s}$ for a given integer $0 \le s \le d-2$, 
we can arrange that $\tA(X,Y) = X^s$.  
Thus for any pair $(n,s)$ with  $2 \le n \le d$ and $n + s = d$, 
there is a $\varphi \in K(z)$ of degree $d$ which has a focused repelling
fixed point at $\zeta_G$ with $\deg_\varphi(\zeta_G) = n$ and $s_\varphi(\zeta_G,\vv_0) = s$.   
\hfill $\Box$   

\medskip
We next consider the properties of bi-focused repelling fixed points.  
First, we will need a lemma.

\begin{lemma} \label{FocusedP1Lemma}  Suppose $P \in \BHH_K$ is of type {\rm II}, and $\vv \in T_P$.  
If there is a focused repelling fixed point $Q \in B_P(\vv)^-$ whose focus $\vv_1$ points towards $P$,
then $s_\varphi(P,\vv) > 0$.  
\end{lemma} 

\begin{proof}  We must show that $\varphi(B_P(\vv)^-) = \BPP_K$.  
Since $(\BPP_K \backslash B_Q(\vv_1)^-) \subset B_P(\vv)^-$, it suffices to show that
\begin{equation*} 
\varphi( \BPP_K \backslash B_Q(\vv_1)^-) \ = \ \BPP_K \ .
\end{equation*} 
To see this, note that 
$\BPP_K \backslash B_Q(\vv_1)^- = \{Q\} \cup (\bigcup_{\vv \in T_Q, \vv \ne \vv _1} B_Q(\vv)^-)$.
By assumption $\varphi(Q) = Q$.
Since $\varphi_* : T_Q \rightarrow T_Q$ is surjective, Proposition \ref{FocusedRepellingProp} shows that 
for each $\vw \in T_Q$ there is at least one $\vv \in T_Q$ with $\vv \ne \vv_1$ 
for which  $\varphi_*(B_Q(\vv)^-) = B_Q(\vw)^-$. 
Since $\BPP_K = \{Q\} \cup (\bigcup_{\vw \in T_Q} B_Q(\vw)^-)$,  the claim follows.
\end{proof} 

\begin{proposition} \label{BeadProp}
A repelling fixed point of $\varphi$ in $\BHH_K$ is a bi-focused repelling fixed point if and only if it belongs 
$\Gamma_\Fix$, but is not a branch point of $\Gamma_{\Fix,\Repel}$.  
Suppose $P$ is a bi-focused repelling fixed point, and let $\vv_1, \vv_2 \in T_P$ be its focal directions.  Then  

$(A)$ For least one $\vv \in \{\vv_1,\vv_2\}$, we have
$\varphi_*(\vv) = \vv$, $m_\varphi(P,\vv) = 1$, and $\#\tF_\varphi(P,\vv) \ge 2$.

$(B)$ For each $\vv \in T_P$ with $\vv \ne \vv_1, \vv_2$, we have $\varphi_*(\vv) \ne \vv$ and $s_\varphi(P,\vv) = 0$.  
\end{proposition} 

\begin{proof} 
Suppose $P$ is a bi-focused repelling fixed point.  By definition, 
there are type I fixed points $\alpha_1, \alpha_2$ of $\varphi$ belonging to distinct directions in $T_P$,  
so $P \in \Gamma_\Fix$.  
However, $P$ cannot be a branch point of $\Gamma_{\Fix}$, because if it were, 
there would be at least three distinct directions in $T_P$ containing type I fixed points.
It also cannot be a branch point of $\Gamma_{\Fix,\Repel}$ which is not a branch point of $\Gamma_\Fix$, 
because if it were, there would be at least one branch of $\Gamma_{\Fix,\Repel} \backslash \Gamma_{\Fix}$
off $P$.  If $\vv \in T_P$ is the corresponding direction, 
then $B_P(\vv)^-$ would contain a type II repelling fixed point $Q$, but no type I fixed points.
Since $Q$ is a focused repelling fixed point, whose focus $\vv_1 \in T_Q$ points towards $\Gamma_{\Fix}$,  
by Lemma \ref{FocusedP1Lemma} we would have $s_\varphi(P,\vv) > 0$.  However, 
this contradicts Lemma \ref{FirstIdentificationLemma}
since $B_P(\vv)^-$ contains no type I fixed points.  

Conversely, suppose $P \in \Gamma_\Fix$ is a type II repelling fixed point, 
but is not a vertex of $\Gamma_{\Fix,\Repel}$. 
Then there are exactly two directions $\vv_1, \vv_2 \in T_P$ containing type I fixed points, 
so $P$ is a bi-focused repelling fixed point.   

\smallskip
To prove assertions (A) and (B), let $P$ be any bi-focused repelling fixed point.  After a change of coordinates,  
we can assume that $P = \zeta_G$, and that $0$ and  $\infty$ are fixed points of $\varphi$.
Let $(F,G)$ be a normalized representation of $\varphi$.  
Since $\varphi(z)$ fixes $P$, it has nonconstant reduction.  Thus 
there are nonzero homogeneous polynomials 
\begin{equation*} 
\tA(X,Y), \tF_0(X,Y), \tG_0(X,Y) \ \in \ \tk[X,Y]
\end{equation*} 
such that $(\tF,\tG) = (\tA \cdot \tF_0, \tA \cdot \tG_0)$,  with $\GCD(\tF_0,\tG_0) = 1$.
Since $P$ is a repelling fixed point,
we have $\delta := \deg_\varphi(P) \ge 2$, and $\deg(\tF_0) = \deg(\tG_0) = \delta$.  
Write $H(X,Y) = X G(X,Y) - Y F(X,Y)$ 
and put $\tH_0(X,Y) = X \tG_0(X,Y) - Y \tF_0(X,Y)$.  Then $\tH(X,Y) = \tA(X,Y) \cdot \tH_0(X,Y)$. 
Since  $0, \infty \in \PP^1(K)$ are fixed points of $\varphi$, and $P$ is a bi-focused repelling fixed point, 
$\vv_0, \vv_\infty \in T_P$ are the only directions $\vv \in T_P$ 
for which the balls $B_P(\vv)^-$ can contain type I fixed points. 
It follows from Lemma \ref{FirstIdentificationLemma} that $\tH(X,Y) = \tc\cdot X^\ell Y^{d+1 - \ell}$ for some 
$\tc \in \tk^\times$ and some $\ell$.  
Since $\tH_0(X,Y) | \tH(X,Y)$, we must have  
$\tH_0(X,Y) = \th \cdot X^{\ell_0} Y^{\delta + 1  - \ell_0}$ for some  $\th \in \tk^{\times}$ 
and some $0 \le \ell_0 \le \delta+1$.  
Since $\delta + 1 \ge 3$, either $\#\tF_\varphi(P,\vv_0) \ge 2$ or $\#\tF_\varphi(P,\vv_\infty) \ge 2$.
If $i \in \{1,2\}$ is such that $\#\tF_\varphi(P,\vv_i) \ge 2$ then necessarily $\varphi_*(\vv_i) = \vv_i$
and $m_\varphi(P,\vv_i) = 1$.  This proves assertion (A).

For each $\vv \in T_P$ with $\vv \ne \vv_0, \vv_\infty$, there are no type I fixed point in $B_P(\vv)^-$,
so Lemma \ref{FirstIdentificationLemma} shows that $\varphi_*(\vv) \ne \vv$ and $s_\varphi(P,\vv) = 0$.
Thus assertion (B) holds.  
\end{proof} 

Finally, we consider multi-focused repelling fixed points.  Recall that the valence of a point $P$ in a graph $\Gamma$
is the number of edges of $\Gamma$ emanating from $P$. 

\begin{proposition} \label{MultiFocusedProp}
A repelling fixed point of $\varphi$ in $\BHH_K$ is multi-focused if and only if it is a branch point of 
$\Gamma_\Fix$.  If $P$ is a multi-focused repelling fixed point, then its valence in $\Gamma_{\Fix,\Repel}$ 
is the same as its valence in $\Gamma_{\Fix}$. 
\end{proposition}

\noindent{\bf Remark.}  The converse to the second assertion in Proposition \ref{MultiFocusedProp} is false.
There can be branch points of $\Gamma_{\Fix}$ 
which are indifferent fixed points of $\varphi$, and branch points which are moved by $\varphi$.  

\begin{proof}[Proof of Proposition \ref{MultiFocusedProp}.] 
If $P$ is a multi-focused repelling fixed point, then at least three directions in $T_P$ 
contain type I fixed points of $\varphi$, so $P$ is a branch point of $\Gamma_{\Fix}$.  Conversely,
if a repelling fixed point $P$ is a branch point of $\Gamma_{\Fix}$, at least three three directions in $T_P$ 
contain type I fixed points of $\varphi$, so $P$ is multi-focused.

The second assertion can be reformulated as saying that 
if $P$ is a multi-focused repelling fixed point of $\varphi$, then there are no 
branches of $\Gamma_{\Fix,\Repel} \backslash \Gamma_{\Fix}$ which 
fork off $\Gamma_{\Fix}$ at $P$.  Suppose to the contrary that there were such a branch, and let 
$\vv \in T_P$ be the corresponding direction.  By the same argument as in the proof of Proposition \ref{BeadProp},  
$B_P(\vv)^-$ would contain a type II repelling fixed point $Q$, but no type I fixed points.
Since $Q$ is a focused repelling fixed point, whose focus $\vv_1$ points towards $P$,  
by Lemma \ref{FocusedP1Lemma} we would have $s_\varphi(P,\vv) > 0$.  However, 
this contradicts Lemma \ref{FirstIdentificationLemma}
since $B_P(\vv)^-$ contains no type I fixed points. 
\end{proof} 

The fact that focused repelling fixed points are endpoints of $\Gamma_{\Fix,\Repel}$, 
while bi-focused repelling fixed points and multi-focused repelling fixed points belong to $\Gamma_{\Fix}$, 
leads one to ask about the nature of points of $\Gamma_{\Fix,\Repel} \backslash \Gamma_{\Fix}$
which are not endpoints of $\Gamma_{\Fix,\Repel}$: 

\begin{proposition} \label{IdIndiffOffFixProp} 
Suppose $Q$ is a focused repelling fixed point of $\varphi$, and 
let $Q_0$ be the nearest point to $P$ in $\Gamma_{\Fix}$.  
Then each type {\rm II} point in $(Q,Q_0]$
is an id-indifferent fixed point of $\varphi$.
\end{proposition}

\begin{proof}  Let $P$ be a type II point in $(Q,Q_0]$, 
and let $\vv \in T_P$ be the direction such that $Q \in B_P(\vv)^-$.  
The focus $\vv_1$ of $Q$ points towards $\Gamma_{\Fix}$,
and hence towards $P$.  By Lemma \ref{FocusedP1Lemma}, we have $s_\varphi(Q,\vv_P) > 0$.  If $\varphi(P) = P$ 
but $P$ is not id-indifferent, this contradicts Lemma \ref{FirstIdentificationLemma} since $B_P(\vv)^-$
does not contain any type I fixed points of $\varphi$.  If $\varphi(P) \ne P$, it contradicts Lemma 
\ref{SecondIdentificationLemma} for the same reason.
Thus, $P$ must be id-indifferent.
\end{proof}

\section{\bf The Tree Intersection Theorem.} \label{TreeTheoremSection} 

\smallskip 
In this section, we establish several important properties of $\Gamma_{\Fix,\Repel}$.
We first note that it never consists of a single point: 

\begin{lemma} Let $\varphi(z) \in K(z)$ have degree $d \ge 2$.  Then $\varphi$ either has at least 
two type {\rm I} fixed points, or it has one type {\rm I} fixed point 
and at least one type {\rm II} repelling fixed point.  In either case, $\Gamma_{\Fix,\Repel}$ is nontrivial.
\end{lemma} 

\begin{proof}  
If $\varphi$ has at least two type I fixed points, we are done.  If it has only one,
that point is necessarily a fixed point of multiplicity $d+1$, hence has multiplier $1$ and 
is an indifferent fixed point.  By a theorem of Rivera-Letelier (see \cite{R-L1}, Theorem B, or \cite{B-R}, Theorem 10.82), 
$\varphi$ has at least one repelling fixed point in $\BPP_K$  
(which may be in either $\PP^1(K)$ or $\BHH_K$), so in this case 
$\varphi$ must have a repelling fixed point in $\BHH_K$.
\end{proof}

\begin{theorem}[The Tree Intersection Theorem] \label{FixRepelThm} 
Let $\varphi \in K(z)$ have degree $d \ge 2$. 
Then $\Gamma_{\Fix,\Repel}$ is the intersection of the trees 
$\Gamma_{\Fix,\varphi^{-1}(a)}$, for all $a \in \PP^1(K)$. 
\end{theorem} 

\begin{proof}  Let $\Gamma_0$ be the intersection of the trees $\Gamma_{\Fix,\varphi^{-1}(a)}$,  
for all $a \in \PP^1(K)$.  

\smallskip
We first show that $\Gamma_{\Fix,\Repel} \subseteq \Gamma_0$.  For this, 
it is enough to show that each repelling fixed point of $\varphi$ in $\BHH_K$ belongs to $\Gamma_0$,  
since $\Gamma_0$ is connected and clearly contains the type I fixed points. 
Let $P$ be a repelling fixed point of $\varphi$ in $\BHH_K$, and fix $a \in \PP^1(K)$.   
We claim that $P$ belongs to $\Gamma_{\Fix,\,\varphi^{-1}(a)}$. 

By a theorem of Rivera-Letelier (see \cite{R-L3}, Proposition 5.1, or \cite{B-R}, Lemma 10.80), $P$ is of type II.  
By (\cite{B-R}, Corollary 2.13(B)), there is a  
$\gamma \in \GL_2(K)$ for which $\gamma(\infty) = a$
and $\gamma(\zeta_G) = P$.  After conjugating $\varphi$ by $\gamma$, 
we can assume that $a = \infty$ and $P = \zeta_G$. 
Let $(F,G)$ be a normalized representation of $\varphi$.  
The poles of $\varphi$ are the zeros of $G(X,Y)$, 
and the fixed points of $\varphi$ are the zeros of 
$H(X,Y) :=  Y F(X,Y) - X G(X,Y)$.   
Let $\tF, \tG \in \tk[X,Y]$ 
be the reductions of $F$ and $G$. Put $\tA = \GCD\big(\tF, \tG\big)$,  
and write $\tF = \tA \cdot \tF_0$, $\tG = \tA \cdot \tG_0$.  
Then $\tphi$ is the map $(X,Y) \mapsto \big(\tF_0(X,Y), \tG_0(X,Y)\big)$ on $\PP^1(\tk)$. 
Since $P$ is a repelling fixed point of $\varphi$, 
we have $\td := \deg(\tphi) \ge 2$.

Put $\tH_0(X,Y) = Y \tF_0(X,Y) - X \tG_0(X,Y)$, so $\deg(\tH_0) = \td+1$.  
The fixed points of $\tphi$ are the zeros of $\tH_0(X,Y)$ in $\PP^1(\tk)$,  
listed with multiplicities. 
Since $\GCD(\tF_0,\tG_0) = 1$ and $\GCD(\tH_0,\tG_0) = \GCD(Y \tF_0,\tG_0)$, 
we must have $\GCD(\tH_0,\tG_0) = 1$ or $\GCD(\tH_0,\tG_0) = Y$.  

If $\GCD(\tH_0,\tG_0) = 1$, the fixed points and poles of $\tphi$ are disjoint.
Since each fixed point of $\tphi$ is the reduction of at least one fixed point of $\varphi$
(Lemma \ref{FirstIdentificationLemma}), 
and each pole of $\tphi$ is the reduction of at least one pole of $\varphi$, 
we conclude that $\varphi$ has a fixed point $z_0$ and a pole $z_1$ lying in different
directions in $T_P$.  Thus $P$ belongs to $[z_0,z_1]$,
and $P \in \Gamma_{\Fix,\, \varphi^{-1}(a)}$.

If $\GCD(\tH_0,\tG_0) = Y$, then $Y$ divides $\tG_0$, 
so $\widetilde{\infty}$ is a pole of $\tphi$. 
This means $\varphi$ has a pole $z_1$ in the ball $B_{\zeta_G}(\vv_\infty)^- \subset \PP^1(K)$.
On the other hand,  $Y^2$ cannot divide both $\tH_0$ and $\tG_0$.
If $Y^2$ does not divide $\tH_0$, then $\tphi$ has at least one fixed point in $\tk$,
so $\varphi$ has a fixed point $z_0$ in $\cO$.  
Since $z_0, z_1$ lie in different tangent directions at $\zeta_G$,
we conclude that $\zeta_G \in \Gamma_{\Fix,\, \varphi^{-1}(a)}$ in this case.
On the other hand, if $Y^2$ does not divide $\tG_0$, then $\tphi$ has at least one pole 
in $\tk$, so $\varphi$ has a pole $z_0 \in \cO$, and again
$P = \zeta_G \in  \Gamma_{\Fix,\, \varphi^{-1}(a)}$.

\smallskip
We next show that $\Gamma_{\Fix,\Repel} = \Gamma_0$. 
Since $\Gamma_{\Fix,\Repel} \subseteq \Gamma_0$, it will suffice to show that
each endpoint of $\Gamma_{\Fix,\Repel}$ which does not belong to $\Gamma_{\Fix}$  
is an endpoint of the intersection of some set of trees $\{ \Gamma_{\Fix,\varphi^{-1}(a_i)} : i \in I\}$,
for an appropriate index set $I$.

By Proposition \ref{FocusedRepellingProp}, 
each endpoint of $\Gamma_{\Fix,\Repel}$ not in $\Gamma_{\Fix}$ is a focused repelling fixed point of $\varphi$.  
Let $P \in \BHH_K$ be a focused repelling fixed point, with focus $\vv_1$.  
Choose distinct directions $\vv_2, \vv_3 \in T_P \backslash \{\vv_1\}$,  
and take $a_2 \in \PP^1(K) \cap  B_P(\vv_2)^-$, $a_3 \in \PP^1(K) \cap B_P(\vv_3)^-$.  We claim that $P$
is an endpoint of $\Gamma_{\Fix,\varphi^{-1}(a_2)} \cap \Gamma_{\Fix,\varphi^{-1}(a_3)}$.  To see this, 
note that if $\xi_2 \in \PP^1(K) \cap (\BPP_K \backslash B_P(\vv_1)^-)$ is a solution to $\varphi(\xi_2) = a_2$
and $\xi_3 \in \PP^1(K) \cap (\BPP_K \backslash B_P(\vv_1)^-)$ is a solution to $\varphi(\xi_3) = a_3$, 
the directions $\vw_2, \vw_3 \in T_P$ such that $\xi_2 \in B_P(\vw_2)^-$ and $\xi_3 \in B_P(\vw_3)^-$ 
are necessarily distinct.  This follows from Proposition \ref{FocusedRepellingProp}, which asserts that 
for each $\vw \in T_P$ with $\vw \ne \vv_1$, the image $\varphi(B_P(\vw)^-)$ is precisely $B_P(\varphi_*(\vw))^-$. 
\end{proof}

The tree $\Gamma_{\Fix,\Repel}$ is spanned by finitely many points:  

\begin{proposition} \label{FiniteGenerationProp} If $\varphi(z) \in K(z)$ has degree $d \ge 2$, then  
 
 $(A)$ $\varphi$ at most $d$ focused repelling fixed points, and 
 
 $(B)$ $\Gamma_{\Fix,\Repel}$ is a finitely generated tree with at most $2d+1$ endpoints. 
Each endpoint of\, $\Gamma_{\Fix,\Repel}$ is either a type {\rm I} fixed point or a type {\rm II} 
focused repelling fixed point.     
\end{proposition} 

\begin{proof}  If $\varphi(z)$ has no focused repelling fixed points, part (A) holds trivially.  
Otherwise, choose any focused repelling fixed point $P$.  By Proposition \ref{FocusedRepellingProp}, 
it is an endpoint of $\Gamma_{\Fix,\Repel}$;  let $\vv \in T_P$ be a direction pointing away from $\Gamma_{\Fix,\Repel}$.  
Fix a point $\alpha \in \PP^1(K) \cap B_P(\vv)^-$, 
and consider the solutions $\xi_1, \ldots, \xi_d$ to $\varphi(z) = \alpha$.  
By Proposition \ref{FocusedRepellingProp}, for each focused repelling fixed point $Q$ there is a 
direction $\vw_Q \in T_Q$ pointing away from $\Gamma_{\Fix,\Repel}$ such that $\alpha \in \varphi(B_Q(\vw_Q)^-)$;
it follows that some $\xi_i$ belongs to $B_Q(\vw_Q)^-$.  For distinct $Q$, the balls $B_Q(\vw_Q)^-$ are pairwise 
disjoint.  Thus $\varphi$ has at most $d$ focused repelling fixed points.  

Since $\varphi$ has at most $d+1$ type I fixed points, the tree $\Gamma_{\Fix,\Repel}$ is spanned by at most 
$2d+1$ points.  Its endpoints are clearly as claimed.
\end{proof}

Next we note some consequences of Theorem \ref{FixRepelThm}.  For us, the most important is

\begin{proposition} \label{MinResLocTreeThm} Let $\varphi(z) \in K(z)$ have degree $d \ge 2$.
Then $\MinResLoc(\varphi)$ is contained in $\Gamma_{\Fix,\Repel}$. 
\end{proposition} 

\begin{proof}  By Theorem \ref{ResThm}, $\MinResLoc(\varphi)$ is contained in $\Gamma_{\Fix,\varphi^{-1}(a)}$
for each $a \in \PP^1(K)$.  Hence it is contained in the intersection of those trees, which is $\Gamma_{\Fix,\Repel}$.
\end{proof}

Another consequence of Theorem \ref{FixRepelThm} is the  ``Identification Lemma'' 
for id-indifferent fixed points:

\begin{lemma}[Third Identification Lemma] \label{ThirdIdentificationLemma} 
Suppose $P$ is a type {\rm II} id-indifferent fixed point of $\varphi$.
Then for each $\vv \in T_P$ such that $s_\varphi(P,\vv) > 0$, the ball $B_P(\vv)^-$
contains either a type {\rm I} fixed point or a type {\rm II} focused repelling fixed point of $\varphi$.
\end{lemma} 

\begin{proof} Suppose $s_\varphi(P,\vv) > 0$.  This means $\varphi(B_P(\vv)^-) = \BPP_K$.  
If $B_P(\vv)^-$ contains a type I fixed point, we are done.  

If not, then there must be some $\vv_0 \in T_P$
with $\vv_0 \ne \vv$ such that $B_P(\vv_0)^-$ contains a type I fixed point.
Since $\varphi(B_P(\vv)^-) = \BPP_K$, for each $a \in \PP^1(K)$ there is a solution 
to $\varphi(x) = a$ in $B_P(\vv)^-$.  The path from this solution to the type I fixed point passes through
$P$, so $P \in \Gamma_{\Fix,\varphi^{-1}(a)}$.  Letting $a$ vary, we see that 
\begin{equation*}
P \ \in \ \bigcap_{a \in \PP^1(K)} \Gamma_{\Fix,\varphi^{-1}(a)} \  = \ \Gamma_0 \ = \Gamma_{\Fix,\Repel} \ .
\end{equation*}

Since $P$ is id-indifferent, we must have $m_{\varphi}(P,\vv) = 1$.  
Hence by a theorem of Rivera-Letelier (see \cite{R-L1}, \S4, or \cite{B-R}, Theorem 9.46)
there is a  $Q \in B_P(\vv)^-$ such that $\varphi$ maps the annulus $\Ann(P,Q)$ to itself  
and fixes each point in $[P,Q]$.  Take any type II point $Z \in (P,Q)$, 
and let $\vw \in T_Z$ be the direction towards $Q$. We claim that $s_\varphi(Z,\vw) > 0$.  If not, 
$\varphi$ maps $B_Z(\vw)^-$ to a ball, and that ball must be $B_Z(\vw)^-$ since $\varphi(Z) = Z$
and $\varphi_*(\vw) = \vw$.  However, this means that 
\begin{eqnarray*}
\varphi(B_P(\vv)^-) & = & \varphi(\Ann(P,Q) \cup B_Z(\vw)^-)
                    \ = \ \varphi(\Ann(P,Q)) \cup \varphi(B_Z(\vw)^-) \\
                    & = & \Ann(P,Q) \cup B_Z(\vw)^- \ = \ B_P(\vv)^- \ ,
\end{eqnarray*}
which contradicts that $s_\varphi(P,\vv) > 0$.  Hence it must be that $s_\varphi(Z,\vw) > 0$, 
and by the same argument as above, it follows that $Z \in \Gamma_{\Fix,\Repel}$.  

Thus $\Gamma_{\Fix,\Repel}$ contains points in $B_P(\vv)^-$, 
so it has an endpoint in $B_P(\vv)^-$.  
Since $B_P(\vv)^-$ does not contain type I fixed points, 
the endpoint must be a focused repelling fixed point.  
\end{proof} 

As Lemma \ref{ThirdIdentificationLemma} shows, when $P$ is a type II id-indifferent fixed point, 
the linkage between directions $\vv \in T_P$ for which $s_\varphi(P,\vv) > 0$ 
and directions containing a type I fixed point breaks down.  
It turns out that it is the `primary terms' in a normalized representation at $P$ 
which determine when $s_\varphi(P,\vv) > 0$, 
and the `secondary terms' which determine the type I fixed points.  
This is made precise by the following class of examples:
  
\medskip
\noindent{{\bf Example B.} (Maps with an id-indifferent fixed point at $\zeta_G$.)}
\smallskip

Fix $d \ge 2$, and let $\tA(X,Y) \in \tk[X,Y]$ be a nonzero homogeneous form of degree $d-1$. 
Lift $\tA(X,Y)$ to $A(X,Y) \in \cO[X,Y]$, and fix an element $\pi \in \cO$ with $\ord(\pi) > 0$.  
Let $F_1(X,Y), G_1(X,Y) \in \cO[X,Y]$ be arbitrary homogeneous forms of degree $d$.  
Put
\begin{eqnarray*}
F(X,Y) & = & A(X,Y) \cdot X + \pi F_1(X,Y) \ , \\
G(X,Y) & = & A(X,Y) \cdot Y + \pi G_1(X,Y) \ ,
\end{eqnarray*}  
For generic $F_1(X,Y), G_1(X,Y)$ we will have $\GCD(F,G) = 1$;  
if that is the case, let $\varphi$ be the map with normalized representation $(F,G)$.  Then  
\begin{enumerate}
\item $\big(\tF(X,Y),\tG(X,Y)\big) \ = \ \tA(X,Y) \cdot \big(X,Y\big)$, \ while 
\item $H(X,Y) \ = \ X G(X,Y) - Y F(X,Y) \ = \ \pi \big(X G_1(X,Y) - Y F_1(X,Y)\big)$ \ . 
\end{enumerate} 
Thus the directions $\vv \in T_{\zeta_G}$ with $s_\varphi(\zeta_G,\vv) > 0$ come from the roots of $\tA(X,Y)$, 
while the type I fixed points of $\varphi$ are the roots of $X G_1(X,Y) - Y F_1(X,Y)$. 
\hfill $\Box$

\medskip

By combining the three Identification Lemmas, we obtain a new type of fixed point theorem   
for balls which $\varphi$ maps onto $\BPP_K$.  
Previously known fixed point theorems 
(see \cite{B-R}, Theorems 10.83, 10.85, and 10.86) have all concerned domains $D \subset \BPP_K$
whose image $\varphi_*(D)$ is another domain, with $D \subset \varphi_*(D)$ or $\varphi_*(D) \subset D$,
or closed sets $X$ with $\varphi(X) \subseteq X$.  

\begin{theorem}[Full Image Fixed Point Theorem] \label{IndifFixedPtThm} 
Let $\varphi(z) \in K(z)$, with $\deg(\varphi) \ge 2$.
Suppose $P \in \BPP_K$ and $\vv \in T_P$.  If $\varphi(B_P(\vv)^-) = \BPP_K$,
then $B_P(\vv)^-$ contains either a $($classical$)$ fixed point of $\varphi$ in $\PP^1(K)$, 
or a repelling fixed point of $\varphi$ in $\BHH_K$.
\end{theorem}

\begin{proof} 
Given points $P, Q \in \PP^1_\Berk$ with $Q \ne P$, 
write $\Ann(P,Q)$ for the component of $\BPP_K \backslash \{P,Q\}$ containing $(P,Q)$.  

Assume $\varphi(B_P(\vv)^-) = \BPP_K$. 
If $P$ is of type II, then $s_\varphi(P,\vv) > 0$ and 
the result follows by combining Lemmas \ref{FirstIdentificationLemma}, 
\ref{SecondIdentificationLemma} and \ref{ThirdIdentificationLemma}.  
If $P$ is of type III or IV, one 
reduces to the case of type II as follows.  
There is a point $Q \in B_P(\vv)^-$ 
such that $\varphi(\Ann(P,Q)) = \Ann(\varphi(P),\varphi(Q))$.  
Take any type II point  $Z \in (P,Q)$ and let $\vw \in T_Z$ be the direction towards $Q$.  
We claim that $\varphi(B_{Z}(\vw)^-) = \BPP_K$.  Otherwise, $\varphi_*(B_Z(\vw)^-)$ would be the 
ball $B_{\varphi(Z)})(\varphi_*(\vw))^-)$. By the mapping properties of annuli $\varphi_*(\vw)$
is the direction in $T_{\varphi(Z))}$ containing $\varphi(Q)$, so 
\begin{equation*} 
\varphi(B_P(\vv)^-) \ = \ \varphi(\Ann(P,Q)) \cup \varphi(B_Z(\vw)^-) 
                      \ = \ \Ann(\varphi(P),\varphi(Q)) \cup B_{\varphi(Z)}(\varphi_*(\vw))^- 
\end{equation*}  
omits the point $\varphi(P)$, contradicting that $\varphi(B_P(\vv)^-) = \BPP_K$.  

If $P$ is of type I, and $B_P(\vv)^- = \BPP_K \backslash \{P\}$ 
contains no type I fixed points of $\varphi$, then $P$ is the only type I fixed point of $\varphi$. 
Hence it is a fixed point of multiplicity $d+1 > 1$ and necessarily has multiplier $1$.  
By a theorem of Rivera-Letelier (see \cite{R-L1}, Theorem B, or \cite{B-R}, Theorem 10.82) 
$\varphi$ has at least one repelling fixed point in $\BPP_K$, 
so it has a repelling fixed point in $\BHH_K$, which clearly lies in $B_P(\vv)^-$.  
\end{proof}  

The author does not know whether Theorem \ref{IndifFixedPtThm} can be generalized to 
domains $D$ with $\varphi(D) = \BPP_K$, when $D$ has more than one boundary point.  

\section{\bf Slope Formulas.} \label{SlopeFormulaSection}

In this section we establish formulas for the slope of $\ordRes_\varphi(\cdot)$ 
at a point $P \in \BHH_K$, in a direction $\vv \in T_P$.  These formulas will be used 
in the proof of the Weight Formula in Section \ref{WeightFormulaSection}, and in the proof of the balance conditions 
characterizing $\MinResLoc(\varphi)$ in Section \ref{BalanceConditionSection}.

\smallskip 

Let $f(z)$ be a function on $\BHH_K$, and write $\rho(P,Q)$ for the logarithmic path distance.  
Given a point $P \in \BHH_K$, and a direction $\vv \in T_P$, 
the slope of $f$ at $P$ in the direction $\vv$ is defined to be 
\begin{equation} \label{SlopeDef} 
\partial_\vv f(P) \ = \ \lim_{\substack{Q \rightarrow P \\ Q \in B_P(\vv)^-}} \frac{F(Q) - F(P)}{\rho(Q,P)}
\end{equation}
provided the limit exists.   
For any two points $P_1, P_2 \in B_P(\vv)^-$, the paths $[P,P_1]$ and $[P,P_2]$ share a common initial segment, 
so the limit in (\ref{SlopeDef}) exists if and only if it exists for $Q$ restricted to $[P,P_1]$.  Since $\rho(P,Q)$
is invariant under the action of $\GL_2(K)$ on $\BHH_K$, 
$\partial_\vv F(P)$ is independent of the choice of coordinates.  

\smallskip
For the remainder of this section we will take $f(\cdot) = \ordRes_{\varphi}(\cdot)$. 

We first consider slopes for a type I point $Q$.  
Since $\ordRes_\varphi(Q) = \infty$, we ``draw back into $\BHH_K$'' along a path $[Q,Q_1]$ 
and compute the slope at points $P \in (Q,Q_1)$:

\begin{proposition} \label{ClassicalFixedPtSlope}
Let $Q$ be a point of type {\rm I}.  Then there is a point $Q_1 \in \BHH_K$ such that 
for each $P \in (Q,Q_1)$, the slope of $f(\cdot) = \ordRes_\varphi(\cdot)$ at $P$, 
in the direction $\vv_1 \in T_P$ which points towards $Q_1$, is
\begin{equation*}
\partial_{\vv_1}f(P) \ = \ 
          \left\{ \begin{array}{ll} -(d^2 - d) & \text{if \ $Q$ is fixed by $\varphi$,} \\
                                    -(d^2 + d) & \text{if \ $Q$ is not fixed by $\varphi$.}
        \end{array} \right. 
\end{equation*}  
\end{proposition} 

\begin{proof}
After a change of coordinates, we can assume that $Q = 0$.  Since $\ordRes_\varphi(\cdot)$ 
is piecewise affine on paths in $\BHH_K$ (relative to the logarithmic path distance),   
it suffices to prove the result when $P = \zeta_{0,r}$ is a type II point with $r$
sufficiently small, and $\vv_1 \in T_P$ is the direction $\vv_\infty$.  

Let $(F,G)$ be a normalized representation of $\varphi$, 
and write $F(X,Y) = a_d X^d + \cdots a_0 Y^d$, $G(X,Y) = b_d X^d + \cdots + b_0 Y^d$.  
Take $A \in K^{\times}$ and put $r = |A|$\,;  then 
\begin{eqnarray*}
\ordRes_\varphi(\zeta_{0,r}) & = & \ordRes(F,G) \ + \ (d^2+d) \ord(A) \\
     &  & \qquad  - 2d \min\big(\min_{0 \le \ell \le d} \ord(A^\ell a_\ell),\min_{0 \le \ell \le d} \ord(A^{\ell+1} b_\ell) \big) \ .
\end{eqnarray*}  

First suppose that $Q$ is fixed by $\varphi$, so $\varphi(0) = 0$.  
In this situation $a_0 = 0$ and $b_0 \ne 0$, so  
if $r = |A|$ is sufficiently small, then 
$\min\big(\min_{0 \le \ell \le d} \ord(A^\ell a_\ell),\min_{0 \le \ell \le d} \ord(A^{\ell+1} b_\ell) \big)$ 
coincides with 
\begin{equation*}
\min(\ord(A a_1), \ord(A b_0)) \ = \ \ord(A) + \min(\ord(a_1), \ord(b_0)) \ . 
\end{equation*} 
For such $r$ we have 
\begin{equation*}
\ordRes_\varphi(\zeta_{0,r}) \ = \ \ordRes(F,G) \ - 2d  \min(\ord(a_1), \ord(b_0))\ \ + \ (d^2-d) \ord(A) \ , 
\end{equation*}  
and since $\ord(A)$ measures the logarithmic path distance and increases as $r \rightarrow 0$,
the slope of $\ordRes_\varphi(\cdot)$ at $\zeta_{0,r}$ in the direction $\vv_\infty$ is $-(d^2-d)$. 

Next suppose $Q$ is not fixed by $\varphi$, so $\varphi(0) \ne 0$. In this case $a_0 \ne 0$.
If $r = |A|$ is sufficiently small, then 
$\min\big(\min_{0 \le \ell \le d} \ord(A^\ell a_\ell),\min_{0 \le \ell \le d} \ord(A^{\ell+1} b_\ell) \big) \ = \ \ord(a_0)$. 
For such $r$ we have 
\begin{equation*}
\ordRes_\varphi(\zeta_{0,r}) \ = \ \ordRes(F,G) \ - 2d  \, \ord(a_0)  \ + \ (d^2+d) \ord(A) \ , 
\end{equation*}  
so the slope of $\ordRes_\varphi(\cdot)$ at $\zeta_{0,r}$ in the direction $\vv_\infty$ is $-(d^2+d)$. 
\end{proof} 

\begin{definition} \label{FRDefs} 
{\em For a point $P \in \Gamma_{\Fix,\Repel}$, 
we write $T_{P,\FR}$ for the tangent space to $P$ in $\Gamma_{\Fix,\Repel}$, the set of directions $\vv \in T_P$
such that that there is an edge of $\Gamma_{\Fix,\Repel}$ emanating from $P$ in the direction $\vv$.

We  write $v_{\FR}(P)$ for the valence of $P$ in $\Gamma_{\Fix,\Repel}$, the number of edges of $\Gamma_{\Fix,\Repel}$
emanating from $P$.}
\end{definition}  

\begin{definition} \label{ShearingDirectionDef} 
{\em If $P$  is a fixed point  of  $\varphi$  in  $\BPP_K$ $($of any type {\rm I}, {\rm II}, {\rm III}, or {\rm IV}$)$,  
a direction  $\vv  \in  T_P$  will be called a {\rm shearing direction}  if  there is a type {\rm I} 
fixed point in $B_P(\vv)^-$, but  $\varphi_*(\vv) \ne \vv$.  
Let  $N_\Shearing(P)$ be the number of shearing directions in  $T_P$.}   
\end{definition}    

\begin{proposition} \label{NonIdIndiffFixedSlope}
Let $P \in \BHH_K$ be a type {\rm II} fixed point of $\varphi$ which is not id-indifferent.  
Then for each $\vv \in T_P$, the slope of $f(\cdot) = \ordRes_\varphi(\cdot)$ at $P$ in the direction $\vv$ is 
\begin{eqnarray} 
\partial_{\vv}f(P) 
& = & (d^2 \, - \, d) \, - \, 2d \cdot \#F_\varphi(P,\vv) \, + 2d \cdot \max(1,\#\tF_\varphi(P,\vv)\big) \ . \notag \\
& = & (d^2 \, - \, d) - 2d \cdot s_{\varphi}(P,\vv) + 2d \cdot 
\left\{ \begin{array}{ll} 0 & \text{if $\varphi_*(\vv) = \vv$ \ ,} \\
                          1 & \text{if $\varphi_*(\vv) \ne \vv$ \ .} 
        \end{array} \right.  \label{FF1B} 
\end{eqnarray}  
Furthermore, if $P \in \Gamma_{\Fix,\Repel}$, then 
\begin{equation} \label{NonIdFixedLaplacian} 
\sum_{\vv \in T_{P,\FR}} \partial_{\vv}f(P) \ = \ 
(d^2-d) \cdot (v_{\FR}(P) - 2) + 2d \cdot \big(\deg_{\varphi}(P) - 1 + N_{\Shearing}(P) \big) 
\end{equation} 
\end{proposition} 

\begin{proof} 
After a change of coordinates, we can assume that $P = \zeta_G$ and $\vv = \vv_0$.  

Let $(F,G)$  be a normalized representation of $\varphi$; write $F(X,Y) = a_d X^d + \cdots + a_0 Y^d$,
$G(X,Y) = b_d X^d + \cdots + b_0 Y^d$. For each $A \in K^\times$ we have 
\begin{eqnarray*}
\ordRes_\varphi(\zeta_{0,|A|}) - \ordRes_\varphi(\zeta_G) & = &  (d^2+d) \ord(A) \\
& & \quad \ - \ 2d \cdot \min\big(\min_{0 \le i \le d} \ord(A^i a_i),\min_{0 \le j \le d} \ord(A^{j+1} b_j) \big) 
\end{eqnarray*}  
By hypothesis, the reductions $\tF$, $\tG$ are nonzero;  thus
there are indices $i, j$ such that $\ta_i \ne 0$ and $\tb_j \ne 0$, or equivalently that $\ord(a_i) = 0$
and $\ord(b_j) = 0$.  Let $\ell_1$ be least index $i$ for which $\ord(a_i) = 0$ 
and let $\ell_2$ be the least index $j$ for which $\ord(b_j) = 0$.  
If $\ord(A) > 0$ and $\ord(A)$ is sufficiently small, then  
\begin{eqnarray*}
\min\big(\min_{0 \le i \le d} \ord(A^i a_i),\min_{0 \le j \le d} \ord(A^{j+1} b_j) \big) & = & 
\min\big(\ord(A^{\ell_1} a_{\ell_1}), \ord(A^{\ell_2 + 1} b_{\ell_2}) \big) \\ 
& = &  \min(\ell_1, \ell_2 + 1) \cdot \ord(A) \ .  
\end{eqnarray*} 
It follows that the slope of $\ordRes_\varphi(\cdot)$ at $P$ in the direction $\vv = \vv_0$ is 
\begin{equation} \label{RawFormula1}
d^2 + d - 2d \cdot \min(\ell_1, \ell_2 + 1) \ . 
\end{equation} 

We will now reformulate (\ref{RawFormula1}) using dynamical invariants.  
Letting $\tF,\tG \in \tk[X,Y]$ be the reductions of $F, G$, we can factor 
$\tF = \tA \cdot \tF_0$, $\tG = \tA \cdot \tG_0$ where $\tA = \GCD(\tF,\tG)$. 
Write $\ord_X(\tF)$ (resp. $\ord_X(\tG)$) for the power that $X$ occurs as a factor of $\tF$
(resp. $\tG$). Using Faber's theorem (\cite{Fab}, I: Lemma 3.17), it follows that  
\begin{equation*}
s_\varphi(P,\vv_0) \ = \ \ord_X(\tA) 
\ = \  \min\big(\ord_X(\tF), \ord_X(\tG\big)  \ = \ \min(\ell_1,\ell_2) \ .   
\end{equation*} 
On the other hand, since $P$ is not id-indifferent for $\varphi$, by Lemma \ref{FirstIdentificationLemma}  
\begin{equation*}
s_\varphi(P,\vv_0) \ = \ \#F_\varphi(P,\vv_0) - \#\tF_\varphi(P,\vv_0) \ . 
\end{equation*}

Note that $\#\tF_\varphi(P,\vv_0) = 0$ holds if and only if $\ord_X(\tF_0) = 0$,
which in turn holds if and only if $\ell_1 \le \ell_2$.  
Thus, if $\#\tF_\varphi(P,\vv_0) = 0$, then $\ell_1 \le \ell_2$ and  
 $\min(\ell_1, \ell_2+1) = \ell_1 = s_\varphi(P,\vv_0) = \#F_\varphi(P,\vv_0)$.   
On the other hand, if $\#\tF_\varphi(P,\vv_0) > 0$ then $\ell_2 < \ell_1$ and $s_\varphi(P,\vv_0) = \ell_2$,
so $\min(\ell_1, \ell_2 + 1) = \ell_2 + 1 = s_\varphi(P,\vv_0) + 1 = \#F_\varphi(P,\vv_0) - \#\tF_\varphi(P,\vv_0) + 1$.
It follows that  
\begin{eqnarray}
\min(\ell_1, \ell_2 + 1) & = & \left\{ \begin{array}{ll} 
                \#F_\varphi(P,\vv_0) & \text{ if $\#\tF_\varphi(P,\vv_0) = 0$ , } \\
                \#F_\varphi(P,\vv_0) - \#\tF_\varphi(P,\vv_0) + 1 & 
                             \text{ if $\#\tF_\varphi(P,\vv_0) > 0$  } \end{array} \right. \notag \\
               & = &  \#F_\varphi(P,\vv_0) - \max\big(1,\#\tF_\varphi(P,\vv_0)\big) + 1 \ . \label{ell1ell2Formula}
\end{eqnarray}
Inserting (\ref{ell1ell2Formula}) in (\ref{RawFormula1}) yields the first formula in (\ref{FF1B}).
The second formula in (\ref{FF1B}) follows from 
$s_{\varphi}(P,\vv) = \#F_{\varphi}(P,\vv) - \#\tF_{\varphi}(P,\vv)$, since $\varphi_*(\vv) = \vv$ 
if and only if $\#\tF_{\varphi}(P,\vv) > 0$. 

\smallskip
If $P \in \Gamma_{\Fix,\Repel}$, note that by Lemma \ref{FirstIdentificationLemma},
for a direction $\vv \in T_{P,\FR}$ then  $\varphi_*(\vv) \ne \vv$ if and only if $\vv$ is a shearing direction.
To obtain (\ref{NonIdFixedLaplacian}), sum the second formula in (\ref{FF1B}) over all $\vv \in T_{P,\FR}$, getting    
\begin{equation} \label{FF1C} 
\sum_{\vv \in T_{P,\FR}} \partial_{\vv}f(P) 
\ = \ (d^2-d) \cdot v_{\FR}(P) - 2d \cdot \sum_{\vv \in T_{P,\FR}} s_{\varphi}(P,\vv) + 2d \cdot N_{\Shearing}(P) 
\end{equation}    
By Lemma \ref{FirstIdentificationLemma}, $T_{P,\FR}$ contains all $\vv \in T_P$ such that $s_{\varphi}(P,\vv) > 0$. 
It follows that 
\begin{equation*}
\sum_{\vv \in T_{P,\FR}} s_{\varphi}(P,\vv) \ = \ d - \deg_{\varphi}(P) \ . 
\end{equation*} 
Inserting this in (\ref{FF1C}), and doing some algebra, yields (\ref{NonIdFixedLaplacian}). 
\end{proof}  

\begin{proposition} \label{IdIndiffFixedSlope} 
Let $P \in \BHH_K$ be a type {\rm II} id-indifferent fixed point of $\varphi$.  
Then for each $\vv \in T_P$, the slope of $f(\cdot) = \ordRes_\varphi(\cdot)$ at $P$ in the direction $\vv$ is 
\begin{equation} \label{FF2A} 
\partial_{\vv}f(P) \ = \ (d^2 \, - \, d) \, - \, 2d \cdot s_\varphi(P,\vv) \ .
\end{equation}  
Furthermore, if $P \in \Gamma_{\Fix,\Repel}$, then 
\begin{equation}\label{IdFixedLaplacian} 
\sum_{\vv \in T_{P,\FR}} \partial_{\vv}f(P) \ = \ (d^2-d) \cdot (v_{\FR}(P) - 2) 
\end{equation} 
\end{proposition} 

\begin{proof} 
After a change of coordinates, we can assume that $P = \zeta_G$ and $\vv = \vv_0$.  

Let $(F,G)$  be a normalized representation of $\varphi$; write $F(X,Y) = a_d X^d + \cdots + a_0 Y^d$,
$G(X,Y) = b_d X^d + \cdots + b_0 Y^d$. Letting $\ell_1$ (resp. $\ell_2$) be the least index such that $\ord(a_i) = 0$ 
(resp. $\ord(b_j) = 0$), just as in Proposition \ref{NonIdIndiffFixedSlope} one sees that
the slope of $\ordRes_\varphi(\cdot)$ at $P$ in the direction $\vv = \vv_0$ is 
\begin{equation} \label{RawFormula2}
d^2 + d - 2d \cdot \min(\ell_1, \ell_2 + 1) \ . 
\end{equation} 

Since $P$ is id-indifferent for $\varphi$, the reductions $\tF$, $\tG$ are nonzero, and
if $\tA = \GCD(\tF,\tG)$ then $\tF(X,Y) = X \cdot \tA(X,Y)$ and $\tG(X,Y) = Y \cdot \tA(X,Y)$.
Thus $\ell_1 = \ord_X(\tA(X,Y)) + 1$ and $\ell_2 = \ord_X(\tA(X,Y))$.  By Faber's theorem (\cite{Fab}, I: Lemma 3.17),
we have $s_\varphi(P,\vv_0) = \ord_X(\tA)$.  Hence  
\begin{equation} \label{ell1ell2Formula2}
\min(\ell_1, \ell_2 + 1) \ = \ \ord_X(\tA) + 1 \ = \ s_\varphi(P,\vv_0) + 1 \ . 
\end{equation} 
Inserting (\ref{ell1ell2Formula2}) in (\ref{RawFormula2}) yields (\ref{FF2A}).

\smallskip
When $P \in \Gamma_{\Fix,\Repel}$, to obtain (\ref{IdFixedLaplacian}), sum (\ref{FF2A}) over all $\vv \in T_{P,\FR}$, 
getting 
\begin{equation} \label{FF2C} 
\sum_{\vv \in T_{P,\FR}} \partial_{\vv}f(P) 
\ = \ (d^2-d) \cdot v_{\FR}(P) - 2d \cdot \sum_{\vv \in T_{P,\FR}} s_{\varphi}(P,\vv) 
\end{equation}    
By Lemma \ref{ThirdIdentificationLemma}, $T_{P,\FR}$ contains all $\vv \in T_P$ such that $s_{\varphi}(P,\vv) > 0$. 
Since $\deg_{\varphi}(P) = 1$, it follows that 
\begin{equation*}
\sum_{\vv \in T_{P,\FR}} s_{\varphi}(P,\vv) \ = \ d - \deg_{\varphi}(P) \ = \ d - 1 \ . 
\end{equation*} 
Inserting this in (\ref{FF2C}), and simplifying, yields (\ref{IdFixedLaplacian}). 
\end{proof} 

\begin{proposition} \label{NonFixedSlope}
Let $P \in \BHH_K$ be a type {\rm II} point with $\varphi(P) \ne P$.    
Then for each $\vv \in T_P$, the slope of \,$f(\cdot) = \ordRes_\varphi(\cdot)$ at $P$ in the direction $\vv$ is 
\begin{equation} \label{FF3A} 
\partial_{\vv}f(P) \ = \ d^2 \, + \, d \, - \, 2d \cdot \#F_\varphi(P,\vv) \ .
\end{equation} 
Furthermore, if $P \in \Gamma_{\Fix,\Repel}$, then 
\begin{equation}\label{NonFixedLaplacian} 
\sum_{\vv \in T_{P,\FR}} \partial_{\vv}f(P) \ = \ (d^2-d) \cdot (v_{\FR}(P) - 2) \, + \, 2d \cdot (v_{\FR}(P) - 2) \ .
\end{equation}  
\end{proposition} 
 
\begin{proof}  To prove (\ref{FF3A}), we use the machinery from Lemma \ref{SecondIdentificationLemma}.  
As in that lemma, we first choose $\gamma \in \GL_2(K)$ such that $\gamma(\zeta_G) = P$ 
and $\gamma^{-1}(\varphi(P))  = \zeta_{0,r}$, for some $r \in |K^\times|$ with $0 < r < 1$.  
By replacing $\varphi$ with $\varphi^\gamma$ we can assume that $P = \zeta_G$ and $\varphi(P) = \zeta_{0,r}$. 

Let $c \in K^{\times}$ be such that $|c| = r$;  put $\Phi(z) = (1/c) \varphi(z)$.  
Then $\Phi(\zeta_G) = \zeta_G$, so $\Phi$ has nonconstant reduction, and $\varphi(z) = c \cdot \Phi(z)$. 
Let $(F,G)$ be a normalized representation of $\Phi$;  then $(cF,G)$ is a normalized representation of $\varphi$. 
Using this representation, put $H(X,Y) = X G(X,Y) - Y \cdot cF(X,Y)$;  this yields 
$\tH(X,Y) \ = \ X \cdot \tG(X,Y)$. 
On the other hand, if the type I fixed points of $\varphi$ (listed with multiplicities) are 
$(a_i:b_i)$ for $i = 1, \ldots, d+1$, and are normalized so that $\max(|a_i|,|b_i|) = 1$ for each $i$,
there is a constant $C$ with $|C| = 1$ such that $H(X,Y) = C \cdot \prod_{i=1}^{d+1} (b_i X -a_i Y)$.  
Reducing this $\pmod{\fM}$ gives 
\begin{equation} \label{Key1B} 
X \tG(X,Y) \ = \ \tH(X,Y) \ = \ \tC \cdot \prod_{i=1}^{d+1} (\tb_i X - \ta_i Y) \ . 
\end{equation}  

We will now consider what this means for the slope of $\ordRes_\varphi(\cdot)$ at $P$ in a direction $\vv \in T_P$.  
We parametrize the directions $\vv \in T_P$ by points $a \in \PP_1(\tk)$.  We will consider three cases, 
corresponding to the directions $\vv_0$, $\vv_a$ for $0 \ne a \in \tk$, and $\vv_\infty$.

\smallskip
First consider the direction $\vv_0 \in T_P$.  We use the normalized representation $(cF,G)$ for $\varphi$,
expanding $cF(X,Y) = a_d X^d + \cdots + a_0 Y^d$, $G(X,Y) = b_d X^d + \cdots + b_0 Y^d$.  As usual, 
for each $A \in K^{\times}$,  
\begin{eqnarray}
\ordRes_\varphi(\zeta_{0,|A|}) - \ordRes_\varphi(\zeta_G) & = &  (d^2+d) \ord(A) \notag \\
& & \ \ \ - \ 2d \cdot \min\big(\min_{0 \le i \le d} \ord(A^i a_i),\min_{0 \le j \le d} \ord(A^{j+1} b_j) \big) 
\label{ordResFormula} 
\end{eqnarray}  
Let $N = N_0 = \#F_\varphi(P,\vv_0)$
be the number of fixed points of $\varphi$ in $B_P(\vv_0)^-$.  By (\ref{Key1B}), we have  
$X^{N-1} ||\, \tG(X,Y)$, so $\ord(b_\ell) > 0$ for $\ell = 0, \ldots, N-2$ and $\ord(b_{N-1}) = 0$.
Since $|\,c| < 1$ we see that $\ord(a_\ell) > 0$ for all $\ell$.  It follows that if $\ord(A) > 0$ 
and $\ord(A)$ is sufficiently small, then 
\begin{equation*} 
\min\big(\min_{0 \le i \le d} \ord(A^i a_i),\min_{0 \le j \le d} \ord(A^{j+1} b_j) \big) 
\ = \ \ord(A^{(N-1) + 1} b_{N-1}) \ = \ N \cdot \ord(A) \ .
\end{equation*} 
Inserting this in (\ref{ordResFormula}) shows that the slope of $\ordRes_\varphi(\cdot)$ at $P$ in the direction
$\vv_0$ is 
\begin{equation} \label{FF3A0} 
\partial_{\vv_0}f(P) \ = \ d^2 + d - 2d \cdot N_0 \ = \ d^2 + d - 2d \cdot \#F_\varphi(P,\vv_0) \ .
\end{equation} 

\smallskip
Next consider a direction $\vv_a \in T_P$, where $0 \ne a \in \tk$.  
Choose an $\alpha \in \cO$ with $\talpha = a$, and conjugate $\varphi$ by 
\begin{equation*} 
\gamma_\alpha \ = \ \left[ \begin{array}{cc} 1 & \alpha \\ 0 & 1 \end{array} \right] \ .
\end{equation*}
 A normalized representation for $\varphi_\alpha := (\varphi)^{\gamma_\alpha}$ is given by 
\begin{equation*}
\left( \!\! \begin{array}{c} F_\alpha(X,Y) \\ G_\alpha(X,Y) \end{array} \!\! \right)  =  
 \left[ \begin{array}{cc} 1 & -\alpha \\ 0 & 1 \end{array} \right] \! \cdot \!
 \left(\!\! \begin{array}{c} cF \\ G \end{array} \!\! \right) \! \cdot \!
 \left[ \begin{array}{cc} 1 & \alpha \\ 0 & 1 \end{array} \right]   =  
\left( \!\! \begin{array}{c} cF(X+ \alpha Y,Y) - \alpha G(X+\alpha Y,Y) \\ G(X+\alpha Y,Y) \end{array} \!\! \right) \ .
\end{equation*} 
Reducing this gives  $\tF_\alpha(X,Y) = - a \tG(X + aY, Y)$ and $\tG_\alpha(X,Y) = \tG(X + a Y, Y)$. 

Let $N = N_a = \#F_\varphi(P,\vv_a)$ be the number of fixed points of $\varphi$ in $B_P(\vv_a)^-$;
from (\ref{Key1B}) it follows that $(X-aY)^N ||\, \tG(X,Y)$. 
The direction $\vv_\alpha$ for $\varphi$ pulls back to $\vv_0$ for $\varphi_\alpha$,
so $N = \#F_{\varphi_\alpha}(P,\vv_0)$.
Thus $X^N ||\, \tF_\alpha(X,Y)$ and $X^N ||\, \tG_\alpha(X,Y)$.  Expanding  
$F_\alpha(X,Y) = a_d X^d + \cdots + a_0 Y^d$, $G_\alpha(X,Y) = b_d X^d + \cdots + b_0 Y^d$,
we see that $\ord(a_\ell), \ord(b_\ell) > 0$ for $\ell = 0, \ldots, N-1$, while $\ord(a_N) = \ord(b_N) = 0$. 
If $\ord(A) > 0$ is sufficiently small, it follows that 
\begin{equation*} 
\min\big(\min_{0 \le i \le d} \ord(A^i a_i),\min_{0 \le j \le d} \ord(A^{j+1} b_j) \big) 
\ = \ \ord(A^N a_N) \ = \ N \cdot \ord(A) \ .
\end{equation*}  
Inserting this in (\ref{ordResFormula}) shows that the slope of $\ordRes_\varphi(\cdot)$ at $P$ in the direction
$\vv_a$ is 
\begin{equation} \label{FF3Aa} 
\partial_{\vv_a}f(P) \ = \ 
d^2 + d - 2d \cdot N_a \ = \ d^2 + d - 2d \cdot \#F_\varphi(P,\vv_a) \ .
\end{equation} 

\smallskip
Finally, consider the direction $\vv_\infty \in T_P$.  
Conjugate $\varphi$ by 
\begin{equation*} 
\gamma_\infty \ = \ \left[ \begin{array}{cc} 0 & 1 \\ 1 & 0 \end{array} \right] \ .
\end{equation*}
 A normalized representation for $\varphi_\infty := (\varphi)^{\gamma_\infty}$ is given by 
\begin{equation*}
\left( \!\! \begin{array}{c} F_\infty(X,Y) \\ G_\infty(X,Y) \end{array} \!\! \right)  =  
 \left[ \begin{array}{cc} 0 & 1 \\ 1 & 0 \end{array} \right] \! \cdot \!
 \left(\!\! \begin{array}{c} cF \\ G \end{array} \!\! \right) \! \cdot \!
 \left[ \begin{array}{cc} 0 & 1 \\ 1 & 0 \end{array} \right]   =  
\left( \!\! \begin{array}{c} G(Y,X) \\ c F(Y,X) \end{array} \!\! \right) \ .
\end{equation*} 
Reducing this gives  $\tF_\infty(X,Y) = \tG(Y, X)$ and $\tG_\infty(X,Y) = 0$. 

Let $N =  N_\infty = \#F_\varphi(P,\vv_\infty)$ be the number of fixed points of $\varphi$ in $B_P(\vv_\infty)^-$;
by formula (\ref{Key1B}) we have $Y^N ||\, \tG(X,Y)$. 
The direction $\vv_\infty$ for $\varphi$ pulls back to $\vv_0$ for $\varphi_\infty$,
so $N = \#F_{\varphi_\infty}(P,\vv_0)$.
Thus $X^N ||\, \tF_\infty(X,Y)$.  Writing 
$F_\infty(X,Y) = a_d X^d + \cdots + a_0 Y^d$ and $G_\infty(X,Y) = b_d X^d + \cdots + b_0 Y^d$, 
we see that $\ord(a_\ell) > 0$ for $\ell = 0, \ldots, N-1$, while $\ord(a_N) = 0$; 
we have $\ord(b_\ell) > 0$ for all $\ell$.
Thus if $\ord(A) > 0$ is sufficiently small,
\begin{equation*} 
\min\big(\min_{0 \le i \le d} \ord(A^i a_i),\min_{0 \le j \le d} \ord(A^{j+1} b_j) \big) 
\ = \ \ord(A^N a_N) \ = \ N \cdot \ord(A) \ .
\end{equation*}  
Inserting this in (\ref{ordResFormula}) shows that the slope of $\ordRes_\varphi(\cdot)$ at $P$ in the direction
$\vv_\infty$ is 
\begin{equation} \label{FF3Ainfty} 
\partial_{\vv_\infty}f(P) \ = \ 
d^2 + d - 2d \cdot N_\infty \ = \ d^2 + d - 2d \cdot \#F_\varphi(P,\vv_\infty) \ .
\end{equation}

Combining (\ref{FF3A0}), (\ref{FF3Aa}) and (\ref{FF3Ainfty})  yields (\ref{FF3A}). 

\smallskip
If $P \in \Gamma_{\Fix,\Repel}$, to obtain (\ref{IdFixedLaplacian}), 
sum (\ref{FF3A}) over all $\vv \in T_{P,\FR}$, getting 
\begin{equation} \label{FF3C} 
\sum_{\vv \in T_{P,\FR}} \partial_{\vv}f(P) 
\ = \ (d^2+d) \cdot v_{\FR}(P) - 2d \cdot \sum_{\vv \in T_{P,\FR}} \#F_{\varphi}(P,\vv)  
\end{equation}    
By Lemma \ref{SecondIdentificationLemma}, $T_{P,\FR}$ contains all $\vv \in T_P$ such that $\#F_{\varphi}(P,\vv)  > 0$.  
Since $\varphi$ has $d+1$ type I fixed points (counting multiplicities), it follows that
\begin{equation*}
\sum_{\vv \in T_{P,\FR}} \#F_{\varphi}(P,\vv) \ = \ d \, + \, 1 \ . 
\end{equation*} 
Inserting this in (\ref{FF3C}), and doing some algebra, yields (\ref{NonFixedLaplacian}).
\end{proof} 

\section{The Weight Formula and the Crucial Set.} \label{WeightFormulaSection} 

In this section, we compute the Laplacian 
of the restriction of $\ordRes_\varphi(\cdot)$ to the tree $\Gamma_{\Fix,\Repel}$.
This leads to a natural definition of weights $w_\varphi(P)$ for points $P \in \BPP_K$,
and a ``Weight Formula'' which says that the sum of the weights over all $P \in \BPP_K$ is $d-1$.
One consequence of this formula is that  
 $\varphi$ can have at most $d-1$ repelling fixed points in $\BHH_K$.


If $\Gamma \subset \BHH_K$ is a finite graph, let $\CPA(\Gamma)$ be the space of functions on $\Gamma$
which are continuous and piecewise affine (with a finite number of pieces) with respect to the metric on $\Gamma$
induced by the logarithmic path distance.  For each $P \in \Gamma$, we write $T_{P,\Gamma}$ for the tangent 
space to $P$ in $\Gamma$, namely the set of $\vv \in T_P$ such that there is an edge of $\Gamma$
emanating from $P$ in the direction $\vv$.  If the valence of $P$ in $\Gamma$ is $v(P)$, 
then $T_{P,\Gamma}$ has $v(P)$ elements.  
If $f \in \CPA(\Gamma)$, the Laplacian of $F$ (as defined in \cite{B-R0}) is the measure
\begin{equation} \label{LaplacianDef} 
\Delta_{\Gamma}(f) \ = \ \sum_{P \in \Gamma} -\big(\sum_{\vv \in T_{P,\Gamma}} \, \partial_\vv(f)(P)\big)\, \delta_P(z)
\end{equation}
where $\delta_P(z)$ is the Dirac measure at $P$.  Here $\Delta_\Gamma(F)$ is a discrete measure, since the inner 
sum in (\ref{LaplacianDef}) is $0$ at any $P$ which is not a branch point of $\Gamma$ and where $F$ does not change 
slope.  It is well-known (c.f. \cite{B-R0}), and easy to see, that $\Delta_\Gamma(F)$ has total mass $0$.  Indeed, if one 
is given a partition of $\Gamma$ into finitely many segments $[P_i,Q_i]$ (disjoint except for their endpoints),
such that the restriction of $F$ to $[P_i,Q_i]$ is affine for each $i$, then the slopes of $F$ at $P_i$ and $Q_i$, 
in the directions pointing into $[P_i,Q_i]$, are negatives of each other.  

\begin{definition} {(Weights).} \label{WeightDefs}  {\em For each $P \in \BPP_K$, 
 the {\rm weight} $w_\varphi(P)$ is the following non-negative integer$:$ 
 
$(1)$ If $P \in \BHH_K$ and $P$ is fixed by $\varphi$, define  
\begin{equation*} 
              w_\varphi(P) \ =  \ \deg_{\varphi}(P) - 1 + N_\Shearing(P) \ .
\end{equation*}     

$(2)$ If $P \in \BHH_K$ and $P$ is not fixed by $\varphi$, 
let $v(P)$ be the number of directions $\vv \in T_P$ such that $B_P(\vv)^-$ contains a
type {\rm I} fixed point of $\varphi$, and define  
\begin{equation*}
w_\varphi(P) \ =  \ \max(0,v(P)-2) \ .
\end{equation*} 

$(3)$ If $P \in \PP^1(K)$, define  $w_\varphi(P) = 0$.}  
\end{definition} 

\noindent{This definition} of weights is motivated by Propositions \ref{NonIdIndiffFixedSlope}, \ref{IdIndiffFixedSlope}, 
and \ref{NonFixedSlope}.   

\smallskip
We begin by giving explicit formulas for the weights in some cases, 
and characterizing the points with positive weight and the points with weight zero:  

\begin{proposition}[Properties of Weights] \label{WeightProperties} 
Let $P \in \BPP_K$.  If $P$ is a focused repelling fixed point, then $w_\varphi(P) = \deg_{\varphi}(P) - 1$.  
If $P$ is an additively indifferent fixed point in $\Gamma_{\Fix}$, or is a multiplicatively indifferent 
fixed point which is a branch point of $\Gamma_{\Fix}$, then $w_\varphi(P) = N_{\Shearing}(P)$. 
If $P$ is an id-indifferent fixed point, then $w_\varphi(P) = 0$.  If $P$ is a non-fixed branch point of $\Gamma_{\Fix}$,
then $w_\varphi(P) = v(P) - 2$. 
In general 

$(A)$ $w_\varphi(P) > 0$ if and only if  

\qquad $(1)$ $P$ is a type {\rm II} repelling fixed point of $\varphi$, or

\qquad $(2)$ $P$ is an additively indifferent fixed point of $\varphi$ which belongs to $\Gamma_{\Fix}$, or

\qquad $(3)$ $P$ is a multiplicatively indifferent fixed point of $\varphi$ 

\qquad \qquad which is a branch point of $\Gamma_{\Fix}$, or 

\qquad $(4)$ $P$ is a branch point of $\Gamma_{\Fix}$ which is not fixed by $\varphi$.  

$(B)$  $w_\varphi(P) = 0$ if and only if 

\qquad $(1)$ $P \notin \Gamma_{\Fix,\Repel}$, or 

\qquad $(2)$ $P$ is not of type {\rm II}, or  

\qquad $(3)$ $P$ is an additively indifferent fixed point which is not in $\Gamma_{\Fix}$, or 

\qquad $(4)$ $P$ is a multiplicatively indifferent fixed point 

\qquad \qquad which is not a branch point of $\Gamma_{\Fix}$, or 

\qquad $(5)$ $P$ is an id-indifferent fixed point, or  

\qquad $(6)$ $P$ is not fixed by $\varphi$, and is not a branch point of $\Gamma_{\Fix}$.
\end{proposition} 
                                      
\begin{proof} 
If $P$ is a focused repelling fixed point, the unique direction $\vv \in T_P$ such that $B_P(\vv)^-$ contains 
type I fixed points is fixed by $\varphi_*$, so $N_{\Shearing}(P) = 0$ and $w_\varphi(P) = \deg_{\varphi}(P) - 1$.  
If $P$ is an additively indifferent or multiplicatively indifferent fixed point, then $\deg_{\varphi}(P) = 1$
so $w_\varphi(P) = N_{\Shearing}(P)$.  If $P$ is id-indifferent, then $\deg_{\varphi}(P) = 1$ 
and each $\vv \in T_P$ is fixed by $\varphi_*$, so $N_{\Shearing}(P) = 0$; 
thus $w_\varphi(P) = 0$.  If $P$ is a non-fixed branch point of $\Gamma_{\Fix}$, then $v(P) > 2$ so $w_\varphi(P) = v(P) -2$.
 
Clearly each type II repelling fixed point has $w_\varphi(P) > 0$. 
By results of Rivera-Letelier (see \cite{R-L1}, Lemmas 5.3 and 5.4, or \cite{B-R}, Lemma 10.80), 
each fixed point of type III or IV  has degree $1$ and each of its tangent directions
is fixed by $\varphi_*$, hence $w_\varphi(P) = 0$.  
By definition, each point of type I has $w_\varphi(P) = 0$.

Each additively indifferent fixed point $P \notin \Gamma_\Fix$
has type I fixed points in exactly one tangent direction, 
and Lemma \ref{FirstIdentificationLemma} shows that direction is fixed 
by $\varphi_*$;  hence $w_\varphi(P) = 0$.   
On the other hand, each additively indifferent fixed point $P \in \Gamma_\Fix$ 
has type I fixed points in at least two tangent directions, one of which must be a shearing direction  
since an additively indifferent fixed point has exactly 
one fixed direction (with multiplicity $2$);  hence $w_\varphi(P) > 0$.  

Likewise, a multiplicatively indifferent fixed point $P$
has two fixed directions, and Lemma \ref{FirstIdentificationLemma} 
shows each of them must contain type I fixed points;  
thus $P$ belongs to $\Gamma_\Fix$.
If $P$ is not a branch point of $\Gamma_\Fix$, its two fixed directions are the only ones containing type I 
fixed points, so it has no shearing directions, and $w_\varphi(P) = 0$.  If $P$ is a branch point of $\Gamma_\Fix$, 
there is at least one $\vv \in T_P$ containing a type I fixed point besides the two fixed directions, 
and that direction is a shearing direction, so $w_\varphi(P) > 0$. 

If $P \in \BHH_K$ is not fixed by $\varphi$, there are three possibilities.  If $P \notin \Gamma_{\Fix}$,
there is only one direction $\vv \in T_P$ containing type I fixed points, so $v(P) = 1$, giving $w_\varphi(P) = 0$.
If $P \in \Gamma_{\Fix}$ but $P$ is not a branch point of $\Gamma_{\Fix}$, then $v(P) = 2$, giving $w_\varphi(P) = 0$.
If $P$ is a branch point of $\Gamma_{\Fix}$, then $v(P) \ge 3$, so $w_\varphi(P) > 0$. 
\end{proof}

Our main result is

\begin{theorem}[Weight Formula] \label{WeightFormulaTheorem} 
Let $\varphi(z) \in K(z)$ have degree $d \ge 2$.  
Then the following weight formula holds\,$:$
\begin{equation} \label{WeightFormula}
            \sum_{P \in \,\BPP_K}  w_\varphi(P) \ = \ d-1 \ .
\end{equation}
Equivalently, 
\begin{eqnarray} 
 & & \sum_{\substack{\text{\rm focused repelling} \\ \text{\rm fixed points}} }
              \big(\deg_\varphi(P) - 1\big) \ + \ 
    \sum_{\substack{\text{\rm bi-focused and} \\ \text{\rm multi-focused fixed points}}} \!\!\!\!\!\!
              \big(\deg_\varphi(P) - 1 + N_{\Shearing}(P)\big) \notag \\
 & & + \ \sum_{ \substack{ \text{\rm  additively indifferent}  \\
                       \text{\rm fixed points in $\Gamma_\Fix$} } }  N_{\Shearing}(P) 
            \ + \ \sum_{ \substack{ \text{\rm  multiplicatively indifferent} \\
                  \text{\rm fixed branch points of $\Gamma_\Fix$} } }  N_{\Shearing}(P)  \notag \\
 & & \qquad \qquad \qquad  + \ \sum_{\text{\rm non-fixed branch points of $\Gamma_\Fix$}} \big(v(P)-2\big) 
            \ = \ d-1 \ . \label{ExpandedWeightFormula}
\end{eqnarray}
\end{theorem} 

\begin{corollary}\label{RepellingFixedPtCor}
Let $\varphi(z) \in K(z)$ have degree $d \ge 2$. Then 
\begin{equation} \label{RepellingFixedPtFormula}
\sum_{ \substack{ \text{\rm repelling fixed points} \\ \text{$P \in \BHH_K$} }} (\deg_\varphi(P) -1 ) \ \le \ d-1 \ .
\end{equation} 
In particular $\varphi$ can have most $d-1$  repelling fixed points in $\BHH_K$. 
\end{corollary} 

\smallskip 
In Example C of \S\ref{ExamplesSection} 
we will see that for each $d \ge 2$, there are functions $\varphi$ of degree $d$ 
which have $d-1$ repelling fixed points in $\BHH_K$, so Corollary \ref{RepellingFixedPtCor} is sharp.  

\smallskip
Before proving the weight formula, 
we will need an identity relating the number of endpoints of a tree to the valences of 
its internal branch points.  If $\Gamma$ is a finite graph, and $P \in \Gamma$, 
the valence $v_{\Gamma}(P)$ is the number
of edges of $\Gamma$ incident at $P$.

\begin{lemma}  \label{VertexSumFormula} 
Let $\Gamma$ be a finite tree with $D$ endpoints.  Then 
\begin{equation} \label{EulerTreeFormula} 
D \ - \ \sum_{\substack{\text{\rm branch points} \\ P \in \, \Gamma}} (v_{\Gamma}(P)-2)  \ = \ 2 \ .
\end{equation}  
\end{lemma}  

\begin{proof} This follows from Euler's formula $V - E + F = 2$ for planar graphs.  

If $B$ is the number of branch points of $\Gamma$, 
then $V = D + B$.  Each edge has two endpoints, so $E = (\sum_{\text{vertices}} v(P))/2$.  Since $\Gamma$ is a tree, 
if it is embedded as a planar graph then $F = 1$.  Inserting these in Euler's formula yields  
\begin{equation} \label{RevisedEuler}
2D + 2 B - \sum_{\text{\rm endpoints}} v(P) - \sum_{\text{\rm branch points}} v(P) \ = \ 2 \ .
\end{equation}
At each endpoint we have $v(P) = 1$, so $\sum_{\text{\rm endpoints}} v(P) = D$.  There are $B$ branch points, so 
$2B = \sum_{\text{\rm branch points}}  2$.  Combining terms in (\ref{RevisedEuler}) 
gives (\ref{EulerTreeFormula}).

\smallskip
It is also easy to prove Lemma \ref{VertexSumFormula} by induction.  
When $D = 2$, then $\Gamma$ is a segment with no branch points,
and (\ref{EulerTreeFormula}) holds trivially.  Fix $D \ge 2$, and suppose (\ref{EulerTreeFormula}) holds for all trees 
with $D$ endpoints.  A tree $\Gamma$ with $D+1$ endpoints can be gotten by attaching a new edge to a tree $\Gamma_0$
with $D$ endpoints.  If the edge is attached at an existing branch point, the valence of that branch point 
increases by $1$, and $D$ increases by $1$, so (\ref{EulerTreeFormula}) continues to hold.  
If it is attached at an interior 
point of some edge, it creates a new branch point with valence $v(P) = 3$;  
since $v(P) - 2 = 1$ both $D$ and the sum over branch points increase by $1$, 
and again (\ref{EulerTreeFormula}) continues to hold.
\end{proof}

\begin{proof}[Proof of Theorem \ref{WeightFormulaTheorem}]
The idea is to restrict $\ordRes_\varphi(\cdot)$ to $\Gamma_{\Fix,\Repel}$, take its Laplacian,
and simplify.  Since $\Gamma_{\Fix,\Repel}$ has branches of infinite length, in order to apply the theory of 
graph Laplacians for metrized graphs from \cite{B-R0}, 
we cut off the type I endpoints of $\Gamma_{\Fix,\Repel}$, 
obtaining a finite metrized tree $\Gamma_{\hFR}$. 
We then restrict $\ordRes_\varphi(\cdot)$ to $\Gamma_{\hFR}$, and take its Laplacian there.  

For each type I fixed point $\alpha_i$ of $\varphi$, choose a type II 
point $Q_i \in \Gamma_{\Fix,\Repel}$ close enough to $\alpha_i$ that 
\begin{enumerate} 
\item there are no branch points of $\Gamma_{\Fix,\Repel}$ in  $[Q_i,\alpha_i]$, and 
\item at each $P \in [Q_i,\alpha_i)$, the slope of $\ordRes_\varphi(\cdot)$ at $P$  in the (unique) 
direction $\vv \in T_P$ pointing away from $\alpha_i$ in $\Gamma_{\Fix,\Repel}$ is $-(d^2 -d)$.
\end{enumerate}
Since any two paths emanating from $\alpha_i$ share a common initial segment, 
Proposition \ref{ClassicalFixedPtSlope} shows that such $Q_i$ exist.  

Let $\Gamma_{\hFR}$
be the subtree of $\Gamma_{\Fix,\Repel}$ spanned by the focused repelling fixed points 
and the points $Q_i$. Suppose there are $D_1$ distinct type I fixed points 
(ignoring multiplicities) and $D_2$ focused repelling fixed points.  Then 
the number of endpoints of $\Gamma_{\hFR}$ is 
\begin{equation*}
D = \ D_1 + D_2 \ .
\end{equation*}  
Let $f(\cdot)$ be the restriction of $\ordRes_\varphi(\cdot)$ to $\Gamma_{\FR}$.  
Then $f$ belongs to $\CPA(\Gamma_{\hFR})$.

For each  $P \in \Gamma_{\hFR}$, write $T_{P,\hFR}$ for its tangent space in $\Gamma_{\hFR}$  
(the set of $\vv \in T_P$ such that there is an edge of $\Gamma_{\hFR}$ emanating from $P$
in the direction $\vv$).  If $P \in \Gamma_{\hFR} \backslash \{Q_1, \ldots, Q_{D_1}\}$, 
then $T_{P,\hFR} = T_{P,\FR}$ and the valence of $P$ in $\Gamma_{\hFR}$ coincides with $v_{\FR}(P)$.    
If $P \in \{Q_1, \ldots, Q_{D_1}\}$ then 
$T_{P,\hFR}$ consists of the single direction $\vv_{P,1} \in T_P$ pointing into $\Gamma_{\hFR}$.  

For the Laplacian $\Delta_{\hFR}(f)$ we have  
\begin{eqnarray}
\lefteqn{ -\Delta_{\hFR}(f) \ = } &  & \notag \\
& & \quad \sum_{ P \in \,\{Q_1, \ldots, Q_{D_1}\} } \partial_{\vv_{P,1}}f(P) \, \delta_P(z) 
   \ +  \sum_{ P \in \,\Gamma_{\hFR} \backslash  \{Q_1, \ldots, Q_{D_1}\} } 
       \big(\sum_{\vv \in T_{P,\FR}} \partial_{\vv}f(P) \big) \, \delta_P(z) \ . \label{ExpandedLaplacian}
\end{eqnarray} 
By Proposition \ref{ClassicalFixedPtSlope}, if $P \in \,\{Q_1, \ldots, Q_{D_1}\}$ 
then $\partial_{\vv_{P,1}}f(P) = -(d^2-d)$.  
We claim that if $P \in \Gamma_{\hFR} \backslash  \{Q_1, \ldots, Q_{D_1}\}$, then 
\begin{equation} \label{DeltaID} 
\sum_{\vv \in T_{P,\FR}} \partial_{\vv}f(P) \ = \ (d^2-d) \cdot (v_{\FR}(P) - 2) \ + \ 2d \cdot w_\varphi(P) \ . 
\end{equation} 

If $P \in \Gamma_{\hFR} \backslash  \{Q_1, \ldots, Q_{D_1}\}$ is of type II  
and is a repelling fixed point, or is a multiplicatively or additively indifferent fixed point, 
then (\ref{DeltaID}) follows from Proposition \ref{NonIdIndiffFixedSlope} and the definition of $w_\varphi(P)$. If $P$ is 
an id-indifferent fixed point, then (\ref{DeltaID}) follows from Proposition \ref{IdIndiffFixedSlope} since $w_\varphi(P) = 0$
by Proposition \ref{WeightProperties}.

If $P \in \Gamma_{\hFR} \backslash  \{Q_1, \ldots, Q_{D_1}\}$ is a type II point with $\varphi(P) \ne P$, 
then by Propositions \ref{FocusedRepellingProp} and \ref{IdIndiffOffFixProp},
$P$ belongs to $\Gamma_{\Fix}$ and is not a point where a branch of $\Gamma_{\Fix,\Repel} \backslash \Gamma_{\Fix}$
attaches to $\Gamma_{\Fix}$.  
It follows that $v_{\FR}(P)$ (which is the valence of $P$ in $\Gamma_{\hFR}$ and $\Gamma_{\Fix,\Repel}$) 
coincides with $v(P)$ (its valence in $\Gamma_{\Fix}$). 
Hence (\ref{DeltaID}) follows from Proposition  \ref{NonFixedSlope} and the definition of $w_\varphi(P)$.

If $P \in \Gamma_{\hFR} \backslash  \{Q_1, \ldots, Q_{D_1}\}$ is a type III point, then $v_{\FR}(P) = 2$ and $w_\varphi(P) = 0$, 
so the right side of (\ref{DeltaID}) is $0$.  The left side of (\ref{DeltaID}) is also $0$, 
since $\ordRes_{\varphi}(\cdot)$ can change slope only at points of type II (see Theorem \ref{ResThm}). 
Each $P \in \Gamma_{\hFR} \backslash  \{Q_1, \ldots, Q_{D_1}\}$ is either of type II or type III,
so this establishes (\ref{DeltaID}) in all cases.

\smallskip
Since $\Delta_{\hFR}(f)$ has total mass $0$, it follows from (\ref{ExpandedLaplacian}) and (\ref{DeltaID}) that 
\begin{equation} \label{FBAX} 
\ D_1 \cdot (-(d^2-d)) \ + \ 
\sum_{ P \in \,\Gamma_{\hFR} \backslash  \{Q_1, \ldots, Q_{D_1}\} } 
      \big((d^2-d) \cdot (v_{\FR}(P) - 2) \ + \ 2d \cdot w_\varphi(P)\big) \ = \ 0 \ .
\end{equation} 
If $P$ is a focused repelling fixed point, then $v_{\FR}(P) = 1$,
so $v_{\FR}(P) - 2 = -1$.  If $P$ is not an endpoint or a branch point of $\Gamma_{\hFR}$ then $v_{\FR}(P) - 2 = 0$.
Moving the terms in (\ref{FBAX}) involving $(d^2-d)$ to the right side, 
and noting that there are $D_2$ focused repelling fixed points, it follows that 
\begin{equation} \label{FBAX2}   
 \sum_{ P \in \,\Gamma_{\hFR} \backslash  \{Q_1, \ldots, Q_{D_1}\} }  2d \cdot w_\varphi(P) 
\ = \ (d^2 - d) \cdot \Big( D_1 + D_2 - \sum_{\text{branch points of $\Gamma_{\hFR}$}} (v_{\FR}(P) - 2) \Big) \ .
\end{equation}
By Lemma \ref{VertexSumFormula} the sum on the right side of (\ref{FBAX2}) equals $2$.  Dividing through by $2d$ gives
\begin{equation} \label{FBAX3} 
\sum_{ P \in \,\Gamma_{\hFR} \backslash  \{Q_1, \ldots, Q_{D_1}\} } w_\varphi(P) \ = \ d \, - \, 1 \ . 
\end{equation} 

If we let the endpoints points $Q_i$ approach the type I fixed points $\alpha_i$,  
the corresponding graphs $\Gamma_{\hFR}$ exhaust $\Gamma_{\Fix,\Repel} \cap \BHH_K$.  
By Proposition \ref{WeightProperties}, we have $w_\varphi(P) = 0$ for each type I fixed point and  
for each $P \in \BPP_K \backslash \Gamma_{\Fix,\Repel}$.
Thus 
\begin{equation*} 
\sum_{P \in \, \BPP_K} w_\varphi(P) \ = \ d \, - \, 1 \ , 
\end{equation*}
which is (\ref{WeightFormula}).  The expanded form of the weight formula (\ref{ExpandedWeightFormula}) 
follows from (\ref{WeightFormula}) and the formulas for weights in Proposition \ref{WeightProperties}.  
\end{proof}

\begin{definition} \label{CrucialSetDef} 
The set of points in $P \in \BPP_K$ with weight $w_\varphi(P) > 0$ will be called the {\em crucial set} \,$\cCr(\varphi)$,
and the weights $w_\varphi(P)$ will be called the {\em crucial weights}.  
The probability measure 
\begin{equation} \label{CrucialMeasure}
\nu_\varphi \ = \ \frac{1}{d-1} \sum_{P \in \, \BPP_K} w_\varphi(P) \, \delta_P(z) 
\end{equation} 
will be called the {\em crucial measure} of $\varphi$.
\end{definition}

The crucial set consists of the repelling fixed points in $\BHH_K$, 
the indifferent fixed points with a shearing direction, and the branch points of $\Gamma_{\Fix}$ which 
are moved by $\varphi$.  It has at most $d-1$ elements.  
If $\varphi$ has potential good reduction, it consists of a single point.  However, it can also 
consist of a single point in many other ways;
see Example G in \S\ref{ExamplesSection}.

\medskip
If $\Gamma$ is a finite metrized graph, there is another measure of total mass $1$ attached to $\Gamma$, 
its ``Canonical Measure'' $\mu_{\Gamma,\Can}$ (see \cite{B-R0}).  When $\Gamma$ is a tree, 
\begin{equation*}
\mu_{\Gamma,\Can} \ = \ \frac{1}{2} \sum_{P \in \Gamma} (2-v_{\Gamma}(P)) \, \delta_P(z) \ .
\end{equation*} 
Note that $\mu_{\Gamma,\Can}$  gives mass $1/2$ to each endpoint 
of $\Gamma$ and negative mass $1-(v_{\Gamma}(P)/2))$ to each branch point, so it is not in general a probability measure.  
The fact that $\mu_{\Gamma,\Can}$ has total mass $1$ follows from Lemma \ref{VertexSumFormula}. 
When $\Gamma = \Gamma_{\hFR}$, we will write $\mu_{\hFR,\Can}$ for its canonical measure. 
 
Using the measures $\mu_{\hFR,\Can}$ and $\nu_{\varphi}$, 
we can decompose the Laplacian of $\ordRes_{\varphi}(\cdot)$ on $\Gamma_{\hFR}$ into a background part which depends only 
on the branching of $\Gamma_{\hFR}$, and a part which depends on the dynamics of $\varphi$, as follows:  

\begin{corollary} \label{LaplacianDecomp} Let $\varphi(z) \in K(z)$ have degree $d \ge 2$. 
Define $\Gamma_{\hFR}$ as above, and let $f(\cdot)$ be the restriction of $\ordRes_{\varphi}(\cdot)$ 
to $\Gamma_{\hFR}$.    Then 
\begin{equation*}
\Delta_{\hFR}(f) \ = \ 2(d^2-d) \cdot (\mu_{\hFR,\Can} - \nu_{\varphi} ) \ .
\end{equation*} 
\end{corollary} 

\begin{proof} This is a reformulation of (\ref{ExpandedLaplacian}), using (\ref{DeltaID}) 
and the definitions of $\mu_{\hFR,\Can}$ and $\nu_{\varphi}$.  

\end{proof} 

\section{Characterizations of the Minimal Resultant Locus} \label{CharacterizationSection}

In this section, we establish a dynamical characterization and a 
moduli-theoretic characterization of $\MinResLoc(\varphi)$. 

\smallskip
We first give the dynamical characterization.  Before stating it, we need two definitions.  

\begin{definition}[Barycenter] \label{BarycenterDef}
{\em The barycenter of a finite positive measure $\nu$ on $\BPP_K$ is the set of points $Q \in \BPP_K$
such that for each direction $\vv \in T_Q$, at most half the mass of $\nu$ lies in $B_Q(\vv)^-$.}  
\end{definition}

The notion of the barycenter of a measure on $\BPP_K$ is due to Benedetto and Rivera-Letelier,
in unpublished work on the ``$s \log(s)$'' bound for the number of $k$-rational preperiodic points of $\varphi(z) \in k(z)$,
when $k$ is a number field.  

\begin{definition}[The Crucial Tree] \label{CrucialTreeDef}
{\em The crucial tree $\Gamma_\varphi$ is the subtree of $\Gamma_{\Fix,\Repel}$ 
spanned by the crucial set of $\varphi$. We define the {\rm vertices} of $\Gamma_\varphi$ to be  
the points of the crucial set $($whether they are endpoints or interior points of $\Gamma_\varphi)$, 
and the branch points of $\Gamma_\varphi$.  The {\em edges} of $\Gamma_\varphi$ are  
the closed segments between adjacent vertices.    
}
\end{definition}

\begin{theorem}[\rm Dynamical Characterization of $\MinResLoc(\varphi)$] \label{BaryCenterTheorem}  
Let $\varphi(z) \in K(z)$ have degree $d \ge 2$.  
Then $\MinResLoc(\varphi)$ is the barycenter of the crucial measure $\nu_\varphi$.  Equivalently, a point 
$Q \in \BPP_K$ belongs to $\MinResLoc(\varphi)$ if and only if for each $\vw \in T_Q$ 
\begin{equation} \label{WeightBalanceFormula}
\sum_{P \in B_Q(\vw)^-} w_\varphi(P) \ \le \ \frac{d-1}{2} \ .
\end{equation} 
If $d$ is even, $\MinResLoc(\varphi)$ is a vertex 
o\!f the crucial tree $\Gamma_\varphi$.  
If $d$ is odd, $\MinResLoc(\varphi)$ is either a vertex
or an edge o\!f \,$\Gamma_\varphi$.  
\end{theorem} 

\begin{proof} To show that $\MinResLoc(\varphi)$ is the barycenter of $\nu_{\varphi}$, we must show that  
for each $Q \in \BPP_K$, then $Q \in \MinResLoc(\varphi)$ if and only if for each $\vw \in T_Q$,
\begin{equation*}
\nu_\varphi(B_Q(\vw)^-) \ \le \ 1/2 \ .
\end{equation*} 
If $Q \notin \Gamma_{\Fix,\Repel}$ this is trivial:  
$Q \notin \MinResLoc(\varphi)$ since $\MinResLoc(\varphi) \subset \Gamma_{\Fix,\Repel}$, 
while if $\vw \in T_Q$ is the direction towards $\Gamma_{\Fix,\Repel}$ then $\nu_\varphi(B_Q(\vv)^-) = 1$.   
Similar reasoning applies when $Q \in \PP^1(K)$.   

Suppose $Q \in \Gamma_{\Fix,\Repel} \cap \BHH_K$, and write $f(\cdot) = \ordRes_{\varphi}(\cdot)$.  
Then $Q \in \MinResLoc(\varphi)$ if and only if 
$\partial_{\vw}f(Q) \ge 0$ for each $\vw \in T_Q$.    
If $\vw \in T_Q$ points away from $\Gamma_{\Fix,\Repel}$ then 
$\partial_{\vw} f(Q) > 0$ and $\nu_\varphi(B_Q(\vw)^-) = 0$.  
Hence it is enough to show that for each $\vw \in T_{Q,\FR}$, we have $\partial_{\vw} f(Q) \ge 0$
if and only if  $\nu_\varphi(B_Q(\vw)^-) \le 1/2$,  or equivalently that (\ref{WeightBalanceFormula}) holds.  
For this, we use an argument like the one in the proof of 
Theorem \ref{WeightFormulaTheorem}.  

Let $\Gamma$ be the graph consisting of 
the part of $\Gamma_{\Fix,\Repel}$ in $B_Q(\vw)^-$, together with $Q$.  The endpoints of $\Gamma$ 
are $Q$ and the type I fixed points and focused repelling fixed points of $\varphi$ in $B_Q(\vw)^-$.
Let $D_1^{\prime}$ be the number of type I endpoints, and let $D_2^{\prime}$ be the number of focused repelling 
endpoints, so $\Gamma$ has $D^\prime = 1 + D_1^{\prime} + D_2^{\prime}$ endpoints in all.   

Suppose the type I fixed points of $\varphi$ in $B_Q(\vw)^-$ are $\alpha_1, \ldots, \alpha_{D_1^{\prime}}$.
For each $\alpha_i$, choose a type II point $Q_i$ in $\Gamma_{\Fix,\Repel}$ close enough to $\alpha_i$ that 
\begin{enumerate} 
\item there are no branch points of $\Gamma$ in  $[Q_i,\alpha_i]$, 
\item at each $P \in [Q_i,\alpha_i)$, the slope of $\ordRes_\varphi(\cdot)$ at $P$  in the (unique) 
direction $\vv_i \in T_P$ pointing away from $\alpha_i$ in $\Gamma$ is $-(d^2 -d)$, and
\item each $P \in [Q_i,\alpha_i]$ has weight $w_\varphi(P) = 0$.
\end{enumerate}
Since only finitely many points have positive weight, 
Proposition \ref{ClassicalFixedPtSlope} shows that such $Q_i$ exist.  
Let $\hGamma$ be the finite metrized graph gotten by cutting 
the terminal segments $(Q_i,\alpha_i]$ off of $\Gamma$.  

To simplify notation, write $Q_0$ for $Q$, and let $\vv_0 = \vw \in T_Q$. 
Then the Laplacian $\Delta_{\hGamma}(f)$ satisfies   
\begin{eqnarray}
\lefteqn{ -\Delta_{\hGamma}(f) \ = } &  & \notag \\
& & \quad \sum_{ P \in \,\{Q_0,Q_1, \ldots, Q_{D_1^{\prime}}\} } \partial_{\vv_i}f(P) \, \delta_P(z) 
   \ +  \sum_{ P \in \,\hGamma \backslash  \{Q_0,Q_1, \ldots, Q_{D_1^{\prime}}\} } 
       \big(\sum_{\vv \in T_{P,\FR}} \partial_{\vv}f(P) \big) \, \delta_P(z) \ . \label{ExpandedLaplacianII}
\end{eqnarray} 
Put $L = \partial_{\vw}f(Q) = \partial_{\vv_0}f(Q_0)$.  
By Proposition \ref{ClassicalFixedPtSlope}, if $P \in \,\{Q_1, \ldots, Q_{D_1^{\prime}}\}$ 
then $\partial_{\vv_i}f(P) = -(d^2-d)$. By formula (\ref{DeltaID}), if 
$P \in \hGamma \backslash  \{Q_0,Q_1, \ldots, Q_{D_1^{\prime}}\}$, then 
\begin{equation*} 
\sum_{\vv \in T_{P,\FR}} \partial_{\vv}f(P) \ = \ (d^2-d) \cdot (v_{\hGamma}(P) - 2) \ + \ 2d \cdot w_\varphi(P) \ . 
\end{equation*} 
Inserting these values in (\ref{ExpandedLaplacianII}) and using that $\Delta_{\hGamma}(f)$ has total mass $0$, 
we see that 
\begin{equation*} 
0 \ = \ L \ + \ D_1^{\prime} \cdot (-(d^2-d)) +  
     \sum_{ P \in \,\hGamma \backslash  \{Q_0,Q_1, \ldots, Q_{D_1^{\prime}}\} } 
          \big((d^2-d) \cdot (v_{\hGamma}(P) - 2) \ + \ 2d \cdot w_\varphi(P) \big) \ .
\end{equation*} 
If $P \in \hGamma$ is not an endpoint or a branch
point, then $(d^2-d) \cdot (v_{\hGamma}(P) - 2) = 0$. 
There are $D_2^{\prime}$ focused repelling endpoints  of $\hGamma$ in $B_Q(\vw)^-$, 
and for each of them we have $(d^2-d) \cdot (v_{\hGamma}(P) - 2) = -(d^2-d)$.   Hence 
\begin{equation} \label{FG1} 
 L \ = \ 
   \big(D_1^{\prime} + D_2^{\prime} - \sum_{\text{branch points of $\hGamma$}} (v_{\hGamma}(P) - 2)\big) \cdot (d^2-d) 
          \ - \ 2d \cdot \!\!\sum_{P \in \,\hGamma \cap B_P(\vw)^-} w_\varphi(P) \ .  
\end{equation} 
Since $\hGamma$ has $D^{\prime} = 1 + D_1^{\prime} + D_2^{\prime}$ endpoints, Lemma \ref{VertexSumFormula} gives
\begin{equation} \label{FG2}   
D_1^{\prime} + D_2^{\prime} - \sum_{ \text{branch points of $\hGamma$} } (v_{\hGamma}(P) - 2) \ = \ 1 \ .
\end{equation} 

It follows from (\ref{FG1}) and (\ref{FG2}) that $L \ge 0$ if and only if (\ref{WeightBalanceFormula}) holds.  
Thus $Q$ belongs to $\MinResLoc(\varphi)$ if and only if $Q$ is in the barycenter of $\nu_{\varphi}$.

\smallskip 
Clearly $\MinResLoc(\varphi) \subseteq \Gamma_\varphi$. 
To see that $\MinResLoc(\varphi)$ is either a vertex or an edge of $\Gamma_\varphi$,
note that as a point $P$ moves along $\Gamma_\varphi$, the distribution of $\nu_\varphi$-mass 
in the various directions $\vv \in T_P$ can change only when $P$ passes through a vertex.

If $\MinResLoc(\varphi)$ consists of a vertex of $\Gamma_\varphi$, we are done.  
Otherwise, $\MinResLoc(\varphi)$ contains a point $Q$ in the interior of an edge $e$ of $\Gamma_\varphi$.
Since there are precisely two directions $\vv \in T_Q$  for which the balls $B_Q(\vv)^-$ 
can contain $\nu_\varphi$-mass, and since each has mass at most $1/2$, each must have mass exactly $1/2$.  
This continues to hold for all $P$ in the interior of $e$, but it changes when $P$ reaches an endpoint of $e$.

At an endpoint $P_0$ of $e$, for the direction $\vv_0 \in T_{P_0}$ pointing into $e$ 
we still have $\nu_\varphi(B_{P_0}(\vv_0)^-) = 1/2$, so for all $\vv \in T_{P_0}$ with $\vv \ne \vv_0$
we necessarily have $\nu_\varphi(B_P(\vw)^-) \le 1/2$, and $P_0$ belongs to $\MinResLoc(\varphi)$.
If $P_0$ is an endpoint of $\Gamma_\varphi$,
this is all that needs to be said. If $P_0$ is an vertex of $\Gamma_\varphi$ which is not a branch point,
then $P_0$ belongs to the crucial set, so $\nu_{\varphi}(\{P_0\}) > 0$.  
Hence when $P$ moves outside $e$, for the direction 
$\vv \in T_P$ pointing towards $e$ we will have $\nu_\varphi(B_P(\vv)^-) > 1/2$, 
so $P \notin \MinResLoc(\varphi)$.
Finally, if $P_0$ is a branch point of $\Gamma_\varphi$,
there are at least two directions $\vv_1, \vv_2 \in T_{P_0}$ with $\vv_1, \vv_2 \ne \vv_0$
such that $\nu_\varphi(B_{P_0}(\vv_i)^-) > 0$.  Hence when $P$ moves outside $e$, 
for the direction $\vv \in T_P$ pointing towards $e$, we will again have $\nu_\varphi(B_P(\vv)^-) > 1/2$, 
and $P \notin \MinResLoc(\varphi)$.
\end{proof} 

\medskip
We next give the moduli-theoretic characterization of $\MinResLoc(\varphi)$.

The basic theorem in Geometric Invariant Theory (GIT) concerning moduli spaces of rational functions 
is due to Silverman. Let $R$ be a commutative ring with $1$, 
and let $\varphi : \PP^1_R \rightarrow \PP^1_R$ be a morphism of degree $d \ge 2$.  
Fixing homogeneous coordinates on $\PP^1_R$,
the map $\varphi$ corresponds to a pair of homogeneous functions 
$F(X,Y) = f_d X^d + \cdots + f_0 Y^d$, $G(X,Y) = g_d X^d + \cdots + g_0 Y^d$ in $R[X,Y]$ 
such that $\Res(F,G)$ is a unit in $R$. 
Writing  $f = (f_d, \ldots, f_0)$, $g = (g_d, \cdots, g_0)$, let
\begin{equation*}  
Z_\varphi \ = \ Z_{f,g} \ = \ (f_d : \cdots : f_0 : g_d : \cdots : g_0) \ \in \ \PP^{2d+1}(R) 
\end{equation*} 
be the point corresponding to $\varphi$.  
If $\mu = \left( \begin{array}{cc} \alpha & \beta \\ \gamma & \delta \end{array} \right) \in \GL_2(R)$, 
the usual conjugation action of $\mu$ on $\varphi$  defined by 
\begin{eqnarray*}
\varphi^{\mu} & = & \left( \begin{array}{cc} \delta & -\beta \\ -\gamma & \alpha \end{array} \right) 
                          \circ \left( \begin{array}{c} F \\ G \end{array} \right) \circ 
                     \left(  \begin{array}{cc} \alpha & \beta \\ \gamma & \delta \end{array} \right) \\
   & = & \left( \begin{array}{c} \delta F(\alpha X + \beta Y, \gamma X + \delta Y)
                                       - \beta G(\alpha X + \beta Y, \gamma X + \delta Y) \\
                              - \gamma F(\alpha X + \beta Y, \gamma X + \delta Y)
                                       + \alpha G(\alpha X + \beta Y, \gamma X + \delta Y) \end{array} \right)
\end{eqnarray*} 
induces an algebraic action of $\GL_2$ on $\PP^{2d+1}$.  If $R = \Omega$ is an algebraically closed field, the groups $\GL_2(\Omega)$ and $\SL_2(\Omega)$ have the same orbits in $\PP^{2d+1}(\Omega)$.  
For technical reasons, in GIT it is better to work with the 
action of $\SL_2$;  Silverman (\cite{Sil2}, Theorems 1.1 and 1.3) proves

\begin{theorem}[Silverman] \label{ModuliThm} There are open subschemes of $\PP^{2d+1}/\Spec(\ZZ)$
\begin{equation*}
\Rat_d \ \subseteq \ (\PP^{2d+1})^s \ \subseteq \ (\PP^{2d+1})^{ss}
\end{equation*}  
$($the subschemes of rational morphisms of degree $d$, stable points, and semistable points$)$, 
which are invariant under the conjugation action of $\SL_2$, such that the quotients
\begin{equation*}
M_d = \Rat_s/\SL_2,  \qquad M_d^s = (\PP^{2d+1})^s/\SL_2, \quad \text{and} \quad  M_d^{ss} = (\PP^{2d+1})^{ss}/\SL_2
\end{equation*}
exist. $M_d$ is a dense open subset of $M_d^s$ and $M_d^{ss}$. 
$M_d$ and $M_d^s$ are geometric quotients, and $M_d^{ss}$ is a categorical quotient
which is proper and of finite type over $\ZZ$. 
\end{theorem} 

Geometrical and categorical quotients are quotients with certain desirable properties 
(see \cite{Mum} for the definitions).  
The fact that $M_d$ is a geometric quotient includes the fact over any  
algebraically closed field $\Omega$, the $\SL_2(\Omega)$ orbits rational functions of degree $d$ 
are in $1 - 1$ correspondence with the points of $M_d(\Omega)$.  
The spaces $M_d^s$ and $M_d^{ss}$ are called the 
spaces of stable and semi-stable conjugacy classes of rational maps, respectively.  
Loosely, the stable locus $(\PP^{2d+1})^s$ is the largest subscheme such that for any algebraically closed field $\Omega$, 
the $\SL_2(\Omega)$-orbits of points in $(\PP^{2d+1})^s(\Omega)$ are closed and are in $1 - 1$ correspondence 
with the points of $M_d^s(\Omega)$. Again loosely, 
the semi-stable locus $(\PP^{2d+1})^{ss}$ is the largest subscheme for which 
a quotient makes sense:  $\SL_2(\Omega)$-orbits of points in $(\PP^{2d+1})^{ss}(\Omega)$ need not be closed, 
but if one defines two orbits to be equivalent if their closures meet, points of $M_d^{ss}(\Omega)$ 
correspond to equivalence classes of $\SL_2(\Omega)$-orbits.  Each equivalence class of orbits in 
$(\PP^{2d+1})^{ss}(\Omega)$ contains a unique minimal closed orbit. 

\smallskip
Recall that $K$ is a complete, algebraically closed
nonarchimedean valued field with ring of integers $\cO$ and residue field $\tk$:


\begin{definition}[Semi-stable and Stable Reduction] \label{SemiStabilityDef}
{\em Let $\varphi(z) \in K(z)$ have degree $d \ge 2$, and let $(F,G)$ be a normalized
representation of $\varphi$. Writing $F(X,Y) = f_d X^d + \cdots + f_0 Y^d$, $G(X,Y) = g_d X^d + \cdots g_0 Y^d$,
let $Z_\varphi = (f_d: \cdots : f_0 : g_d : \cdots : g_0) \in \PP^{2d+1}(K)$ be the point corresponding to $\varphi$. 
We will say that $\varphi$ has {\rm semi-stable reduction} if   
\begin{equation*} 
\tZ_\varphi \ = \ (\tf_d: \cdots : \tf_0 : \tg_d : \cdots : \tg_0) \ \in \ \PP^{2d+1}(\tk)  
\end{equation*} 
belongs to $(\PP^{2d+1})^{ss}(\tk)$. We will say that $\varphi$ has {\rm stable reduction} 
if $\tZ_\varphi$ belongs to $(\PP^{2d+1})^s(\tk)$.
}
\end{definition}  

Making precise the Hilbert-Mumford numerical criteria, 
Silverman (\cite{Sil2}, Proposition 2.2) gives necessary and sufficient conditions for a point 
to be semi-stable or stable:
 
\begin{proposition}[Silverman] \label{SilCriteria} Let $\SL_2$ act on $\PP^{2d+1}$ 
as above, and suppose $\tZ \in \PP^{2d+1}(\tk)$.  Then 
    
$(A)$ $\tZ$ belongs to $(\PP^{2d+1})^{ss}(\tk)$ if and only if for each $\ttau \in \SL_2(\tk)$, 
when $\tZ^{\ttau}$ is written as $(\ta_d : \cdots : \ta_0 : \tb_d : \cdots : \tb_0)$,   
either there is some $k$ with $(d+1)/2 \le k \le d$ 
such that $\ta_k \ne 0$, or there is some $k$ with $(d-1)/2 \le k \le d$ such that $\tb_k \ne 0$.  

$(B)$ $\tZ$ belongs to $(\PP^{2d+1})^{s}(\tk)$ if and only if for each $\ttau \in \SL_2(\tk)$, 
when $\tZ^{\ttau}$ is written as $(\ta_d : \cdots : \ta_0 : \tb_d : \cdots : \tb_0)$,  
either there is some $k$ with $(d+1)/2 < k \le d$ 
such that $\ta_k \ne 0$, or there is some $k$ with $(d-1)/2 < k \le d$ such that $\tb_k \ne 0$.  
\end{proposition}

\begin{proof}  This is (\cite{Sil2}, Proposition 2.2) in the special case $\Omega = \tk$, 
with two  modifications.
Silverman formulates conditions (A) and (B) as characterizing 
the ``unstable'' and ``not stable'' points of $\PP^{2d+1}(\tk)$.   
Our assertions are the contrapositives of his: by definition, ``semi-stable'' is ``not unstable'' 
and ``stable'' is ``not `not stable' ''.
Second, Silverman writes rational functions as $\varphi(z) = (a_0 z^d + \cdots + a_d)/(b_0 z^d + \cdots + b_d)$, 
indexing coefficients in the opposite order than we do. We have adjusted his coefficient ranges 
to match our notation.
\end{proof} 

The connection between semi-stability and having minimal resultant is due to Szpiro, Tepper, and Williams, 
who proved the implication ``semi-stable reduction $\Rightarrow$ minimal resultant'' in the context 
of  rational functions over a number field or the function field of a curve, 
using a moduli-theoretic argument  
(see (\cite{STW}, Theorem 3.3); their result holds for morphisms $\varphi : \PP^n \rightarrow \PP^n$ 
in arbitrary dimension $n$).

\begin{theorem}[\rm Moduli-Theoretic Characterization of $\MinResLoc(\varphi)$] \label{ModuliTheorem}
Suppose $\varphi(z) \in K(z)$ has degree $d \ge 2$.  Let $P \in \BHH_K$ be a point of type {\rm II}, and let 
$\gamma \in \GL_2(K)$ be such that $P = \gamma(\zeta_G)$.  Then 

$(A)$ $P$ belongs to $\MinResLoc(\varphi)$ if and only if $\varphi^\gamma$ is has semi-stable reduction.

$(B)$ If $P$ belongs to $\MinResLoc(\varphi)$, 
then $P$ is the unique point in $\MinResLoc(\varphi)$ if and only if  
$\varphi^\gamma$ has stable reduction.
\end{theorem} 

\begin{proof}  

We begin by proving part (A). Note that $P \in \MinResLoc(\varphi)$ 
if and only if $\partial_{\vv}\ordRes_\varphi(P) \ge 0$ for each $\vv \in T_P$. 
After replacing $\varphi$ with $\varphi^\gamma$,
we can assume that $P = \zeta_G$.  
Let $Z_\varphi \in \PP^{2d+1}(K)$ be the point corresponding to $\varphi$,
and let $\tZ_{\varphi} \in \PP^{2d+1}(\tk)$ be its reduction.  
We will index directions $\vv \in T_{\zeta_G}$ by points $a \in \PP^1(\tk)$. 

Fix a direction $\vv_a \in T_{\zeta_G}$, and choose $\ttau \in \SL_2(\tk)$ so that $\ttau(\infty) = a$.  
We will show that $\partial_{\vv_a}\ordRes_\varphi(\zeta_G) \ge 0$ if and only if the condition of 
Proposition \ref{SilCriteria}(A) holds for $(\tZ_{\varphi})^{\ttau}$.  

Lift $\ttau$ to $\tau \in \SL_2(\cO)$.  Let $(F,G)$ be a normalized representation of $\varphi^{\tau}$, 
and write $F(X,Y) = a_d X^d + \cdots + a_0 Y^d$, $G(X,Y) = b_d X^d + \cdots + b_0 Y^d$. 
Since the action of $\SL_2$ commutes with reduction, it follows that 
\begin{equation*} 
(\tZ_{\varphi})^{\ttau} \ = \ \widetilde{(Z_{\varphi^{\tau}})} \ = \ (\ta_d : \cdots \ta_0 : \tb_d : \cdots : \tb_0) \ .
\end{equation*} 
Since $\tau_*(\vv_\infty) = \vv_a$, the function $\ordRes_{\varphi}(\cdot)$ is non-decreasing in the direction
$\vv_a$ at $\zeta_G$ if and only if $\ordRes_{\varphi^\tau}(\cdot)$ 
is non-decreasing in the direction $\vv_\infty$.  Take $A \in K^{\times}$ with $|A| \ge 1$.  
By (\cite{RR-MRL}, formula (13)),  we have 
\begin{eqnarray}
\lefteqn{\ordRes_{\varphi^\tau}(\zeta_{0,|A|}) - \ordRes_{\varphi^\tau}(\zeta_G)} \qquad \qquad \qquad & & \notag \\
& = & (d^2 + d) \ord(A) - 2d \cdot \min_{0 \le k \le d} \big(\ord(A^k a_k), \ord(A^{k+1} b_k)\big) \notag \\
& = & \max_{0 \le i \le d} \big( (d^2 + d - 2d k) \ord(A) - 2d \, \ord(a_k), \notag \\ 
& & \qquad \qquad \qquad (d^2 + d - 2d (k+1)) \ord(A) - 2d \, \ord(b_k) \big) \ . \label{fIncrease} 
\end{eqnarray}
The terms in (\ref{fIncrease}) which are non-decreasing as $|A|$ increases are the ones involving $\ord(a_k)$
for $(d+1)/2 \le k \le d$ and the ones involving $\ord(b_k)$ for $(d-1)/2 \le k \le d$.

Since $(F,G)$ is normalized, we have $\ord(a_k) \ge 0$, $\ord(b_k) \ge 0$ for all $k$, and 
there is at least one $k$ for which $\ord(a_k) = 0$ or $\ord(b_k) = 0$.   
Hence the terms  with $\ord(a_k) > 0$ or $\ord(b_k) > 0$ 
cannot be the maximal ones in (\ref{fIncrease}) when $\ord(A)$ is near $0$.  
It follows that $\partial_{\vv_\infty}\ordRes_{\varphi^\tau}(\zeta_G) \ge 0$ if and only if 
\begin{equation*}
\left\{ \begin{array}{l}
           \text{ $\ord(a_k) = 0$ for some $k$ with $(d+1)/2 \le k \le d$, or } \\
           \text{ $\ord(b_k) = 0$ for some $k$ with $(d-1)/2 \le k \le d$.  } 
     \end{array} \right.
\end{equation*}  
Since $\ord(a_k) = 0$ iff $\ta_k \ne 0$, and $\ord(b_k) = 0$ iff $\tb_k \ne 0$,
these are precisely the conditions on $(\tZ_\varphi)^{\ttau}$ from Proposition \ref{SilCriteria}(A).

Since $\tau_*(\vv_\infty)$ runs over all $\vv_a \in T_P$ as $\ttau$ runs over $\SL_2(\tk)$, 
it follows that $P$ belongs to $\MinResLoc(\varphi)$ if and only if $\varphi^\gamma$ has semi-stable reduction. 

\smallskip
Part (B) is proved similarly, using that $P$ is the unique point in $\MinResLoc(\varphi)$ if and only if 
$\partial_{\vv}\ordRes_{\varphi}(P) > 0$ for each $\vv \in T_P$, and that 
the terms in (\ref{fIncrease}) which are increasing as $|A|$ increases are the ones involving $\ord(a_k)$
for $(d+1)/2 < k \le d$ and the ones involving $\ord(b_k)$ for $(d-1)/2 < k \le d$.
\end{proof} 


\section{Supplemental Balance Conditions} \label{BalanceConditionSection} 

Theorem \ref{BaryCenterTheorem} gives necessary and sufficient `balance conditions' 
for a point $P \in \BPP_K$ 
to belong to $\MinResLoc(\varphi)$, in terms of the crucial weights $w_\varphi(P)$.
In this section we will give alternate balance conditions.
They are easier to check, but only apply in certain cases. 
 

\smallskip
When $P$ is a type II fixed point, there are balance conditions for $P$ to belong to $\MinResLoc(\varphi)$ 
in terms of the surplus multiplicities $s_\varphi(P,\vv)$:

\begin{proposition} \label{SvarphiProp} Let $\varphi(z) \in K(z)$ have degree $d \ge 2$, 
and let $P \in \BHH_K$ be of type {\rm II}.  Suppose $\varphi(P) = P$.  Then $P \in \MinResLoc(\varphi)$
if and only if for each $\vv \in T_P$, either 

$(1)$ $s_\varphi(P,\vv) \le \frac{d-1}{2}$, or 

$(2)$ $s_\varphi(P,\vv) \le \frac{d+1}{2}$ and $\varphi_*(\vv) \ne \vv$.

\noindent{Furthermore,} $\MinResLoc(\varphi) = \{P\}$ if and only if $(1)$ and $(2)$ hold with strict inequality.
\end{proposition} 

\begin{proof} Writing $f(\cdot) = \ordRes_\varphi(\cdot)$, we have $P \in \MinResLoc(\varphi)$
if and only if $\partial_\vv f(P) \ge 0$ for each $\vv \in T_P$, 
and $\MinResLoc(\varphi) = \{P\}$ if and only if $\partial_\vv f(P) > 0$ for each $\vv \in T_P$.   
Suppose $\varphi(P) = P$.  
If $P$ is not id-indifferent, then by Proposition \ref{NonIdIndiffFixedSlope}, for each $\vv \in T_P$ 
\begin{equation*}
\partial_\vv f(P) \ = \ (d^2 -d) - 2d \cdot s_{\varphi}(P,\vv) + 2d \cdot 
\left\{ \begin{array}{ll} 0 & \text{if $\varphi_*(\vv) = \vv$ \ ,} \\
                          1 & \text{if $\varphi_*(\vv) \ne \vv$ \ ,} 
        \end{array} \right.  
\end{equation*}
which translates easily into conditions (1) and (2).
\end{proof} 

\begin{corollary} \label{SpecialBalanceConditionsCor} Let $\varphi(z) \in K(z)$ have degree $d \ge 2$,
and let $P \in \BHH_K$ be of type {\rm II}.
If $P$ is id-indifferent, then $P \in \MinResLoc(\varphi)$ if and only if 
$s_\varphi(P,\vv) \le \frac{d-1}{2}$ for each $\vv \in T_P$,
and $\MinResLoc(\varphi) = \{P\}$ if and only if  $s_\varphi(P,\vv) < \frac{d-1}{2}$ for each $\vv \in T_P$.

If $P$ is a focused repelling fixed point,  
then $P \in \MinResLoc(\varphi)$ if and only if $\deg_\varphi(P) \ge \frac{d+1}{2}$,
and $\MinResLoc(\varphi) = \{P\}$ if and only if $\deg_\varphi(P) > \frac{d+1}{2}$.   
\end{corollary}

\begin{proof} 
By Proposition \ref{IdIndiffFixedSlope},  when $P$ is id-indifferent 
$\partial_\vv f(P) = (d^2 - d) - 2d \cdot s_\varphi(P,\vv)$ for each $\vv \in T_P$.
Since $\varphi_*(\vv) = \vv$ for each $\vv \in T_P$, the assertions in the Corollary follow
from Proposition \ref{SvarphiProp}.

When $P$ is a focused repelling fixed point
with focus $\vv_1$, then Proposition \ref{FocusedRepellingProp} gives $s_\varphi(P,\vv_1) = d - \deg_\varphi(P)$, 
while $s_\varphi(P,\vv) = 0$ for each $\vv \in T_P$ with $\vv \ne \vv_1$.  Since $\varphi_*(\vv_1) = \vv_1$,
Proposition \ref{SvarphiProp} shows that $P \in \MinResLoc(\varphi)$ if and only if $\deg_\varphi(P) \ge (d+1)/2$, 
and that $\MinResLoc(\varphi) = \{P\}$ if and only if $\deg_\varphi(P) > (d+1)/2$.        
\end{proof}  

When $P$ is a type II point which is moved by $\varphi$, or is fixed but is not id-indifferent, 
there are  balance conditions for $P$ to belong to $\MinResLoc(\varphi)$ in terms of 
the directional fixed point multiplicities $\#F_\varphi(P,\vv)$ and $\#\tF_\varphi(P,\vv)$:

\begin{proposition} \label{FProp} Let $\varphi(z) \in K(z)$ have degree $d \ge 2$, 
and let $P \in \BHH_K$ be of type {\rm II}.  

If $\varphi(P) \ne P$, then $P \in \MinResLoc(\varphi)$
if and only if  $\#F_\varphi(P,\vv) \le \frac{d+1}{2}$ for each $\vv \in T_P$,
and $\MinResLoc(\varphi) = \{P\}$ if and only if  $\#F_\varphi(P,\vv) <\frac{d+1}{2}$ for each $\vv \in T_P$.

If $\varphi(P) = P$ but $P$ is not id-indifferent, then $P \in \MinResLoc(\varphi)$ if and only if 
for each $\vv \in T_P$, either 

$(1)$ $\#F_\varphi(P,\vv) \le \frac{d+1}{2}$, or 

$(2)$ $\#F_\varphi(P,\vv) > \frac{d+1}{2}$ but $\#F_\varphi(P,\vv) - \#\tF_\varphi(P,\vv) \le \frac{d-1}{2}$. 

\noindent{Furthermore,} $\MinResLoc(\varphi) = \{P\}$ if and only if for each $\vv \in T_P$, either 

$(1^\prime)$ $\#F_\varphi(P,\vv) < \frac{d+1}{2}$, or 

$(2^\prime)$ $\#F_\varphi(P,\vv) \ge \frac{d+1}{2}$ but $\#F_\varphi(P,\vv) - \#\tF_\varphi(P,\vv) < \frac{d-1}{2}$.
\end{proposition}

\noindent{Note} that in Proposition \ref{FProp}, 
since there are exactly $d+1$ fixed points of $\varphi$ (counting multiplicities), 
there can be at most one $\vv \in T_P$ with $\#F_\varphi(P,\vv) > (d+1)/2$.

\begin{proof}  Writing $f(\cdot) = \ordRes_\varphi(\cdot)$, again we have $P \in \MinResLoc(\varphi)$
if and only if $\partial_\vv f(P) \ge 0$ for each $\vv \in T_P$, 
and $\MinResLoc(\varphi) = \{P\}$ if and only if $\partial_\vv f(P) > 0$ for each $\vv \in T_P$.

If $\varphi(P) \ne P$, then for each $\vv \in T_P$ we have 
$\partial_\vv f(P) = (d^2 + d) - 2d \cdot \#F_\varphi(P,\vv)$ by Proposition \ref{NonFixedSlope}, and the assertions 
in the Proposition follow immediately.

If $\varphi(P) = P$ but $P$ is not id-indifferent, then by Proposition \ref{NonIdIndiffFixedSlope} 
for each $\vv \in T_P$ we have
\begin{equation*}
\partial_\vv f(P) \ = \ (d^2-d) - 2d \cdot \#F_\varphi(P,\vv) + 2d \cdot \max\big(1,\#\tF_\varphi(P,\vv)\big) \ .
\end{equation*} 
Thus when $\#\tF_\varphi(P,\vv) \le 1$, we have $\partial_\vv f(P) \ge 0$  if and only 
if $\#F_\varphi(P,\vv) \le (d+1)/2$.  When $\#\tF_\varphi(P,\vv) \ge 2$, we have 
$\partial_\vv f(P) \ge 0$ if and only if $\#F_\varphi(P,\vv) - \#\tF_\varphi(P,\vv) \le (d-1)/2$. However, if 
$\#F_\varphi(P,\vv) \le (d+1)/2$ then $\#F_\varphi(P,\vv) - \#\tF_\varphi(P,\vv) \le (d-1)/2$ is automatic, 
so we only need the more complicated condition $(2)$ when condition $(1)$ fails.  
 In the situation where it is required that $\partial_\vv f(P) > 0$, 
 similar arguments yield $(1^\prime)$ and $(2^\prime)$,
\end{proof} 

{\noindent {\bf Remark.}} In Proposition \ref{IndiffBalanceProp} we will give  
balance conditions for an id-indifferent fixed point $P$ to belong to $\MinResLoc(\varphi)$, 
using directional fixed point multiplicities for points $Q$ in the boundary of 
the component of the `locus of id-indifference' $U_{\id}(P)$.  

\section{\bf Persistence Lemmas and the Locus of Id-Indifference.} \label{PersistenceTheoremSection} 

In this section, we establish three ``Persistence Lemmas'' which shed light on the dynamics of $\varphi$ 
near the part of $\Gamma_{\Fix,\Repel}$ fixed by $\varphi$.  
We show that various reduction behaviors at $P$, including id-indifference
and rotational indifference, propagate to nearby points.  
In the following section we give applications of these lemmas. 

\smallskip
So far we have only defined id-indifference, rotational indifference, and additive indifference for type II 
points.  To formulate these notions for points of types III and IV, we follow a suggestion of 
Xander Faber, using his inclusion map for $\BPP_K$ under base change.  In the following result,
if $L$ is any complete, algebraically closed, nonarchimedean valued field, we use a subscript $L$ to denote
objects ($\BPP_K$, discs, balls, tangent spaces, etc.) associated to $L$.  
Faber (\cite{Fab}, I: Theorem 4.1 and Corollary 4.4) shows

\begin{proposition}[Faber] \label{FaberProp} 
For each extension $L/K$ of complete, algebraically closed, nonarchimedean valued fields
$($with the valuation on $L$ normalized so that it extends the valuation on $K)$, 
there is a canonical inclusion map $\iota_K^L : \BPP_K \rightarrow \BPP_L$
with the following properties.

$(1)$ Let $\{D_K(a_i,r_i)\}_{i \in \NN}$ be a decreasing sequence of discs in $K$,
and let $\{D_L(a_i,r_i)\}_{i \in \NN}$ be the corresponding sequence of discs 
$($with the same centers and radii$)$ in $L$.
Let $\zeta_{a,r,K} \in \BPP_K$ $($resp. $\zeta_{a,r,L} \in \BPP_L)$ be the points
associated to these sequences by Berkovich's classification theorem $($\cite{B-R}, {\rm p.5}$)$. 
Then $\iota_K^L(\zeta_{a,r,K}) = \zeta_{a,r,L}$.  In particular, $\iota_K^L$ extends the 
natural inclusion $\PP^1(K) \hookrightarrow \PP^1(L)$.

$(2)$ If $M/L$ is a further extension of complete, algebraically closed, nonarchimedean valued fields, 
then $\iota_K^M = \iota_L^M \circ \iota_K^L$.

$(3)$ $\iota_K^L$ is continuous for the weak topologies on $\BPP_K$ and $\BPP_L$.
In particular $\iota_K^L(\BPP_K)$ is a compact subset of $\BPP_L$ for the weak topology.

$(4)$ $\iota_K^L$ is an isometry for the logarithmic path distance $\rho(x,y)$.

$(5)$ Write $\iota = \iota_K^L$.  For each $P \in \BPP_K$, 
there is an injective map $\iota_* : T_{P,K} \rightarrow T_{\iota(P),L}$ on the tangent spaces, 
such that $\iota_*(B_{P,K}(\vv)^-) \subset B_{\iota(P),L}(\iota_*(\vv))^-$ for each $\vv \in T_{P,K}$.

$(6)$ If $\varphi_K(z) \in K(z)$ has degree $d \ge 1$ and $\varphi_L(z) \in L(z)$ is given by extension of scalars, 
and if we write $\iota$ for $\iota_K^L$, then 

\quad $(a)$ $\varphi_L \circ \iota = \iota \circ \varphi_K;$

\quad $(b)$ For each $P \in \BPP_K$ we have $\deg_{\iota(P)}(\varphi_L) = \deg_P(\varphi_K);$ 

\quad $(c)$ For each $P \in \BPP_K$ and each $\vv \in T_{P,K}$, 

\qquad \quad we have $m_{\varphi_L}(\iota(P), \iota_*(\vv)) = m_{\varphi_K}(P,\vv)$ and 
$s_{\varphi_L}(\iota(P), \iota_*(\vv)) = s_{\varphi_K}(P,\vv)$.
\end{proposition} 

If $P \in \BPP_K$ is a point of type III or IV, there is an always extension $L/K$ of complete,
algebraically closed valued fields such that $\iota_K^L(P)$ is of type II in $\BPP_L$. 
When $P$ is of type III, such an $L$ can be obtained by taking the completion of the algebraic closure
of the quotient field of the ring $K_r$ constructed by Berkovich in (\cite{Berk}, p.21);  when $P$ is of type IV, 
$L$ can be obtained by taking the completion of the algebraic closure of the field $K(u)$ constructed
by Kaplansky in (\cite{Kapl}, I: Theorem 2, p.306).  

By iterating such constructions and using Zorn's Lemma, one can obtain a maximally complete algebraically closed
nonarchimedean field $L/K$ whose value group is $\RR$ and whose residue field is the same as the residue field $\tk$
of $K$.  Since $\RR$ is divisible and $\tk$ is algebraically closed, such an $L$ satisfies Kaplansky's ``Hypothesis A''
(\cite{Kapl}, I: p.312), so it is unique up to isomorphism by (\cite{Kapl}, I: Theorem 5, p. 312).  For this $L$,
every $P \in \BPP_L$ is either of type I or type II.  

For any $L$ such that $\iota_K^L(P)$ is of type II, $\varphi_L(z)$ has a normalized representation 
$(F_L(X,Y),G_L(X,Y))$ over $\cO_L$, which can be used to define the reduction $\tphi_P(z)$ at $P$.  Since 
any two such fields $L_1, L_2$ can be embedded
in a common field $L_3$ by Kaplansky's results, and since $(F_L,G_L)$ is unique up to scaling 
by a unit in $\cO_L^{\times}$, the 
type of reduction $\varphi$ has at $P$ is independent of the choice of $L$.  

\begin{definition}[Generalized Reduction Types]
{\em Let $\varphi(z) \in K(z)$ have degree $d \ge 1$, and let $P \in \HH_{\Berk,K}$.  Suppose $\varphi(P) = P$. 
Let $L/K$ be an extension of complete, algebraically closed nonarchimedean valued fields such that
$\iota_K^L(P)$ is of type {\rm II} in $\BPP_L$.  
We will call $P$ id-indifferent,
multiplicatively indifferent, additively indifferent, or repelling for $\varphi(z)$,  
according as \,$\iota_K^L(P) \in \BPP_L$ has the corresponding property for $\varphi_L(z)$.

If $P$ is multiplicatively indifferent, and $\iota_K^L(P)$ has reduced rotation number $\tlambda$ for an axis,
we will say that $P$ has reduced rotation number $\tlambda$ for that axis.  
}   
\end{definition} 

\smallskip
In the following, given $K$, 
we write $\cB_{\rho}(P,\varepsilon)^- = \{ z \in \BHH_K : \rho(z,P) < \varepsilon\}$ 
for the ball in the strong topology corresponding to $P \in \BHH_K$ and $\varepsilon \in \RR_{> 0}$. 

\begin{lemma}[First Persistence Lemma] \label{FirstPersistenceLemma}
Let $\varphi(z) \in K(z)$ have degree $d \ge 2$.   
Suppose $P \in \BHH_K$ is id-indifferent for $\varphi$. 
Then there is a ball $\cB_{\rho}(P,\varepsilon)^-$  
such that each $Q \in \cB_{\rho}(P,\varepsilon)^-$ is id-indifferent for $\varphi$. 
\end{lemma}

\begin{proof} 
By extending $K$ and using Proposition \ref{FaberProp} we can assume $P$ is of type II.  
Since the path metric $\rho(x,y)$ is invariant under $\GL_2(K)$, 
after a change of coordinates we can arrange that $P = \zeta_G$.  
Let $(F,G)$ be a normalized representation 
of $\varphi$.  Since $\zeta_G$ is id-indifferent, there is a nonzero homogeneous polynomial $\tA(X,Y) \in \tk[X,Y]$
of degree $d-1$ such that $\tF(X,Y) = \tA(X,Y) \cdot X$ and $\tG(X,Y) = \tA(X,Y) \cdot Y$. 
Let $\hA(X,Y) \in \cO[X,Y]$ be a homogeneous polynomial of degree $d-1$ lifting $\tA(X,Y)$, 
and put $\hF(X,Y) = \hA(X,Y) \cdot X$, $\hG(X,Y) = \hA(X,Y) \cdot Y$.  
Write $F(X,Y) = a_d X^d + \cdots + a_0 Y^d$, $\hF(X,Y) = \ha_d X^d + \cdots + \ha_0 Y^d$,
$G(X,Y) = b_d X^d + \cdots + b_0 Y^d$, $\hG(X,Y) = \hb_d X^d + \cdots + \hb_0 Y^d$. 
Since $(F,G)$ and $(\hF,\hG)$ have the same reductions mod $\fM$, there is an $\eta > 0$ such that
\begin{equation*}
\ord(a_i - \ha_i) \ge \eta, \ \ord(b_i-\hb_i) \ge \eta \quad \text{for $i = 0, \ldots, d$.} 
\end{equation*}
We will abbreviate this by saying $\ord(F-\hF) \ge \eta$, $\ord(G-\hG) \ge \eta$.  

Put $\varepsilon = \eta/(d+1) > 0$. 
To prove the lemma it will suffice to show that each $Q \in \cB_\rho(P,\varepsilon)^-$ is id-indifferent. 
Fix such a $Q$;  after extending $K$ if necessary, we can assume $Q$ is of type II, 
so $Q \in [\alpha,\zeta_G]$ for some $\alpha \in \PP^1(K)$.  

Since $\GL_2(\cO)$ acts transitively on type I points, there is a $\gamma \in \GL_2(\cO)$ with $\gamma(0) = \alpha$.
Write $\gamma = \left(\!\! \begin{array}{cc} a & b \\ c & d \end{array} \!\!\right)$, and  
define $F^\gamma, G^\gamma, \hF^\gamma, \hG^\gamma, \hA^\gamma \in \cO[X,Y]$ by 
\begin{equation*}
\left(\begin{array}{cc} F^\gamma \\ G^\gamma \end{array} \right) \ = \ 
\gamma^{-1} \circ \left(\begin{array}{cc} F \\ G \end{array} \right) \circ \gamma \ , \qquad
\left(\begin{array}{cc} \hF^\gamma \\ \hG^\gamma \end{array} \right) \ = \ 
\gamma^{-1} \circ \left(\begin{array}{cc} \hF \\ \hG \end{array} \right) \circ \gamma \ ,
\end{equation*} 
and $\hA^\gamma(X,Y) = A(aX+bY,cX+dY)$.  
Then $(F^\gamma,G^\gamma)$ is a normalized representation of $\varphi^\gamma$.
It is easy to see that
\begin{equation*}
\hF^\gamma(X,Y) = \hA^\gamma(X,Y) \cdot X \ , \quad \hG^\gamma(X,Y) = \hA^\gamma(X,Y) \cdot Y \ , 
\end{equation*}   
and that $\ord(F^\gamma-\hF^\gamma) \ge \eta$, $\ord(G^\gamma-\hG^\gamma) \ge \eta$.  
Hence after changing coordinates by $\gamma$ and replacing $(F,G)$, $(\hF,\hG)$ and $\hA$ with 
$(F^\gamma,G^\gamma)$, $(\hF^\gamma,\hG^\gamma)$ and $\hA^\gamma$, 
we can assume that $Q \in [0,\zeta_G]$.

Since $Q \in \cB_\rho(\zeta_G,\varepsilon)^- \bigcap [0,\zeta_G]$, 
there is a $\beta \in |K^\times|$ with $0 \le \ord(\beta) < \varepsilon$ such that $Q = \zeta_{0,|\beta|}$.
Put 
\begin{equation*}
\tau \ = \ \left( \begin{array}{cc} \beta & 0 \\ 0 & 1 \end{array} \right) \ \in \ \GL_2(K) \ ,
\end{equation*} 
so $\tau(\zeta_G) = \zeta_{0,|\beta|} = Q$, and let $\hA^\tau(X,Y) = A(\beta X, Y)$.  
Then $F^\tau(X,Y) = \beta^{-1} F(\beta X,Y)$, $G^\tau(X,Y) = G(\beta X,Y)$, 
$\hF^\tau(X,Y) = \hA^\tau(X,Y) \cdot X$, and $\hG^\tau(X,Y) = \hA^\tau(X,Y) \cdot Y$.
 
Write $\delta = \ord(\beta)$. The pair $(F^\tau,G^\tau)$ is not in general normalized, but 
\begin{equation*} 
\ord(F^\tau - \hF^\tau), \ \ord(G^\tau - \hG^\tau) \ \ge \ \eta - \delta  \ > \ 0 \ . 
\end{equation*} 
If we write $\hA(X,Y) = c_{d-1} X^{d-1} + c_{d-2} X^{d-2} Y + \cdots + c_0 Y^{d-1}$, then 
\begin{equation*} 
\hA^\tau(X,Y) \ = \ c_{d-1} \beta^{d-1} X^{d-1} + c_{d-2} \beta^{d-2} X^{d-2} Y + \cdots + c_0 Y^{d-1} \ .
\end{equation*} 
Since each $c_k$ belongs to $\cO$ and at least one $c_k$ belongs to $\cO^\times$, it follows that 
$0 \le \ord(\hA^\tau) \le (d-1) \delta$.  Fix $\mu \in K^\times$ with $|\mu| = \ord(\hA^\tau)$.  
Then $\mu^{-1} \hA^\tau, \mu^{-1} \hF^\tau, \mu^{-1} G^\tau \in \cO[X,Y]$ and 
\begin{equation*} 
\ord(\mu^{-1} \hA^\tau) \ = \  \ord(\mu^{-1} \hF^\tau) \ = \ \ord(\mu^{-1} \hG^\tau) \ = \ 0 \ .
\end{equation*} 
Furthermore 
\begin{equation*} 
\ord(\mu^{-1} F^\tau - \mu^{-1} \hF^\tau), \ \ord(\mu^{-1} G^\tau - \mu^{-1} \hG^\tau) 
\ \ge \ \eta - d \delta \ > \ \eta - d \varepsilon \ = \ \varepsilon \ > \ 0 \ ,
\end{equation*}
so $\mu^{-1} F^\tau, \mu^{-1} G^\tau \in \cO[X,Y]$ and $\ord(\mu^{-1} F^\tau) = \ord(\mu^{-1} G^\tau) = 0$.
Thus  $(\mu^{-1} F^\tau, \mu^{-1}G^\tau)$ is a normalized representation of $\varphi^\tau$, 
and 
\begin{eqnarray*} 
\mu^{-1} F^\tau & \equiv & \mu^{-1} \hF^\tau \ \equiv \ \mu^{-1} \hA^\tau \cdot X  \pmod{\fM_K} \ , \\
\mu^{-1} G^\tau & \equiv & \mu^{-1} \hG^\tau \ \equiv \ \mu^{-1} \hA^\tau \cdot Y  \pmod{\fM_K} \ .
\end{eqnarray*} 
It follows that $Q$ is id-indifferent for $\varphi$. 
\end{proof} 

For future use, note that in Lemma \ref{FirstPersistenceLemma} we have actually proved the following: 

\begin{corollary} \label{FirstPersistenceCor}  Suppose $\zeta_G$ is id-indifferent for $\varphi$, 
and $(F,G)$ is a normalized representation of $\varphi$.  Write $\tF(X,Y) = \tA(X,Y) \cdot X$ and  
$\tG(X,Y) = \tA(X,Y) \cdot Y$, and let $\hA(X,Y) \in \cO[X,Y]$ be a homogeneous lift of $\tA(X,Y)$
of degree $d-1$.  Put $\hF(X,Y) = \hA(X,Y) \cdot X$, $\hG(X,Y) = \hA(X,Y) \cdot Y$, and let 
$\eta = \min(\ord(F-\hF), \ord(G-\hG)) > 0$.  Take $\varepsilon = \eta/(d+1)$.  
Then each $Q \in \cB_\rho(\zeta_G,\varepsilon)^-$ is id-indifferent for $\varphi$.
\end{corollary}

\begin{definition}[The Locus of Id-indifference]
{\em Let $U_{\id}$ be the set of all $P \in \BHH_K$ such that $P$ is id-indifferent for $\varphi$. 
We call $U_{\id}$ the {\rm locus of id-indifference} for $\varphi$.} 
\end{definition} 

\begin{corollary} \label{IdIndiffOpenCor} Let $\varphi(z) \in K(z)$ have degree $d \ge 2$. 
Then $U_{\id} \subset \BHH_K$ is open for the strong topology.  
\end{corollary}

\begin{proof} This is immediate from Lemma \ref{FirstPersistenceLemma}.
\end{proof} 

Given $P \in U_{\id}$,  we write $U_{\id}(P)$ for the connected component of $U_{\id}$ containing $P$. 

\begin{lemma}[Second Persistence Lemma] \label{SecondPersistenceLemma}
Suppose $\varphi(z) \in K(z)$ has degree $d \ge 2$.
Let $P \in \BHH_K$ be a type {\rm II} fixed point of $\varphi$ which is not id-indifferent, 
and let $\vv \in T_P$ be a direction 
with $\varphi_*(\vv) = \vv$. Indexing directions in $T_P$ by points of $\PP^1(\tk)$,
write $\vv = \vv_a;$ then $\tphi(a) = a$. Let $\tlambda \in \tk$ be the multiplier of $\tphi$ at $a$. Then 

$(A)$ $\tlambda = 0$ 
       if and only if $m_\varphi(P,\vv) > 1$.  
       In this case $\#\tF_\varphi(P,\vv) = 1$.

$(B)$ $\tlambda = 1$  
       if and only if $m_\varphi(P,\vv) = 1$ and there is a segment $(P,P_0) \subset B_P(\vv)^-$ 
       such that each \,$Q \in (P,P_0)$ is id-indifferent.  
       In this case $\#\tF_\varphi(P,\vv) \ge 2$.  
       
$(C)$ $\tlambda \in \tk^{\times}$ with $\tlambda \ne 1$ 
       if and only if $m_\varphi(P,\vv) = 1$ and there is a segment $(P,P_0) \subset B_P(\vv)^-$ 
       such that each $Q \in (P,P_0)$ is multiplicatively indifferent, 
       with reduced rotation number $\tlambda$ for the axis $(P,P_0)$.  
       In this case $\#\tF_\varphi(P,\vv) = 1$.    
\end{lemma} 

\begin{proof}
The proof is similar to that of Lemma \ref{FirstPersistenceLemma}. 

After a change of coordinates, we can assume that $P = \zeta_G$ and that $\vv = \vv_0$.  Let $(F,G)$
be a normalized representation of $\varphi$.  
By hypothesis, there are nonzero homogeneous polynomials $\tA(X,Y), \tF_0(X,Y), \tG_0(X,Y) \in \tk[X,Y]$, 
with $\GCD(\tF_0,\tG_0) = 1$ and $D := \deg(\tF_0) = \deg(\tG_0) = \deg_\varphi(P) \ge 2$,  
such that $\tF = \tA \cdot \tF_0$ and $\tG = \tA \cdot \tG_0$.  Since $\varphi_*(\vv_0) = \vv_0$, 
if we write $\tF_0 = \tf_D X^D + \cdots + \tf_1 X Y^{D-1} + \tf_0 Y^D$ and
$\tG_0 = \tg_D X^D + \cdots + \tg_0 Y^D$, then 
$\tf_0 = 0$ and $\tg_0 \ne 0$.  After scaling $F$ and $G$ by a common unit, we can assume that $\tg_0 = 1$.
In this situation, $\tf_1 = \tlambda$ is the multiplier of $\tphi$ at $z = 0$. 

By definition, $m_\varphi(P,\vv)$ is the multiplicity of $z = 0$ as a root of $\tphi$, 
so $m_\varphi(P,\vv) > 0$ if and only if $\tlambda = 0$.  If $\tlambda = 0$, then $\#\tF_\varphi(P,\vv) = 1$
since a fixed point of $\tphi$ can have multiplicity $\ge 2$ only if its multiplier is $1$. 

Henceforth assume $\tlambda \ne 0$.  
Lift $\tA(X,Y)$, $\tF_0(X,Y)$, and $\tG_0(X,Y)$ to homogeneous polynomials 
$\hA(X,Y)$, $\hF_0(X,Y)$, and $\hG_0(X,Y)$ in $\cO[X,Y]$, 
and write $\hA(X,Y) = \ha_m X^m + \cdots + \ha_1 X Y^{m-1} + \ha_0 Y^m$, 
$\hF_0(X,Y) = \hf_D X^D + \cdots + \hf_1 X Y^{D-1} + \hf_0 Y^D$, 
and $\hG_0(X,Y) = \hg_D X^D + \cdots + + \hg_1 X Y^{D-1} + \hg_0 Y^D$, 
where $m + D = d = \deg(\varphi)$.  Without loss, we can assume that $\hf_0 = 0$ and $\hg_0 = 1$.  
By hypothesis, at least one of the $\ha_i$ is a unit in $\cO$; $\hf_1$ is a unit since $\tlambda = 0$,
and clearly $\hg_0$ is a unit.  
Put $\hF = \hA \cdot \hF_0$, $\hG = \hA \cdot \hG_0$. 
Then $F \equiv \hF \pmod{\fM}$ and $G \equiv \hG \pmod{\fM}$;  let  
\begin{equation*}
\eta \ = \ \min\big(\ord(F - \hF), \ord(G-\hG)\big) \ > \ 0 \ . 
\end{equation*}

Given $Q \in (0,\zeta_G)$, we have $Q = \zeta_{0,r}$ for some $0 < r < 1$.  After enlarging $K$ 
we can assume that $r \in |K^{\times}|$.  
Take $t \in K$ with $r = |t|$ 
and put $\gamma = \left( \begin{array}{cc} t & 0 \\ 0 & 1 \end{array} \right) \in \GL_2(K)$;
then $\gamma(\zeta_G) = Q$. 
Let $F^t(X,Y) = t^{-1}F(tX,Y)$, $G^t(X,Y) = G(tX,Y)$; 
then $(F^t,G^t)$ is a representation (not in general normalized) of $\varphi$ at $Q$
(that is, a representation of $\varphi^{\gamma}$).  Let 
\begin{equation*} 
\hA^t(X,Y) = \hA(tX,Y), \quad 
\hF_0^t(X,Y) = t^{-1} \hF_0(tX,Y), \quad \hG_0^t(X,Y) \ = \ \hG(tX,Y) \ . 
\end{equation*}

If $Q$ is close enough to $\zeta_G$, we can use $\hA$, $\hF_0$ and $\hG_0$ to obtain 
a normalized representation for $\varphi$ at $Q$. 
Put $\delta = \ord(t)$, and assume $0 < (d+1) \delta <  \eta$.  
Then $ F^t, G^t, \hA^t, \hF_0^t, \hG_0^t \in \cO[X,Y]$ and 
\begin{equation*}
\min\big(\ord(F^t - (\hA^t \cdot \hF_0^t)), \ord(G^t-(\hA^t \cdot \hG_0^t))\big) \ \ge \ \eta - \delta \ > \ 0 \ .
\end{equation*} 
Since $\hA^t(X,Y) = \ha_m t^m X^m + \cdots + \ha_1 t XY^{m-1} + \ha_0 Y^m$ 
and at least one $\ha_i$ is a unit, it follows that
$\ord(\hA^t) \le m \cdot \eta \le d \cdot \eta$.  
Since $\hf_1$ and $\hg_0$ are units in $\cO$ and $\hf_0 = 0$, 
we have 
$\ord(\hF^t) = \ord(\hG^t) = 0$. Let $\beta \in K^\times$ satisfy $\ord(\beta) = \ord(\hA^t)$, 
and put $\hA_t = \beta^{-1} \cdot \hA^t$, $F_t = \beta^{-1} \cdot F^t$, and $G_t = \beta^{-1} \cdot G^t$.  
Then 
\begin{equation} \label{BFNG1}
\min\big(\ord(F_t - (\hA_t \cdot \hF_0^t)), \ord(G_t-(\hA_t \cdot \hG_0^t))\big) 
\ \ge \ \eta - (d+1) \delta \ > \ 0 \ .
\end{equation} 
By construction, $\ord(\hA_t) = 0$.  Since $\hf_0 = 0$ and $\hf_1$ is a unit, we have  
$\ord(\hF_0^t) = 0$; since $\ord(\hg_0) = 0$ we have $\ord(\hG_0^t) = 0$.  By Gauss's lemma,
$\ord(\hA_t \cdot \hF_0) = \ord(\hA_t \cdot \hG_0) = 0$.  It follows from (\ref{BFNG1}) that 
$\ord(F_t) = \ord(G_t) = 0$, so  $(F_t,G_t)$ is a normalized representation for $\varphi$ at $Q$.

Recall that  $\hf_0 = 0$, $\hg_0 = 1$, and $\hf_1 \equiv \tlambda \pmod{\fM}$.  Since $\ord(t) > 0$,  we have 
\begin{eqnarray*}
\hF_0^t(X,Y) & = & \hf_D t^{D-1} X^D + \cdots + \hf_1 X Y^{D-1} 
     \qquad \quad \equiv \ \tlambda X Y^{D-1} \pmod{\fM} \ , \\
\hG_0^t(X,Y) & = & \hg_D t^{D} X^D + \cdots + \hg_1 t X Y^{D-1} + Y^D 
      \ \, \equiv \ Y^{D} \pmod{\fM} \ .
\end{eqnarray*} 
Letting $\tA_t(X,Y) \in \tk[X,Y]$ be the reduction of $\hA_t \pmod{\fM}$, it follows from (\ref{BFNG1}) that 
\begin{eqnarray*}
F_t(X,Y) & \equiv & \big( \tA_t(X,Y) \cdot Y^{D-1} \big) \cdot \tlambda X \pmod{\fM} \ , \\ 
G_t(X,Y) & \equiv & \big( \tA_t(X,Y) \cdot Y^{D-1} \big) \cdot Y \pmod{\fM} \ . 
\end{eqnarray*}   
Thus the reduction of $\varphi$ at $Q$ 
(or equivalently of $\varphi^{\gamma}$ at $\zeta_G$) is $\tphi_P(z) = \tlambda z$.  

\smallskip
Fix $t_0$ with $\ord(t_0) = \eta/(d+1)$, and put $P_0 = \zeta_{0,|t_0|}$;
recall that $P = \zeta_G$ by our initial reductions.   

If $\tlambda = 1$, then $\varphi$ has id-indifferent reduction at each $Q \in (P,P_0)$. 
Our assumption that
$P$ is not id-indifferent for $\varphi$ means that $\tH_0(X,Y) = X \tG_0(X,Y) - Y\tF_0(X,Y) \ne 0$,
and our assumptions that $\tf_0 = 0$, $\tf_1 = \tlambda = 1$, and $\tg_0 = 1$ mean that $X^2 | \tH_0(X,Y)$.  
Thus $\#\tF_\varphi(P,\vv) \ge 2$.

If $\tlambda \in \tk^\times$ but $\tlambda \ne 1$,  
then $\varphi$ has multiplicatively indifferent reduction, 
with reduced rotation number $\tlambda$ for the axis $(P,P_0)$, at each $Q \in (P,P_0)$.    
At $P$, since $\tf_1 \ne 1$ the reduction $\tphi_P$ is not tangent to the identity 
at $z=0$, and so $\#\tF_\varphi(P,\vv) = 1$. 
\end{proof} 

\begin{corollary} \label{FocusedBoundaryCor}
Let $\varphi(z) \in K(z)$ have degree $d \ge 2$.  Then each focused or bi-focused repelling 
fixed point of $\varphi$ is a boundary point of $U_{\id}$.  
\end{corollary} 

\begin{proof}  By Propositions \ref{FocusedRepellingProp} and \ref{BeadProp}, if $P$ is a focused
or bi-focused repelling fixed point of $\varphi$, there is a direction $\vv \in T_P$ 
such that $m_\varphi(P,\vv) = 1$ and $\#\tF_\varphi(P,\vv) \ge 2$. 
By Lemma \ref{SecondPersistenceLemma} this can happen if and only if there is a segment $(P,P_0) \subset B_P(\vv)^-$
such that each $Q \in (P,P_0)$ is id-indifferent.  Thus $P$ is in the closure of $U_{\id}$.  However, 
$P \notin U_{\id}$ since $P$ is repelling.  Thus $P \in \partial U_{\id}$.  
\end{proof} 

\begin{corollary} \label{NoTypeIIICor} Let $\varphi(z) \in K(z)$ have degree $d \ge 2$.
Then no boundary point of $U_{\id}$ can be of type {\rm III}.  
\end{corollary}

\begin{proof} 
If $Q$ were a type III boundary point of a component $U_{\id}(P)$, by continuity it would be fixed by $\varphi$. 
The tangent space $T_Q$ contains precisely two directions $\vv_1, \vv_2$.  
By a result of Rivera-Letelier (see \cite{R-L1}, Lemmas 5.3 and 5.4, or \cite{B-R}, Lemma 10.80), 
$Q$ is an indifferent fixed point of $\varphi$
and $\varphi_*(\vv_1) = \vv_1$, $\varphi_*(\vv_2) = \vv_2$.  

Let $L$ be a complete, algebraically closed nonarchimedan valued field containing $K$  
such that $\iota_K^L(Q)$ is of type II.  Write $\iota$ for $\iota_K^L$.  
By Proposition \ref{FaberProp}.(6b), $\iota(Q)$ is still
an indifferent fixed point of $\varphi$, and $\iota_*(\vv_1), \iota_*(\vv_2) \in T_{Q,L}$ 
are both fixed by $(\varphi_L)_*$.  

Clearly $\iota(Q)$ cannot be additively indifferent,
since an additively indifferent fixed point has only one fixed direction 
in its tangent space.  If $\iota(Q)$ were multiplicatively indifferent, 
then $\varphi_L$ would would have reduced a rotation number $\tlambda \ne 0$ at $\iota(Q)$, and  
$\iota_*(\vv_1), \iota_*(\vv_2)$ would be the only fixed directions in $T_{Q,L}$. 
Suppose $\vv_1$ is the direction pointing into $U_{\id}(P)$. 
By Lemma \ref{SecondPersistenceLemma} there would be a point $Q_1$ in the direction $\vv_1$
such that each point of $(Q,Q_1)$ would be multiplicatively indifferent with reduced 
rotation number $\tlambda$.  This contradicts that each point of $(Q,Q_1) \cap U_{\id}(P)$ is id-indifferent.
Hence $\iota(Q)$ cannot be multiplicatively indifferent.  

The only remaining possibility is that $Q$ is id-indifferent.  However by Proposition \ref{FirstPersistenceLemma},
there would then be a ball $\cB_\rho(Q,\eta)^-$ such that each $T \in \cB_\rho(Q,\eta)^-$ was id-indifferent,
and this contradicts that $Q$ is a boundary point of $U_{\id}(P)$.  

Thus, no boundary point of $U_{\id}$ can be of type III.
\end{proof} 

\noindent{\bf Remark.}  Suppose $P \in \BHH_K$ is a boundary point of $U_{\id}$, and $\vv \in T_P$ 
is the direction pointing into $U_{\id}$.  Then $\varphi_*(\vv) = \vv$.  If we extend $K$ so that
$P$ becomes type II, then by Lemma \ref{SecondPersistenceLemma} we have $m_{\varphi}(P,\vv) = 1$ 
and $\#\tF_{\varphi}(P,\vv) \ge 2$.  Put $\tM = \#\tF_{\varphi}(P,\vv)$.   
Using the methods of Lemma \ref{SecondPersistenceLemma}, 
it can be shown that there are an $\varepsilon > 0$ and
a point $P_0 \in B_P(\vv)^-$ such that $U_{\id} \bigcap B_P(\vv)^- \bigcap \cB_\rho(P,\varepsilon)^-$ 
is the `cone with sides of slope $\tM-1$', given by 
\begin{equation} \label{FCone1}
\Big( \bigcup_{S \in (P,P_0)} \{Q \in \BHH_K : [P,S] \cap (S,Q] = \phi, \rho(S,Q) < (\tM-1) \cdot \rho(P,S) \} \Big) 
\cap \cB_\rho(P,\varepsilon)^- \ .
\end{equation} 
We will not need this, so we omit the proof. However, we note that it shows one could define id-indifference
for points of type III and IV without using Faber's base change map, by saying that $Q \in U_{\id}$ 
if and only if there is a neighborhood $\cB_\rho(Q,\varepsilon)^-$ such that each type II 
point $P \in \cB_\rho(Q,\varepsilon)^-$ is id-indifferent for $\varphi$.

\smallskip
Lastly, we give a description of 
the locus of id-indifference as it approaches a type I fixed point.  
Recall that if $\alpha \in \PP^1(K)$ satisfies $\varphi(\alpha) = \alpha$, and if coordinates are chosen
so that $\alpha \ne \infty$, then the {\em multiplier} of $\alpha$ is the derivative 
$\lambda = \lambda_\alpha = \varphi^{\prime}(\alpha)$.  It is independent of the choice of coordinates.  
By standard terminology, 
the fixed point $\alpha$ is {\em superattracting} if $\lambda = 0$, {\em attracting} if $0 < |\lambda| < 1$, 
{\em indifferent} if $|\lambda| = 1$, and {\em repelling} if $|\lambda| > 1$.  We refine the classification
of indifferent fixed points  as follows:

\begin{definition} \label{ReducedMultiplierDef} 
{\em Suppose $\alpha \in \PP^1(K)$ is an indifferent fixed point of $\varphi(z)$ 
with multiplier $\lambda$;
we call its reduction $\tlambda \in \tk$ the {\rm reduced multiplier} of $\alpha$.  Then  

$(1)$ If $\tlambda = 1$, we say $\alpha$ is $\tid$-indifferent;
  
$(2)$ If $\tlambda \ne 1$, we say $\alpha$ is $\trot$-indifferent.  

\noindent{If $P_0 \in \BHH_K$ and $r > 0$,}
 we define the {\em strong tube} $T((\alpha,P_0),r)^-$ to be 
the union of the balls $\cB_\rho(Q,r)^-$ for all $Q \in (\alpha,P_0)$.} 
\end{definition}

\begin{lemma}[Third Persistence Lemma] \label{ThirdPersistenceLemma} 
Let $\varphi(z) \in K(z)$ have degree $d \ge 2$.  
Suppose $\alpha \in \PP^1(K)$ is a type {\rm I} fixed point
of $\varphi$, with multiplier $\lambda$.  Then 

$(A)$ $\alpha$ is $\tid$-indifferent $($that is, $\tlambda = 1)$ 
if and only if $\alpha$ is a boundary point $U_{\id}$.  
In that case, there are a $P \in U_{\id}$ and an $r > 0$ 
such that $U_{\id}$ contains the strong tube $T((\alpha,P),r)^-$.  
If \,$\lambda = 1$, there is 
a sequence of points $\{P_n\}_{n \ge 1}$ in $(\alpha,P)$, converging to $\alpha$,
such that $U_{\id}$ contains the strong tube $T((\alpha,P_n),n)^-$ for each $n$.  

$(B)$ $\alpha$ is $\trot$-indifferent, with reduced multiplier $\tlambda \ne 1$,  
if and only if there is a $P \in \BHH_K$ such that each $Q \in (\alpha,P)$ is 
multiplicatively indifferent and has reduced rotation number $\tlambda$ for the axis $(\alpha,P)$.
\end{lemma}  

\begin{proof}  If $\alpha \in \PP^1(K)$ is an attracting or repelling fixed point, 
there is a neighborhood $V$ of $\alpha$ in $\BPP_K$ 
such that each $P \in V$ with $P \ne \alpha$ is moved by $\varphi$;  
thus $\alpha$ is not a boundary point of $U_{\id}$. 

Henceforth suppose $\alpha \in \PP^1(K)$ is an indifferent fixed point, 
so its multiplier $\lambda$ satisfies $|\lambda| = 1$.  
After a change of coordinates, we can assume that $\alpha = 0$.  Let $(F,G)$ be 
a normalized representation of $\varphi$ at $\zeta_G$, and write $F(X,Y) = a_d X^d + \cdots + a_0 Y^d$,
$G(X,Y) = b_d X^d + \cdots + b_0 Y^d$.  Since $\varphi(0) = 0$, we have $a_0 = 0$, $b_0 \ne 0$,  
and $a_1/b_0 = \lambda$; since $\alpha$ is indifferent, it follows that $|a_1| = |b_0|$.  
Consider $\varphi(z)$ on the path $(\alpha,\zeta_G)$.  For each $t \in K^{\times}$, if we conjugate
$\varphi(z)$ by $\gamma = \big( \begin{array}{cc} t & 0 \\ 0 & 1 \end{array} \big)$, then 
$\varphi^{\gamma}(z)$ has the representation $(F^t,G^t)$ where 
\begin{equation*}
F^t(X,Y)  =  t^d a_d X^d + \cdots + t a_1 XY^{d-1} ,  \quad G^t(X,Y)  =  t^{d+1} b_d X^d + \cdots + t b_0 Y^d \ .
\end{equation*} 
If $|t|$ is small enough, then $t a_1$ and $t b_0$ will be the unique coefficients of $F^t$ and $G^t$ with largest
absolute value.  In this situation, dividing $F^t(X,Y)$ and $G^t(X,Y)$ by $t b_0$ and setting
\begin{equation*}
F_t(X,Y)  =  (t^{d-1} a_d/b_0) X^d + \cdots + \lambda XY^{d-1} ,  \quad G_t(X,Y)  =  (t^d b_d/b_0) X^d + \cdots + Y^d \ .
\end{equation*}
gives a normalized representation $(F_t,G_t)$ for $\varphi^{\gamma}(z)$. 
The reductions of $F_t$ and $G_t$ are $\tF_t(X,Y) = \tlambda XY^{d-1}$, $\tG_t(X,Y) = Y^d$, so $\GCD(\tF_t,\tG_t) = Y^{d-1}$ 
and $\tphi^{\gamma}$ has the representation $(\tF_{t,0},\tG_{t,0}) = (\tlambda X, Y)$.  

Thus $\tphi^{\gamma}(z) = \tlambda z$ for all sufficiently small $|\,t|$.  It follows that $\tlambda = 1$ 
if and only if $\alpha = 0$ is $\tid$-indifferent, and this holds if and only if there is a 
$P_0 \in (\alpha,\zeta_G)$ such that each type II point $Q \in (\alpha,P_0)$ is id-indifferent.  
Likewise, $\tlambda \ne 1$ if and only if $\alpha = 0$ is $\trot$-indifferent with reduced multiplier $\tlambda$,
and this holds if and only if there is a $P_0 \in (\alpha,\zeta_G)$ such that each type II point $Q \in (\alpha,P_0)$
is multiplicatively indifferent, with reduced rotation number $\tlambda$ for the axis $(\alpha,P_0)$.  
By enlarging $K$ and using Proposition \ref{FaberProp}, these assertions apply to all points in $(\alpha,P_0)$. 
This proves (B), and the first part of (A). 

Now suppose $\tlambda = 1$, but $\lambda \ne 1$.  Put $\hF(X,Y) = X Y^{d-1}$ and $\hG(X,Y) = Y^d$, and let
$\eta = \ord(\lambda -1) > 0$.  For all sufficiently small $|\,t|$, we will have $\ord(F_t-\hF) = \ord(G_t-\hG) = \eta$.
Put $r = \eta/(d+1)$.  By Corollary \ref{FirstPersistenceCor}, there is a $P \in (\alpha,\zeta_G)$
such that for each $Q \in (\alpha,P)$, the ball $\cB_\rho(Q,r)^-$ is contained in $U_{\id}$.  
The strong tube $T((\alpha,P),r)^-$ is the union of these balls, so it is contained in $U_{\id}$.  

Finally, suppose $\lambda = 1$, and let $\hF(X,Y)$, $\hG(X,Y)$ be as above.  For each positive integer $n$, 
there is an $R_n > 0$ such that if $0 < |\,t| < R_n$, then $\ord(F_t-\hF),\ord(G_t-\hG) > n (d+1)$. 
Take $P_n = \zeta_{0,R_n}$, and put $\eta_n = n \cdot (d+1)$.  By Corollary \ref{FirstPersistenceCor},
for each type II point $Q \in (\alpha,P_n)$, the strong ball $\cB_\rho(Q,n)^-$ is contained in $U_{\id}$.
The strong tube $T((\alpha,P_n),n)^-$ is the union of these balls, so it is contained in $U_{\id}$.  
\end{proof} 

By Corollary \ref{NoTypeIIICor}, no boundary point of $U_{\id}$ can be of type III.    
By Lemma \ref{SecondPersistenceLemma}, type II boundary points of $U_{\id}$ 
are either repelling fixed points or additively indifferent fixed points. 
By extending $K$ and using Theorem \ref{WeightFormulaTheorem} and Proposition \ref{FaberProp}, we see that  
type IV boundary points of $U_{\id}$ are necessarily additively indifferent.  
By Lemma \ref{ThirdPersistenceLemma}, Type I boundary points of $U_{\id}$ are classical indifferent fixed points, 
and in particular are endpoints of $\Gamma_\Fix$.    

\begin{corollary} \label{IdIndiffBoundaryCor} 
If $Q$ is a boundary point of $U_{\id}$, 
then $Q$ is either 

$(A)$ a repelling fixed point of $\varphi$ in $\BHH_K$, or

$(B)$ an additively indifferent fixed point of $\varphi$ in $\BHH_K$, or

$(C)$ a $\tilde{1}$-indifferent fixed point of $\varphi$ in $\PP^1(K)$.

\noindent{If $Q \in \BHH_K$, and if $\vv \in T_Q$} 
is the direction pointing into $U_{\id}$, 
then $\#\tF_\varphi(Q,\vv) \ge 2$.
\end{corollary} 

\begin{proof} Suppose $Q$ is a boundary point of a component $U_{\id}(P)$.  
Since the path $(Q,P)$ consists of id-indifferent fixed points, 
by continuity $Q$ is fixed by $\varphi$ and $\varphi_*(\vv) = \vv$.  
Hence the result follows from Corollary \ref{NoTypeIIICor} and 
Lemmas \ref{SecondPersistenceLemma}, \ref{ThirdPersistenceLemma}.
\end{proof} 

\noindent{\bf Remark.}  Suppose $\alpha \in \PP^1(K)$ is a type I boundary point of $U_{\id}$.     
By Lemma \ref{ThirdPersistenceLemma}, $\alpha$ is an indifferent fixed point of $\varphi$ 
with reduced multiplier $\tlambda = \tid$.  
Let the multiplicity of $\alpha$ as a fixed point of $\varphi$ be $M \ge 1$.    
By a more complicated argument using the methods of Lemma \ref{ThirdPersistenceLemma}, 
it can be shown that there are a point $P \in U_{\id}$ and a constant $C = C(\varphi,P,\alpha) > 0$ 
such that if $\vv \in T_P$ is the direction pointing toward $\alpha$, 
then $U_{\id} \cap B_P(\vv)^-$ is `the cone with sides of equation $y = (M-1)x + C$' given by 
\begin{equation} \label{FCone2}
 \bigcup_{S \in (P,\alpha)} \{Q \in \BHH_K : [P,S] \cap (S,Q] = \phi, 
\rho(S,Q) < (M-1) \cdot \rho(P,S) + C\}  \ .
\end{equation}

\section{Additional Structure in the Dynamics of $\varphi$.} \label{ApplicationsSection} 
  
In this section we apply the Persistence Lemmas to the dynamics 
of $\varphi$.  We first obtain a formula for $s_\varphi(P,\vv)$ 
as a sum of terms $\#F_\varphi(Q,\vv)$ and $\#\tF_\varphi(Q,\vv)$ when $P$ is id-indifferent, 
which shows (among other things) that the locus of id-indifference 
has at most $\lfloor (d+1)/2\rfloor$ components. 
We show that points which are multiplicatively indifferent belong to `maximal rotational axes' whose 
endpoints are highly constrained.  Finally, we sharpen Theorem \ref{BaryCenterTheorem} 
(the Dynamical Characterization of $\MinResLoc(\varphi)$) in the case
when $\MinResLoc(\varphi)$ is a segment. 

\medskip
{\bf Balance Conditions for Id-Indifferent Points.}
For type II points $P$ which are not id-indifferent,   
Proposition \ref{FProp} gives ``balance conditions'' for $P$ to belong to $\MinResLoc(\varphi)$,
using the directional fixed point multiplicities $\#F_{\varphi}(P,\vv)$ and $\#\tF_\varphi(P,\vv)$.
 
We can now extend Proposition \ref{FProp} to id-indifferent points. 
Recall that if $P$ is a type II id-indifferent fixed point, 
then $U_{\id}(P)$ is the component of the locus of id-indifference containing $P$.  
Recall also that $\#F_\varphi(P,\vv)$ is the number of type I fixed points in $B_P(\vv)^-$,
counted with multiplicity.  Given $\vv \in T_P$, we now define 
\begin{equation*}
\#F_\varphi(P,\vv)_\Visible
\end{equation*} 
to be the number of type I fixed points in $\partial U_{\id}(P) \cap B_P(\vv)^-$, counted with multiplicity. 
 
When $P$ is id-indifferent and $\vv \in T_P$, there is a formula for $s_{\varphi}(P,\vv)$ 
as a sum of directional fixed point multiplicities, but it extends over the boundary of $U_{\id}(P)$ 
rather than being localized at $P$.  Using Proposition \ref{SvarphiProp}, this yields 
balance conditions for $P$ to belong to $\MinResLoc(\varphi)$, in terms of directional fixed point multiplicities:   

\begin{proposition} \label{IndiffBalanceProp} Let $P$ be a type {\rm II} id-indifferent fixed point, and let
$U_{\id}(P)$ be the component of the locus of id-indifference containing $P$.  
Given $Q \in \partial U_{\id}(P)$, let $\vv_{Q,P} \in T_Q$ be the direction pointing into $U_{\id}(P)$. 
Suppose $\vv \in T_P$.  Then
\begin{eqnarray} 
s_\varphi(P,\vv) & = & 
    \sum_{ \substack{ \text{\rm type II points $Q$ in } \\ \partial U_{\id}(P) \cap \Gamma_{\Fix,\Repel} \cap B_P(\vv)^-} } 
\Big(\#\tF_\varphi(Q,\vv_{Q,P}) - 2  \ + \ 
         \sum_{\substack{\vw \in T_Q \\ \vw \ne \vv_{Q,P}}} \#F_\varphi(Q,\vw) \Big)  \label{idIndiffFormula} \\
& \ge & \#F_\varphi(P,\vv) - \#F_\varphi(P,\vv)_\Visible  \ . \label{idIndiffBound} 
\end{eqnarray}
Furthermore, $P \in \MinResLoc(\varphi)$ if and only if $s_\varphi(P,\vv) \le \frac{d-1}{2}$
for each $\vv \in T_P$, and $\MinResLoc(\varphi) = \{P\}$ if and only if 
$s_\varphi(P,\vv) < \frac{d-1}{2}$ for each $\vv \in T_P$.  
\end{proposition} 

\noindent{\bf Remark.} If $Q \in \partial U_{\id}(P) \cap B_P(\vv)^-$ is a focused repelling fixed point, 
then $\#\tF_\varphi(Q,\vv_{Q,P}) = \deg_\varphi(Q) + 1$, and $\#F_\varphi(Q,\vw)  = 0$ 
for each $\vw \in T_Q$ with $\vw \ne \vv_{Q,P}$.  Hence the contribution to (\ref{idIndiffFormula})
from $Q$ is $\#\tF_\varphi(Q,\vv_{Q,P}) - 2 = \deg_\varphi(Q) - 1$. 

\begin{proof}
Fix $\vv \in T_P$.  
To compute $s_\varphi(P,\vv)$, it suffices to choose a type I point $\alpha \notin B_P(\vv)^-$,
and count the number of solutions to $\varphi(z) = \alpha$ in $B_P(\vv)^-$. 
Since $U_{\id}(P) \subset \BHH_K$, there are no solutions in $U_{\id}(P)$.  Since the type I boundary points
of $U_{\id}(P)$ are fixed, they do not give solutions either.

Thus, the number of solutions to 
$\varphi(z) = \alpha$ in $B_P(\vv)^-$ is the sum of the number of solutions in the balls $B_Q(\vw)^-$, 
as $Q$ runs over all points $\partial U_{\id}(P) \cap B_P(\vv)^-$ and $\vw$ runs over 
$T_Q \backslash \{\vv_{Q,P}\}$. 
Fix $Q \in \partial U_{\id}(P) \cap B_P(\vv)^-$.  
If $Q$ is of type I  or type IV, then $\vv_{Q,P}$ is the only element of $T_Q$.   
By Corollary \ref{NoTypeIIICor}, $Q$ cannot be of type III. 
It remains to consider the case where $Q$ is of type II. 

\smallskip 
First suppose $Q \notin \Gamma_{\Fix,\Repel}$.  By Lemma \ref{SecondPersistenceLemma} 
$Q$ must be an additively indifferent fixed point.    
For such a $Q$, there is a unique
$\vw_0 \in T_Q$ fixed by $\varphi_*$, and Lemma  \ref{SecondPersistenceLemma} shows $\vw_0 = \vv_{Q,P}$.

Fix $\vw \in T_Q$ with $\vw \ne \vv_{Q,P}$.  
We claim that $\vw$ points away from $\Gamma_{\Fix,\Repel}$.  
Otherwise, $\Gamma_{\Fix,\Repel} \subset B_Q(\vw)^-$, and the $d+1$ type I fixed points of $\varphi$ 
would all belong to $B_Q(\vw)^-$, giving $\#F_\varphi(Q,\vw) = d+1$.  
However, $\varphi_*(\vw) \ne \vw$, so $\#\tF_\varphi(Q,\vw) = 0$.  Since $Q$ is not id-indifferent,
Lemma \ref{FirstIdentificationLemma} shows that $s_\varphi(Q,\vw) = \#F_\varphi(Q,\vw) - \#\tF_\varphi(Q,\vw) = d+1$.
However, this contradicts the universal inequality $s_\varphi(Q,\vw) \le d - \deg_\varphi(Q) = d-1$.  
Hence $\vw$ points away from $\Gamma_{\Fix,\Repel}$. 
This means $\#F_\varphi(Q,\vw) = 0$, so Lemma \ref{FirstIdentificationLemma} gives $s_\varphi(Q,\vw) = 0$.  
Thus $\varphi(B_Q(\vw)^-) = B_Q(\varphi_*(\vw))^-$.  However, since $\deg_\varphi(Q) = 1$, 
and since $\varphi_*(\vv_{Q,P}) = \vv_{Q,P}$, we cannot have $\varphi_*(\vw) = \vv_{Q,P}$.  
This means $\varphi_*(\vw)$ points away from $P$, so $B_Q(\varphi_*(\vw))^- \subset B_P(\vv)^-$. 
Hence there are no solutions to $\varphi(z) = \alpha$ in $B_Q(\vw)^-$.  

\smallskip
Next suppose $Q \in \Gamma_{\Fix,\Repel}$. Take any $\vw \in T_Q$ with $\vw \ne \vv_{Q,P}$.  
If $\varphi_*(\vw) = \vv_{Q,P}$, then the number of solutions to $\varphi(z) = \alpha$ in $B_Q(\vw)^-$
is $m_\varphi(Q,\vw) + s_\varphi(Q,\vw)$; if $\varphi_*(\vw) \ne \vv_{Q,P}$, the number of solutions
is $s_\varphi(Q,\vw)$.  
As $\vw$ varies, the total number of solutions to $\varphi(z) = \alpha$ in $\BPP_K \backslash B_Q(\vv_{Q,P})^-$ is
\begin{eqnarray*}
\sum_{ \substack{\vw \in T_Q, \, \vw \ne \vv_{Q,P} \\ \varphi_*(\vw) = \vv_{Q,P} } } 
           \Big(m_\varphi(Q,\vw) & + & s_\varphi(Q,\vw) \Big) \ + \ 
\sum_{ \substack{\vw \in T_Q, \, \vw \ne \vv_{Q,P} \\ \varphi_*(\vw) \ne \vv_{Q,P} } }  s_\varphi(Q,\vw) \\
& = & \sum_{ \substack{\vw \in T_Q, \, \vw \ne \vv_{Q,P} \\ \varphi_*(\vw) = \vv_{Q,P} } } 
          m_\varphi(Q,\vw) \ + \ 
\sum_{ \substack{ \vw \in T_Q \\ \vw \ne \vv_{Q,P} } } s_\varphi(Q,\vw) \\
& = & \Big( \deg_\varphi(Q) - 1 \Big) + \sum_{ \substack{\vw \in T_Q \\ \vw \ne \vv_{Q,P} } } 
\Big( \#F_\varphi(Q,\vw) - \#\tF_\varphi(Q,\vw) \Big) 
\end{eqnarray*}
Here, the equality between the second and third lines follows from Lemma \ref{SecondPersistenceLemma}
(which gives $m_\varphi(Q,\vv_{Q,P}) = 1$)  and from Lemma \ref{FirstIdentificationLemma} (which applies 
because $Q$ is not id-indifferent). 
Continuing on, and using that $\sum_{\vw \in T_Q} \#\tF_\varphi(Q,\vw) = \deg_\varphi(Q) + 1$,
we see that the number of solutions to $\varphi(z) = \alpha$ in $\BPP_K \backslash B_Q(\vv_{Q,P})^-$ is 
\begin{eqnarray*} 
\quad & = & \Big( \deg_\varphi(Q) - 1 \Big) 
  \ + \ \Big( \sum_{ \substack{\vw \in T_Q \\ \vw \ne \vv_{Q,P} } } \#F_\varphi(Q,\vw) \Big) 
  \ - \ \Big( \deg_\varphi(Q)+1 - \#\tF_\varphi(Q,\vv_{Q,P}) \Big) \\
& = & \Big(\#\tF_\varphi(Q,\vv_{Q,P}) -  2 \Big) \ + \ \Big( \sum_{ \substack{\vw \in T_Q \\ \vw \ne \vv_{Q,P} } } \#F_\varphi(Q,\vw) \Big) \ . \
\end{eqnarray*} 

Summing over all $Q \in \partial U_{\id}(P) \cap \Gamma_{\Fix,\Repel} \cap B_P(\vv)^-$
yields (\ref{idIndiffFormula}). 
However, Lemma \ref{SecondPersistenceLemma} gives $\#\tF_\varphi(Q,\vv_{Q,P}) \ge 2$ for each 
$Q \in \partial U_{\id}(P)$ of type II.  Hence by (\ref{idIndiffFormula}) 
\begin{equation*}
s_\varphi(P,Q) \ \ge \ \!\!\!
\sum_{ \substack{ \text{\rm type II points $Q$ in } \\ \partial U_{\id}(P) \cap \Gamma_{\Fix,\Repel} \cap B_P(\vv)^-} } \!\!\!
\Big( \sum_{ \substack{\vw \in T_Q \\ \vw \ne \vv_{Q,P} } } \#F_\varphi(Q,\vw) \Big) 
\ = \ \#F_\varphi(P,\vv) - \#F_\varphi(P,\vv)_\Visible \ ,
\end{equation*}
which is (\ref{idIndiffBound}).  
The final assertions in the Proposition
follow from Proposition \ref{SvarphiProp}.
\end{proof} 

\begin{corollary} \label{IdIndiffClosureCountCor}
The closure of each component $U_{\id}(P)$ of the locus of id-indifference contains at least two 
type {\rm I} fixed points $($counting multiplicities$)$.
\end{corollary}

\begin{proof} Suppose $U_{\id}(P)$ were a component having at most one type I fixed point in its closure
(counting multiplicities).  Without loss we can assume $P$ is of type II.  
By Proposition \ref{IndiffBalanceProp} 
\begin{eqnarray*}
\sum_{\vv \in T_P} s_\varphi(P,\vv) & \ge &
\sum_{\vv \in T_P} \Big(\#F_\varphi(P,\vv) - \#F_\varphi(P,\vv)_\Visible \Big) \\
& \ge & (d+1) - 1 \ = \ d \ .
\end{eqnarray*}
This contradicts the universal inequality $\sum_{\vv \in T_P} s_\varphi(P,\vv) \le d-1$.
\end{proof} 

\begin{corollary} \label{IdIndiffComponentCountCor} 
The locus of id-indifference of $\varphi$ has at most $\lfloor \frac{d+1}{2} \rfloor$ components.
\end{corollary}

\begin{proof} This follows from Corollary \ref{IdIndiffClosureCountCor}, 
since each type I fixed point can belong to the closure of at most one component $U_{\id}(P)$.
\end{proof}

\begin{corollary} \label{BranchPtIndiffCor} Suppose $Q \in \Gamma_{\Fix}$ is a branch point of $\Gamma_{\Fix,\Repel}$ but is not
a branch point of $\Gamma_{\Fix}$.  Then $Q$ is id-indifferent.  
\end{corollary}

\begin{proof}
Let $P$ be an id-indifferent point in some branch of $\Gamma_{\Fix,\Repel}$ off $\Gamma_{\Fix}$ at $Q$,
and let $U_{\id}(P)$ be the corresponding component of the locus of id-indifference.  
By Corollary \ref{IdIndiffClosureCountCor} $U_{\id}(P)$ has at least one type I fixed point $\alpha$
in its closure, and the path $(P,\alpha)$ goes through $Q$.  Since $(P,\alpha) \subset U_{\id}(P)$, 
it follows that $Q$ is id-indifferent. 
\end{proof} 

\begin{corollary} \label{BiFocusBoundCor} 
A given edge of $\Gamma_{\Fix}$ can contain at most two bi-focused repelling fixed points.
\end{corollary} 

\begin{proof}
Suppose an edge contained bi-focused repelling fixed points  $P_1, P_2, P_3$,  with $P_2$ between $P_1$ and $P_3$.  
By Proposition \ref{BeadProp}, $P_2$ is a boundary point of a component of the locus of 
id-indifference.  By Corollary \ref{IdIndiffClosureCountCor} that component has a type I fixed point $\alpha$ in its boundary, 
so the interior of the path $[P_2,\alpha]$ would be contained in it.  This is impossible,  
because the path would necessarily pass through $P_1$ or $P_3$, which are not id-indifferent. 
\end{proof} 
 
\medskip
{\bf Maximal Rotational Axes.}  If $P$ is a type II point where $\varphi$ has multiplicatively indifferent reduction, 
and has reduced rotation number $\tlambda$  for an axis $(P_0,P_1)$, 
then by Lemma \ref{SecondPersistenceLemma} 
there is a segment $(P_0,P_1)$ containing $P$ such that each type II point $Q \in (P_0,P_1)$ has multiplicatively 
indifferent reduction and has reduced rotation number $\tlambda$.  If $(T_0,T_1)$ is another segment 
(not necessarily containing $P$) such that each type II point $Q \in (T_0,T_1)$ 
has multiplicatively indifferent reduction with reduced rotation number $\tlambda^\prime$,  
and if $(P_0,P_1) \cap (T_0,T_1)$ is nonempty, then the overlap contains a type II point $Q$.
This means that $\tlambda = \tlambda^\prime$ (for an appropriate orientation of $(T_0,T_1)$, 
so $(P_0,P_1) \cup (T_0,T_1)$ is another 
segment with the same property.  Hence, there is a maximal segment $(P_0,P_1)$ containing $P$ with the property
that each type II point $Q \in (P_0,P_1)$ has multiplicatively indifferent reduction  with reduced rotation number $\tb$.
We will call this segment {\em the maximal rotational axis} of $P$.  

\begin{corollary} \label{AxisEndpointsCor} Let $\varphi(z) \in K(z)$ have degree $d \ge 2$.
Suppose $P$ is a type {\rm II} point where $\varphi$ has multiplicatively indifferent reduction,
and let $(P_0,P_1)$ be the maximal rotational axis of $P$.  Then $(P_0,P_1) \subset \Gamma_{\Fix}$,
and each endpoint of $(P_0, P_1)$ is either 

$(A)$ a type {\rm I} $\trot$-indifferent fixed point, or

$(B)$ a type {\rm II} repelling fixed point.  
\end{corollary} 

\begin{proof}  Since $P$ has multiplicatively indifferent reduction, there are exactly two tangent directions 
$\vv_0, \vv_\infty \in T_P$ which are fixed by $\varphi_*$.  By Lemma \ref{FirstIdentificationLemma}, 
each of $B_P(\vv_0)^-$ and $B_P(\vv_\infty)^-$ contains a type I fixed point of $\varphi$.  Thus,
$P \in \Gamma_{\Fix}$.  The same argument applies to each type II point $Q \in (P_0,P_1)$, 
and since the type II points are dense in $(P_0,P_1)$ for the strong topology, 
it follows that $(P_0,P_1) \subset \Gamma_{\Fix}$.  

By continuity, both $P_0$ and $P_1$ are fixed by $\varphi$.
By an argument similar to the one in Corollary \ref{NoTypeIIICor}, 
neither $P_0$ nor $P_1$ can be of type III, 
and they cannot be of type IV since $(P_0,P_1) \subset \Gamma_{\Fix}$, 
so they must be either of type I or II.   Consider $P_0$;  similar reasoning applies to $P_1$.  
If $P_0$ is of type I, we are done.  

If  $P_0$ is of type II, it cannot be id-indifferent because then there would be a ball $\cB_\rho(P,\varepsilon)^-$
such that each type II point $Q \in \cB_\rho(P,\varepsilon)^-$ was id-indifferent, and this ball would contain type II
points from $(P_0,P_1)$.  If $P_0$ were multiplicatively indifferent, then the direction $\vv_0 \in T_{P_0}$ (say) 
containing $(P_0,P_1)$ would be fixed by $\varphi_*$, so by Lemma \ref{SecondPersistenceLemma}
$\varphi$ would have reduced rotation number $\ta$ at $P_0$.  There would be another direction $\vv_\infty \in T_{P_0}$ 
fixed by $\varphi_*$, so by Lemma \ref{SecondPersistenceLemma} the segment $(P_0,P_1)$ would not be maximal.
If $P_0$ were additively indifferent, then the direction $\vv_0 \in T_{P_0}$ containing $(P_0,P_1)$ would be fixed 
by $\varphi_*$, hence it would be the unique $\vv \in T_P$ fixed by $\varphi_*$, 
so by Lemma \ref{SecondPersistenceLemma} there would be type II points in $(P_0,P_1)$ arbitrarily near
$P_0$ which are id-indifferent.  These contradictions show $P_0$ cannot be an indifferent fixed point, so it must
be repelling.  
\end{proof} 

{\bf Refinement of the Dynamical Characterization of $\MinResLoc(\varphi)$.} 
We can now refine Theorem \ref{BaryCenterTheorem}, giving more details in the 
case where $\MinResLoc(\varphi)$ is a segment:

\begin{theorem}  \label{SegmentCharacterization} 
Let $\varphi(z) \in K(z)$ have odd degree $d \ge 3$, and suppose $\MinResLoc(\varphi)$
is an edge $[A,B]$ of $\Gamma_\varphi$.  $A$ and $B$ may or may not have the same reduction type, 
but $(A,B)$ consists of points of only one type$:$  
each point of $(A,B)$ is either moved by $\varphi$, or is multiplicatively indifferent,
or is id-indifferent.  

$(A)$ If $(A,B)$ consists of points moved by $\varphi$, then both $A$ and $B$ belong to the crucial set.
They can be  additively indifferent, multiplicatively indifferent, or repelling fixed points, 
or points that are moved by $\varphi$,  but 
they cannot be id-indifferent.
There can be no branches of $\Gamma_{\Fix,\Repel}$ off $(A,B)$.  

$(B)$ If $(A,B)$ consists of points that are multiplicatively indifferent, 
then all points in $(A,B)$ have the same reduced rotation number for the axis $(A,B)$, 
and both $A$ and $B$ belong to the crucial set.
They can be multiplicatively indifferent or repelling fixed points, but 
they cannot be additively indifferent, id-indifferent, or moved by $\varphi$.
There can be no branches of $\Gamma_{\Fix,\Repel}$ off $(A,B)$. 

$(C)$ If $(A,B)$ consists of points that are id-indifferent, 
then $A$ and $B$ may or may not belong to the crucial set.  
They can be additively indifferent, id-indifferent, or repelling fixed points, 
but they cannot be multiplicatively indifferent, or moved by $\varphi$.
There may be branches of $\Gamma_{\Fix,\Repel}$ off $(A,B);$  
$(A,B)$ and the interiors of any such branches belong to a single component $U_{\id}(P)$
of the locus of id-indifference. 
The endpoints $($not in $(A,B))$ of branches of $\Gamma_{\Fix,\Repel}$ off $(A,B)$ must be $\tid$-indifferent 
type {\rm I} fixed points.
\end{theorem} 

\begin{proof}  First note that since $[A,B]$ is an edge of $\Gamma_\varphi$, 
no  $P \in (A,B)$ can belong have $w_\varphi(P) > 0$, since the $P$ would belong to the crucial set. 
 It follows that no $P \in (A,B)$ can be a repelling fixed point
or an  additively indifferent fixed point, since such points necessarily have positive weight.  
Thus each $P \in (A,B)$ is either moved by $\varphi$, or is multiplicatively indifferent, or is id-indifferent.  

Next, we claim that all the points in $(A,B)$ are of the same reduction type.  Suppose to the contrary that 
$P_1, P_2 \in (A,B)$ were of different types.  Consider the segment $(P_1,P_2)$.  There are two cases:

\begin{itemize}
\item[$(1)$] If one of $P_1, P_2$ is moved by $\varphi$, assume without loss that $P_1$ is moved and $P_2$ is fixed.  
Let $P \in [P_1, P_2]$ be the closest point to $P_1$ which is fixed by $\varphi$. 
Then each point of $[P_1,P)$ is moved by $\varphi$, so if $\vv_1 \in T_P$ is the direction containing $P_1$, 
then either $\varphi_*(\vv_1) \ne \vv_1$, in which case $\vv_1$ is a shearing direction at $P$, hence $w_\varphi(P) > 1$; 
or else $\varphi_*(\vv_1) = \vv_1$ and $\varphi(Q) \ne Q$ for each $Q \in [P_1,P)$.  In this case, 
$m_\varphi(P,\vv_1) > 1$ since if $m_\varphi(P,\vv_1) = 1$ there would be an subsegment $(P,Q) \subset (P,P_1)$
which was pointwise fixed by $\varphi$.  It follows that $\deg_\varphi(P) > 1$, so $P$ is a repelling fixed point,
and again $w_\varphi(P) > 1$.  This contradicts that $[A,B]$ is an edge of $\Gamma_\varphi$.

\smallskip
\item[$(2)$]  If both $P_1$ and $P_2$ are fixed by $\varphi$, then one must be id-indifferent and the other must be
multiplicatively indifferent.  Suppose $P_1$ is id-indifferent;  then $(P_1,P_2)$ would contain an endpoint $P$ of the 
locus of id-indifference $U_{\id}(P_1)$.  By Lemma \ref{ThirdPersistenceLemma}, 
$P$ must either be a repelling fixed point, or an additively indifferent fixed point, 
and both cases are impossible since then $w_\varphi(P) > 1$. 
\end{itemize}   

Next we claim that if $(A,B)$ consists of multiplicatively indifferent fixed points, then all $P \in (A,B)$ 
have the same reduced rotation number for the axis $(A,B)$.  If $P_1, P_2 \in (A,B)$ had different 
reduced rotation numbers $\tlambda_1, \tlambda_2$, let $P \in (P_1,P_2)$ be the nearest point to $P_1$ 
with reduced rotation number $\tlambda \ne \tlambda_1$.  
Then $P$ would either be an endpoint of the maximal rotational axis for $P_1$,
or it would be a point where the maximal rotational axis of $P_1$ branched off of $(P_1,P_2)$.  In the first case, 
$P$ would be either repelling fixed point or a type I fixed point, and both are impossible.  In the second case,
the direction $\vv_2 \in T_P$ towards $P_2$ would be a shearing direction, so $w_\varphi(P) > 1$, which is also impossible.    

\smallskip 
If $(A,B)$ consists of points which are moved by $\varphi$, or if $(A,B)$ 
consists of multiplicatively indifferent fixed points, 
then no $P \in (A,B)$ can be a branch point of $\Gamma_{\Fix,\Repel}$, since such a $P$ would 
necessarily have $w_\varphi(P) > 1$.  

\smallskip
If $(A,B)$ consists of id-indifferent fixed points, 
Examples D and E below show that $\Gamma_{\Fix,\Repel}$ 
may have branches off $(A,B)$.  Let $\Gamma$ be such a branch.  
Then $\Gamma$ can contain no points with $w_\varphi(Q) \ge 1$, since
$[A,B]$ is an edge of $\Gamma_\varphi$ and by definition the vertices of $\Gamma_\varphi$ are either 
points of $\cCr(\varphi)$ or branch points of the tree they span.  
Since $U_{\id}(P)$ is open,
$\Gamma$ contains points of $U_{\id}(P)$.  Let $Q$ be an endpoint of $U_{\id}(P)$ in $\Gamma$.  
If $Q \in \BHH_K$, then by Lemma \ref{SecondPersistenceLemma} $Q$ would be a repelling fixed point
or an additively indifferent fixed point of $\varphi$ belonging to $\Gamma_{\Fix}$, and in either case $w_\varphi(Q) \ge 1$,
a contradiction.  Thus $Q$ must be of type I, and by Lemma \ref{ThirdPersistenceLemma} it is $\tid$-indifferent.
This also shows that each interior point of $\Gamma$ belongs to $U_{\id}(P)$.

\smallskip
Now consider the nature of the endpoints $A$, $B$.  Since $[A,B]$ is an edge of $\Gamma_\varphi$, 
its endpoints must belong to the crucial set or be branch points of $\Gamma_\varphi$. 
 
First suppose $(A,B)$ consists of points moved by $\varphi$. We claim that both $A$ and $B$ must 
belong to the crucial set. Consider $A$.  If it belongs to the crucial set, we are done.  If not, it must
be a branch point of $\Gamma_\varphi$, hence a branch point of $\Gamma_{\Fix,\Repel}$.  
If $\varphi(A) \ne A$, then $w_\varphi(P) = v(A) - 2 > 0$, so $A$ belongs to the crucial set.  If $\varphi(A) = A$, 
then $A$ must either be a repelling fixed point, or must be multiplicatively or additively indifferent;  
it cannot be id-indifferent, since otherwise $(A,B)$ would contain points of the component $U_{\id}(A)$
of the locus of id-indifference.  
If $A$ is a repelling fixed point then $w_\varphi(P) \ge 1$,
and if it is a multiplicatively or additively indifferent branch point of $\Gamma_{\Fix,\Repel}$ it necessarily 
has a shearing direction, so again $w_\varphi(P) \ge 1$.  Thus $A$ belongs to the crucial set; similar arguments apply to $B$.
In the argument above we have seen that $A$ and $B$ cannot be id-indifferent;  they can be repelling fixed points,
or multiplicatively or additively indifferent fixed points, or they can be moved by $\varphi$.  
Examples F(1) and F(2) below 
show they may or may not have the same reduction type.    

If $(A,B)$ consists of multiplicatively indifferent fixed points, again we claim that $A$ and $B$ must belong
to the crucial set.  Consider $A$.  If it belongs to the crucial set, we are done.  If not, since $[A,B]$
is an edge of $\Gamma_\varphi$, then $A$ must be a branch point of $\Gamma_\varphi$, and hence a 
branch point of $\Gamma_{\Fix,\Repel}$.  Since each point of $(A,B)$ is fixed by $\varphi$, by continuity
$A$ is fixed as well.  $A$ cannot be id-indifferent, since otherwise $(A,B)$ would contain points of
the component $U_{\id}(A)$ of the locus of id-indifference.  This means $A$ would be a multiplicatively 
or additively indifferent branch point of $\Gamma_{\Fix,\Repel}$, so it would have a shearing direction, 
and again $w_\varphi(P) \ge 1$.  Hence $A$ belongs to the crucial set;  similarly for $B$.  
We have seen that $A$ and $B$ are fixed by $\varphi$ but cannot be
id-indifferent;  they can be repelling fixed points, or multiplicatively or additively indifferent fixed points.
Example F(4) below shows they need not be of the same reduction type.

\smallskip
If $(A,B)$ consists of id-indifferent fixed points, then 
by continuity, $A$ and $B$ 
are both fixed by $\varphi$;  thus they can be repelling fixed points, or additively indifferent or id-indifferent 
fixed points.  However, they cannot be multiplicatively indifferent, because by Lemma \ref{SecondPersistenceLemma} 
endpoints of $U_{\id}$ in $\Gamma_{\Fix,\Repel}$ belonging to $\BHH_K$ are necessarily repelling fixed points 
or additively indifferent fixed points. 

In Example D of \S\ref{ExamplesSection}, $A$ and $B$ are repelling fixed points, 
and in Example E of \S\ref{ExamplesSection} they are id-indifferent. 
Example E shows that $A$, $B$ need not belong to the crucial set.  
Corollary \ref{SpecialBalanceConditionsCor}, together with the construction in Example A of \S\ref{GammaFixRepelSection} 
can be used to give examples where at least one of $A$ and $B$ is a focused repelling fixed point.  
A modification of the construction in Example C below
can be used to give functions $\varphi(z)$ where at least one of $A, B$ is id-indifferent and 
there are no branches of $\Gamma_{\Fix,\Repel}$ off $(A,B)$.   
If neither $A$ or $B$ is id-indifferent, then since the component $U_{\id}(P)$ of the locus of 
id-indifference containing $(A,B)$ has type I fixed points in its closure (Corollary \ref{IdIndiffClosureCountCor}), 
there must be at least one branch of $\Gamma_{\Fix,\Repel}$ off $(A,B)$. 
\end{proof}  

\noindent{{\bf Remark.} We do not know if all possibilities for $A$ and $B$ allowed by 
Theorem \ref{SegmentCharacterization} actually occur. The examples in the following section
illustrate several possibilities.

\section{\bf Examples.} \label{ExamplesSection} 

In this section we illustrate some possible configurations of $\Gamma_{\Fix,\Repel}$ and the crucial set.  
In Example C, we construct a function $\varphi$ of degree $d \ge 2$ which has $d-1$ repelling fixed points
in $\BHH_K$.  This shows that the bound in Corollary \ref{RepellingFixedPtCor} is sharp.  In Example D, 
we construct a $\varphi$ for which $\MinResLoc(\varphi)$ is a segment whose interior consists of
id-indifferent fixed points, such that there are many branches of $\Gamma_{\Fix,\Repel}$ off $\MinResLoc(\varphi)$. 
In Example E, we construct a $\varphi$ for which $\MinResLoc(\varphi)$ is a segment, 
and contains no elements of the crucial set.  
In Example F, when $\varphi$ has degree $d =3$, we give four configurations of $\MinResLoc(\varphi)$ and its 
endpoints which can occur when $\MinResLoc(\varphi)$ is a segment.  Finally, in Example G, we describe all the ways
that the crucial set of $\varphi$ can consist of a single point.

\smallskip
\noindent{\bf Example C.} (A function $\varphi$ of degree $d$, with $d-1$ repelling fixed points in $\BHH_K$.)
 
{\rm Fix $d \ge 2$.  
In this example we construct a rational function $\varphi(z) \in K(z)$ of degree $d$ 
with $d-1$ repelling fixed points, the maximum number allowed by the weight formula.  
This example is interesting for other reasons as well:
\begin{enumerate}
\item The tree $\Gamma_{\Fix,\Repel}$ has a branch off $\Gamma_{\Fix}$ which forks into $d-1$ segments.  
This shows that branches of $\Gamma_{\Fix,\Repel}$ off $\Gamma_{\Fix}$ need not just be segments. 
\item $\MinResLoc(\varphi) = \{\zeta_G\}$ consists of a single id-indifferent point, 
namely the branch point of $\Gamma_{\Fix,\Repel} \backslash \Gamma_{\Fix}$ from $(1)$.  This shows that
$\MinResLoc(\varphi)$ need not contain elements of the crucial set.  
\end{enumerate}}

We use the procedure for constructing id-indifferent points given in Example B of \S\ref{TreeTheoremSection}.
Take distinct elements $a_1, \ldots, a_{d-1} \in \tk$, and lift them to $\alpha_1, \ldots, \alpha_{d-1} \in \cO$;  
put $A(X,Y) = \prod_{i=1}^{d-1} (X-\alpha_i Y)$.  Choose $\beta_1, \ldots, \beta_d \in \cO$ with $|\beta_i| = 1$ 
for each $i$, and put $F_1(X,Y) = \prod_{i=1}^d (\pi X-\beta_i Y)$, $G_1(X,Y) = 0$.
Fix $\pi \in \cO$ with $|\,\pi| < 1$, and set 
\begin{eqnarray*} 
F(X,Y) & = & X \cdot \prod_{i=1}^{d-1} \big(X-\alpha_i Y\big) 
          \ + \ \pi \cdot \prod_{i=1}^d \big(\pi X-\beta_i Y\big) \ , \\
G(X,Y) & = & Y \cdot \prod_{i=1}^{d-1} \big(X-\alpha_i Y\big) \ .
\end{eqnarray*} 
Then $\GCD(F,G) = 1$, since if $L(X,Y)$ is a nontrivial divisor of $F(X,Y)$ and $G(X,Y)$, then $L(X,Y)$
divides $X \cdot G(X,Y) - Y \cdot F(X,Y) = \ - \pi Y \cdot F_1(X,Y)$.  However, this is impossible because 
$Y \nmid F(X,Y)$, and $(\pi X - \beta_i Y) \nmid G(X,Y)$ for each $i$.

Write $P = \zeta_G$. 
The function $\varphi(z)$ with normalized representation $(F,G)$ has the type I fixed points 
$\beta_1/\pi, \ldots, \beta_d/\pi$ and $\infty$, which all lie in the ball $B_P(\vv_\infty)^-$.  
On the other hand, it has $s_\varphi(\zeta_G,\vv) > 0$ in the directions 
$\vv_{a_1}, \ldots, \vv_{a_{d-1}} \in T_P$.  Since none of these directions contains a type I fixed point,
each must contain a focused repelling fixed point $P_i$.  By the weight formula, each 
$P_i$ has $\deg_\varphi(P_i) = 2$, and the crucial set is precisely $\{P_1, \ldots, P_{d-1}\}$.  
Thus, $\varphi$ has exactly $d-1$ repelling fixed points, each of degree $2$.

Since the $\beta_i/\pi$ and $\infty$ all lie in the same tangent direction at $P$, the paths from $P$ 
to the fixed points share a common initial segment.  Thus $P \in \Gamma_{\Fix,\Repel} \backslash \Gamma_{\Fix}$.  
Furthermore, $P$ belongs to $\cE(\varphi)$ and satisfies the balance conditions in Theorem \ref{BaryCenterTheorem};
by construction it is id-indifferent.  By moving slightly away from $P$ in each direction in $\Gamma_{\Fix,\Repel}$,
one sees that that no other point in $\Gamma_{\Fix,\Repel}$ can satisfy the balance conditions.  
Hence $\MinResLoc(\varphi) = \{P\}$.   

\smallskip
\noindent{\bf Example D.} (A function $\varphi$ where $\Gamma_{\Fix,\Repel}$ has many branches off of $\MinResLoc(\varphi)$.)
   
{\rm Let $d \ge 3$ be odd, and for simplicity, assume $\Char(\tk) \ne 2$.    
In this example we construct a rational function $\varphi(z) \in K(z)$ of degree $d$ 
for which $\MinResLoc(\varphi)$ consists of a segment connecting two repelling fixed points.
Each interior point of the segment is id-indifferent, and there are $d-1$ branches of $\Gamma_{\Fix,\Repel}$ 
off the interior of the segment which lead to $\tid$-indifferent type I fixed points. 
} 

\smallskip
We again use the procedure for constructing id-indifferent points from Example B of \S\ref{TreeTheoremSection}.
Let $\lambda, \pi, \mu \in \cO$ be nonzero parameters.  
We will require $\ord(\lambda) \gg \ord(\pi) \gg \ord(\mu) > 0$, 
but knowing the precise values of the parameters is not important. 
Write $d = 2n+1$, and take $A(X,Y) = (X-\lambda Y)^n (Y-\lambda X)^n$.  Put 
\begin{equation*}
D(X,Y) \ = \ X \cdot Y \cdot \prod_{k=1}^{n-1} (X - \mu^k Y) \cdot \prod_{k=1}^{n-1} (Y - \mu^k X) \ ,
\end{equation*}    
and set  
\begin{eqnarray*}     
F(X,Y) & = & X \cdot A(X,Y) \ + \ \pi Y \cdot D(X,Y) \ , \\ G(X,Y) & = & Y \cdot A(X,Y) \ + \ \pi X \cdot D(X,Y) \ .
\end{eqnarray*}
Then $\GCD(F,G) = 1$, since if $L(X,Y)$ is a nontrivial divisor of $F(X,Y)$ and $G(X,Y)$, then $L(X,Y)$
divides $X \cdot G(X,Y) - Y \cdot F(X,Y) = \ - \pi \cdot (X^2-Y^2) \cdot D(X,Y)$.  
However, this is impossible since by construction $Y \nmid F(X,Y)$ and $X \nmid G(X,Y)$ 
and for $k = 1, \ldots, n-1$ we have $(X - \pi^k Y) \nmid A(X,Y)$, $(Y-\pi^k X) \nmid A(X,Y)$;  furthermore  
$(X \pm Y) \nmid F(X,Y)$ since $\tF(X,Y) = X^{n+1} Y^n$ and $(X \pm Y) \nmid X^{n+1} Y^n$.  

Write $P = \zeta_G$. 
The rational function $\varphi(z)$ with normalized representation $(F,G)$ is invariant under conjugation by 
$\gamma(z) = 1/z$, and has $d+1$ type I fixed points 
$0$, $\infty$, $\pm 1$, $\mu, \mu^2, \ldots, \mu^{n-1}$, $\mu^{-1}, \mu^{-2}, \ldots, \mu^{-(n-1)}$.  
The tree $\Gamma_{\Fix}$ consists of the path $[0,\infty]$ together with the $d-1$ 
branches off it leading to the other fixed points.  

Consider the function $\ordRes_\varphi(\cdot)$ on $[0,\infty]$.  Writing
$F(X,Y) = a_d X^d + \cdots + a_0 Y^d$, $G(X,Y) = b_d X^d + \cdots b_0 Y^d$, for each $A \in K^\times$ we have 
\begin{eqnarray*}
\lefteqn{ \ordRes_{\varphi}(\zeta_{0,|A|}) - \ordRes_\varphi(\zeta_G) \ = \ } & &  \\ 
  & &  \min_{0 \le \ell \le d} \big( (d^2 + d - 2d \ell) \ord(A) - 2d \, \ord(a_\ell), 
                 (d^2 + d - 2d (\ell+1)) \ord(A) - 2d\,\ord(b_\ell) \big) \ .
\end{eqnarray*} 
Here $\ord(a_{n+1}) = \ord(b_n) = 0$, while $\ord(a_\ell), \ord(b_\ell) \ge \ord(\pi) > 0$ for other values of $\ell$.
If we require $\ord(\pi)$ to be sufficiently large relative to $\ord(\mu)$, 
and $\ord(\lambda)$ to be sufficiently large relative to $\ord(\pi)$, 
then the restriction of $\ordRes_\varphi(\cdot)$ to $[0,\infty]$ will have three affine pieces: 
there will be an $N > (n-1) \ord(\mu)$ in the value group $\ord(K^{\times})$ such that there is 
a piece with slope $-(d^2-d)$ for $\ord(A) \le -N$, 
a piece with slope $0$ for $-N \le \ord(A) \le N$, 
and a piece with slope $(d^2 - d)$ for $\ord(A) \ge N$.

Let $\zeta_{0,R_1}$, $\zeta_{0,R_2}$ be the points where the slope changes. 
It follows that $\MinResLoc(\varphi) = [\zeta_{0,R_1}, \zeta_{0,R_2}]$ and  $\Gamma_{\Fix,\Repel} = \Gamma_{\Fix}$.   
Since $\zeta_G \in (\zeta_{0,R_1},\zeta_{0,R_2})$ is id-indifferent, Corollary \ref{SegmentCharacterization} shows 
each point in $(\zeta_{0,R_1}, \zeta_{0,R_2})$ must id-indifferent. 
By construction, $\Gamma_{\Fix}$ has $d-1$ branches off the interior of $[\zeta_{0,R_1}, \zeta_{0,R_2}]$.
By Lemma \ref{SecondPersistenceLemma}, 
$\zeta_{0,R_1}$ and $\zeta_{0,R_2}$ must be bi-focused repelling fixed points with degree $n = (d-1)/2$.

\medskip
\noindent{\bf Example E.} (A function $\varphi$ where $\MinResLoc(\varphi)$ contains no points in the crucial set.) 

{\rm  Assume $\Char(K) \ne 2, 3$.  
In this example, we will construct a rational function $\varphi(z) \in K(z)$ of degree $5$, 
such that $\MinResLoc(\varphi)$ consists of a segment joining two id-indifferent points, neither of which 
belongs to the crucial set.
The crucial set consists of four repelling fixed points of degree $2$.  
The tree $\Gamma_\varphi$ they span consists of a central bar with a `$Y$' off each end,
and $\MinResLoc(\varphi)$ is the central bar.   
The type I fixed points lie on branches off the middle of the central bar, and are all $\tid$-indifferent
as required by Proposition \ref{SegmentCharacterization}(C).  
Each point of the interior of $\Gamma_\varphi$, and of $\Gamma_{\Fix,\Repel}$, is id-indifferent.}  

\smallskip
Fix  $\alpha \in K$ with $0 < |\,\alpha| < 1$. 
Let $\varphi(z) \in K(z)$ be the function with normalized representation $(F,G)$, where
\begin{eqnarray*}
F(X,Y) & = & \alpha^4 X^5 + \alpha X^4 Y + (1+\alpha) X^3 Y^2 +  \alpha X^2 Y^3 + \alpha^4 X Y^4 + \alpha^4 Y^5 \ , \\
G(X,Y) & = & \alpha^4 X^5 + \alpha^4 X^4 Y + \alpha X^3 Y^2 + (1+\alpha) X^2 Y^3 + \alpha X Y^4 + \alpha^4 Y^5 \ ,
\end{eqnarray*} 
Note that $F(X,Y) = G(Y,X)$,  so $\varphi$ is invariant under conjugation 
by $\gamma = \left( \begin{array}{cc} 0 & 1 \\ 1 & 0 \end{array} \right)$, 
and  $\varphi(1/z) = 1/\varphi(z)$.  One sees easily that  
\begin{equation*}
X \cdot G(X,Y) - Y \cdot F(X,Y) \ = \ \alpha^4 (X^6 - Y^6) \ , 
\end{equation*} 
so (identifying $\PP^1(K)$ with $K \cup \{\infty\}$) 
the fixed points of $\varphi(z)$ are the $6^{th}$ roots of unity, 
and they lie on branches off $\zeta_G$ in directions other than $\vv_0, \vv_\infty$. 

Reducing $(F,G)\pmod \fM$ we see that $(\tF,\tG) = X^2 Y^2 \cdot (X,Y)$.  
Thus $\varphi$ has id-indifferent reduction at $\zeta_G$, 
and $s_{\varphi}(\zeta_G,\vv_0) = s_{\varphi}(\zeta_G,\vv_{\infty}) = 2$. 

Conjugating by $\gamma = \left( \begin{array}{cc} \alpha & 0 \\ 0 & 1 \end{array} \right)$, which 
brings $\zeta_{0,1/\alpha}$ to $\zeta_G$, yields 
\begin{eqnarray*}
F^\alpha(X,Y) & = & \alpha^3 X^3 Y^2 + \alpha^9 X^5 + \alpha^5 X^4 Y + \alpha^4 X^3 Y^2 
                   +  \alpha^3 X^2 Y^3 + \alpha^5 X Y^4 + \alpha^4 Y^5 \\
              & \equiv & \alpha^3\big(X(X+Y)Y^2 \cdot X) \pmod{\alpha^4 \cO}  
\end{eqnarray*}
and 
\begin{eqnarray*}
G^\alpha(X,Y) & = & \alpha^3 X^2 Y^3 + \alpha^10 X^5 + \alpha^9 X^4 Y + \alpha^5 X^3 Y^2 
              +  \alpha^4 X^2 Y^3 + \alpha^3 X Y^4 + \alpha^5 Y^5 \\
         & \equiv & \alpha^3\big(X(X+Y)Y^2 \cdot Y) \pmod{\alpha^4 \cO}  \ .
\end{eqnarray*}
Dividing through by $\alpha^3$ yields a normalized representation $(F_\alpha,G_\alpha)$, and reducing it $\pmod{\fM}$ 
gives $(\tF_\alpha,\tG_\alpha) = X(X+Y)Y^2 \cdot (X,Y)$. 
Thus, $\varphi$ has id-indifferent reduction at $P = \zeta_{0,|\alpha|}$,
and $s_{\varphi}(P,\vv_0) = s_\varphi(P,\vv_{-1}) = 1$.  Since all the fixed points lie in the direction $\vv_\infty$
from $P$, it follows that $\varphi$ has a focused repelling fixed point in each of $B_P(\vv_0)^-$ and 
$B_P(\vv_{-1})^-$.

Since $\varphi(z)$ is invariant under conjugation by $z \mapsto 1/z$, it follows that $\varphi$ has
id-indifferent reduction at $Q = \zeta_{0,1/|\alpha|}$ and that $\varphi$ has a focused repelling fixed point 
in each of $B_Q(\vv_\infty)^-$ and $B_Q(\vv_{-1})$.  As $\varphi$ has degree $5$, these four focused repelling
fixed points account for all elements of the crucial set, and each must have weight $1$ and degree $2$.  

The tree $\Gamma_\varphi$ has a central bar $[P,Q]$, with two forks at each end leading to the focused repelling
repelling fixed points.  It is easy to see that the barycenter of the crucial measure $\nu_{\varphi}$,
which gives mass $1/4$ to each of the focused repelling fixed points, is the segment $[P,Q]$.
Thus $\MinResLoc(\varphi) = [P,Q]$.

The tree $\Gamma_{\Fix,\Repel}$ is composed of $\Gamma_\varphi$ together
with the branches off  $\zeta_G$ leading to the type I fixed points; one can also view $\Gamma_{\Fix,\Repel}$
as being gotten from the tree $\Gamma_{\Fix}$ by adding branches off $\zeta_G$,  
leading to the focused repelling fixed points, in the directions $\vv_0$ and $\vv_\infty$.  
We have seen that $\varphi$ is id-indifferent at $\zeta_G$.
By Proposition \ref{IdIndiffOffFixProp}, each interior point of the branches off $\Gamma_{\Fix}$ 
leading to the focused repelling fixed points must be id-indifferent, and by 
Proposition \ref{SegmentCharacterization} each interior point of the branches of $\Gamma_{\Fix,\Repel}$ 
leading to the type I fixed points must be id-indifferent.  Thus, each interior point of $\Gamma_{\Fix,\Repel}$ 
is id-indifferent for $\varphi$, and so is each interior point of $\Gamma_\varphi$. 

In particular, neither endpoint of $\MinResLoc(\varphi)$ belongs to the crucial set, 
and each point of $\MinResLoc(\varphi)$ is id-indifferent for $\varphi$.

\medskip
\noindent{\bf Example F.}  (Cubic functions $\varphi$ for which $\MinResLoc(\varphi)$ is a segment.) 

{\rm In this example, we give several functions $\varphi(z) \in K(z)$ with $\deg(\varphi) = 3$, 
for which $\MinResLoc(\varphi)$ is a segment $[A,B]$.  These functions illustrate different configurations 
which can occur for $A$ and $B$.} 
\smallskip

If $\varphi(z) \in K(z)$ has degree $3$, and has four distinct fixed points, 
then after conjugation we can assume
it has fixed points $0, 1, \alpha$ and $\infty$, for some $\alpha \in K$ with $0 < |\,\alpha| \le 1$.  
It follows that there are $A,B,C \in K$ such that $\varphi(z)$ has the form 
\begin{equation*}
\varphi(z) \ = \ \frac{(A+1)z^3 + (B-1-\alpha) z^2 + (C + \alpha)z}{Az^2 + Bz + C}  \ .
\end{equation*} 

Below we will assume $|\alpha| < 1$, 
and consider particular functions arising from different choices of $A$, $B$, and $C$.
Note that $\Gamma_{\Fix}$ is the union of the segments $[\zeta_{0,|\alpha|},\zeta_G]$, $[0,\zeta_{0,|\alpha|}]$, 
$[\alpha,\zeta_{0,|\alpha|}]$, $[1,\zeta_G]$ and $[\infty,\zeta_G]$.  In our examples, $\MinResLoc(\varphi)$
will be the segment $[\zeta_{0,|\alpha|},\zeta_G]$. 
We will be concerned with the nature of the points of $(\zeta_{0,|\alpha|},\zeta_G)$
(whether they are moved by $\varphi$, multiplicatively indifferent, or id-indifferent), 
and the dynamical behavior of $\zeta_{0,|\alpha|}$ and $\zeta_G$.  We will write $\tphi(z)$
for the reduction of $\varphi$ at $\zeta_G$, and $\tpsi(u)$ for its reduction at $\zeta_{0,|\alpha|}$,
obtained by letting $\psi(u)$ be the conjugate of $\varphi$ by $z = \alpha u$ and then reducing.  

\smallskip
(1) If $A = 0$, $B = 0$, and $C = -\alpha$, then $\varphi(z) = (z^3 - (1+\alpha)z^2)/(-\alpha)$.  In this 
case $\varphi$ moves each point of $(\zeta_{0,|\alpha|},\zeta_G)$, and $\varphi$ moves $\zeta_G$ 
but fixes $\zeta_{0,|\alpha|}$;  here, $\zeta_{0,|\alpha|}$ is a repelling fixed point of degree $2$.  
We have $\varphi(\zeta_G) = \zeta_{0,1/|\alpha|}$, so $w_\varphi(\zeta_G) = \max(0,3-2) = 1$.  
Also, $\tpsi(u) = u^2$, and $\tpsi(0) = 0$, $\tpsi(1) = 1$,
and $\tpsi(\infty) = \infty$, so $\deg_\varphi(\zeta_{0,|\alpha|}) = 2$ and $\zeta_{0,|\alpha|}$ has no shearing directions.  
Thus $w_\varphi(\zeta_{0,|\alpha|}) = 2 -1 + 0= 1$.  

\smallskip
(2) If $A = 0$, $B = 0$, and $C = \alpha^2$, then $\varphi(z) = (z^3 - (1+\alpha)z^2 + (\alpha + \alpha^2)z)/(\alpha^2)$. 
In this case $\varphi$ moves each point of $(\zeta_{0,|\alpha|},\zeta_G)$, 
and $\varphi$ moves both $\zeta_G$ and $\zeta_{0,|\alpha|}$.  
We have $\varphi(\zeta_G) = \zeta_{0,1/|\alpha|^2}$ so $w_\varphi(\zeta_G) = \max(0,3-2) = 1$,
and $\varphi(\zeta_{0,|\alpha|}) = \zeta_{0,1/|\alpha|}$ so $w_\varphi(\zeta_{0,|\alpha|}) = \max(0,3-2) = 1$.

\smallskip
(3) Fix $\tlambda \in \tk$ with $\tlambda \ne 0, 1$, and choose $A \in \cO$ so that $\tA = 1/(\tlambda -1)$. 
Take $B = \alpha$, $C = -\alpha$. Then $\varphi(z) = ((1+A)z^3 - z^2)/(Az^2 + \alpha z - \alpha)$. 
In this case $\varphi$ moves each point of $(\zeta_{0,|\alpha|},\zeta_G)$, and $\varphi$ fixes $\zeta_G$ 
but moves $\zeta_{0,|\alpha|}$; here, $\zeta_G$ is multiplicatively indifferent
for $\varphi$, with rotation number $\tlambda$ for the axis $(1,\infty)$.  We have
$\tphi(z) = ((1+\tA)/\tA) z - (1/\tA) = \tlambda z - (1/\tA)$.  Note that $\tphi(0) = -1/\tA = 1 - \tlambda$, $\tphi(1) = 1$,
and $\tphi(\infty) =\infty$, so $\zeta_G$ has one shearing direction, and $w_\varphi(\zeta_G) = 1 - 1 + 1 = 1$.
Since $\zeta_{0,|\alpha|}$ is moved we have $w_\varphi(\zeta_{0,|\alpha|}) = \max(0,3-2) = 1$.

(4) Fix $\tlambda \in \tk$ with $\tlambda \ne 0, 1$, and choose $A \in \cO$ so that $\tA = 1/(\tlambda -1)$.
Take $B = -A$ and $C = -\alpha$.  Then $\varphi(z) = ((1+A)z^3 - (1+A + \alpha) z^2)/(Az^2 - A z - \alpha)$.
In this case each point of $(\zeta_{0,|\alpha|},\zeta_G)$ is multiplicatively indifferent, 
with rotation number $\tlambda$ for the axis $(0,\infty)$.  
Furthermore, $\zeta_G$ is multiplicatively indifferent, 
while $\zeta_{0,|\alpha|}$ is a repelling fixed point of degree $2$.  
We have $\tphi(z) = (((1+\tA)/\tA) z = \tlambda z$, and $\tphi(0) = 0$, $\tphi(1) = \tlambda$,
and $\tphi(\infty) = \infty$, 
so $\zeta_G$ is multiplicatively indifferent, with rotation number $\tlambda$ for the axis $(0,\infty)$.
There is one shearing direction at $\zeta_G$, so $w_\varphi(\zeta_G) = 1 - 1 + 1 = 1$. 
We have $\tpsi(u) = (1+\tA) u^2/(A u + 1)$ so $\deg_\varphi(\zeta_{0,|\alpha|}) = 2$, 
and $\tpsi(0) = 0$, $\tpsi(1) = 1$, $\psi(\infty) = \infty$. 
There are no shearing directions at $\zeta_{0,|\alpha|}$;  
the multiplier for $\tpsi$ at $\infty$ is $1/\tlambda$,   
and $w_\varphi(\zeta_{0,|\alpha|}) = 2-1+0 = 1$.

\medskip
\noindent{\bf Example G.} (All ways the crucial set of $\varphi$ can consist of one point.)  
 
{\rm In this example, we illustrate the ways the crucial set can consist of one point. 
We give examples of rational functions $\varphi(z) \in K(z)$ of arbitrary degree $d \ge 2$
such that $\cCr(\varphi)$ is a single point $P$ with $w_\varphi(P) = d-1$, and $P$ is  

\begin{itemize} 
\item[(1)] a repelling fixed point of arbitrary degree $2 \le k \le d$,
\item[(2)] an additively indifferent fixed point, 
\item[(3)] a multiplicatively indifferent fixed point, 
\item[(4)] or is moved by $\varphi$.
\end{itemize} 

\noindent{These} the only ways one could have $\cCr(\varphi) = \{P\}$ since an id-indifferent point 
necessarily has weight $w_\varphi(P) = 0$.  
The author thanks Xander Faber for suggesting this example, and for providing the construction for 
repelling fixed points with degree $2 \le k < d$.
}

\smallskip
Let $\varphi(z) \in K(z)$ have degree $d \ge 2$. 

If $\varphi$ has good reduction, then $P = \zeta_G$ is a repelling fixed point of degree $d$, so $w_\varphi(P) = d-1$.  
Similarly, if $\varphi$ has potential good reduction, and it achieves good reduction at $P$, 
then $P$ is a repelling fixed point of degree $d$ and $w_\varphi(P) = d-1$.

For an example where $P$ is a repelling fixed point of degree $2 \le k < d$, 
choose a polynomial $\tf(z) \in \tk[z]$ of degree $d-k$ whose roots are distinct, nonzero,  
and are not ${k-1}^{st}$ roots of unity.   
Let $f_1(z), f_2(z) \in \cO[z]$ be lifts of $\tf(z)$ with no common roots; put 
\begin{equation*} 
\varphi(z) \ = \ 
\frac{z^k f_1(z)} {f_2(z)}  \ .
\end{equation*}
Then $\tphi(z) = z^k$, so $\varphi$ fixes $P = \zeta_G$ and $\deg_\varphi(P) = 2$.
For each root $a$ of $\tf(z)$, we have $s_{\varphi}(P,\vv_a) = 1$.  By Lemma \ref{FirstIdentificationLemma}, 
$\vv_a$ contains a type I fixed point of $\varphi$, and by construction $\tphi(a) \ne a$,
so $\vv_a$ is a shearing direction.  Thus $\varphi$ has at least $d-k$ shearing directions at $P$, 
and $w_\varphi(P) = \deg_\varphi(P) -1 + N_{\Shearing}(P) \ge (k-1) + (d-k) = d-1$.  Since trivially $w_\varphi(P) \le d-1$, 
we must have $w_\varphi(P) = d-1$ and the construction is complete.    
 
Examples where $P$ is additively indifferent or multiplicatively indifferent can be constructed 
in a similar way.  Fix $\lambda \in \cO$ whose reduction $\tlambda \in \tk$ is not $0$ or $1$.
Choose $\tf(z) \in \tk[z]$ of degree $d-1$ whose roots are distinct and nonzero, and different from $\tlambda$.
Let $f_1(z), f_2(z) \in \cO[z]$ be lifts of $\tf(z)$ with no common roots, and put  
\begin{equation*}
\varphi(z) \ = \ \frac{(z+\lambda) \cdot f_1(z)}{f_2(z)}  \qquad 
 ( \text{ resp. \quad $\varphi(z) = \frac{\lambda z \cdot f_1(z)}{f_2(z)}$ } \ ) \ . 
\end{equation*} 
Then $P = \zeta_G$ is additively indifferent (resp. multiplicatively indifferent) for $\varphi$,
and by Lemma \ref{FirstIdentificationLemma}, each direction $\vv_a$ corresponding to a root of $\tf(z) = 0$
contains a type I fixed point of $\varphi$.  Since $\tphi(z)$ moves each of these directions, they are shearing
directions.  As in the previous case, one concludes $w_\varphi(P) = d-1$.

For an example where $P$ is moved by $\varphi$, 
suppose $\alpha_1, \ldots, \alpha_d \in \cO$ belong to distinct classes of $\cO/\fM$, 
and let $0 \ne \pi \in \cO$ have $\ord(\pi) > 0$.      
Write $s_k(x_1, \ldots, x_d)$ for the $k^{th}$ symmetric polynomial in $x_1, \ldots, x_d$, 
put $\valpha = (-\alpha_1, \ldots, -\alpha_d)$, and take
\begin{equation*} 
\varphi(z) \ = \ 
\frac{z^d + s_1(\valpha) z^d + s_2(\valpha) z^{d-1} + \cdots + (s_{d-1}(\valpha) + \pi) z + s_d(\valpha)} {\pi} \ .
\end{equation*} 
The fixed points of $\varphi(z)$ are $\alpha_1, \ldots, \alpha_d, \infty$.  Since these lie in distinct 
tangent directions at $\zeta_G$, and since $\varphi$ moves $P = \zeta_G$ to $\zeta_{0,1/|\pi|}$, we have 
$w_\varphi(P) = \max(0,v(P)-2) = d-1$.

\end{document}